\documentclass[a4paper,pdftex,reqno,10pt]{amsart}

\allowdisplaybreaks[1]

\usepackage{geometry}
\usepackage{natbib}
\let\cite=\citep
\setcitestyle{aysep={}}

\usepackage[norelsize,lined,algoruled,linesnumbered,titlenotnumbered]{algorithm2e} %% for pseudo codes

\usepackage{graphicx}
\graphicspath{{Figures/}}

\usepackage{booktabs}
\usepackage{xcolor}
\usepackage{eurosym}
\usepackage{mathtools}

\theoremstyle{plain}
\newtheorem{Theorem}{Theorem}[section]

\newtheorem{Lemma}[Theorem]{Lemma}

\theoremstyle{definition}

\theoremstyle{remark}

\newcommand{\abs}[1]{\ensuremath{\lvert #1 \rvert}}
\newcommand{\setof}[2]{\ensuremath{\left\{ #1 \;\colon\; #2 \right\}}}
\newcommand{\confconfidence}[4]{\ensuremath{(#1,#2;#3,#4)}}
\newcommand{\confconviction}[2]{\ensuremath{(#1,#2)}}
\newcommand{\varconfidence}[6]{\ensuremath{v^{#1}_{#2, \confconfidence{#3}{#4}{#5}{#6}}}}
\newcommand{\varconviction}[2]{\ensuremath{p_{\confconviction{#1}{#2}}}}
\newcommand{\vardistleft}[2]{\ensuremath{\lambda^{#1}_{#2}}}
\newcommand{\vardistright}[2]{\ensuremath{\rho^{#1}_{#2}}}
\newcommand{\varvoterconf}[3]{\ensuremath{u^{#1}_{#2,#3}}}
\newcommand{\varcontrolconf}[2]{\ensuremath{s^{#1}_{#2}}}
\newcommand{\rev}[1]{#1}
\newcommand{\revJ}[1]{#1}
\newcommand{\revR}[1]{#1}

\begin{document}

%%%%%%%%%%%%%%%%%%%%%%%%%%%%%%%%%%%%%%%%%%%%%%%%%%%%%%%%%%%%%%%%%%%%%%%%%%%%%%
%% Title:
%%%%%%%%%%%%%%%%%%%%%%%%%%%%%%%%%%%%%%%%%%%%%%%%%%%%%%%%%%%%%%%%%%%%%%%%%%%%%%

\title{Optimal Opinion Control: The Campaign Problem}

%%%%%%%%%%%%%%%%%%%%%%%%%%%%%%%%%%%%%%%%%%%%%%%%%%%%%%%%%%%%%%%%%%%%%%%%%%%%%%
%% Author Information:
%%%%%%%%%%%%%%%%%%%%%%%%%%%%%%%%%%%%%%%%%%%%%%%%%%%%%%%%%%%%%%%%%%%%%%%%%%%%%%
\author[R.~Hegselmann]{Rainer Hegselmann}
\author[S.~K{\"o}nig]{Stefan K{\"o}nig}
\author[S.~Kurz]{Sascha Kurz}
\author[C.~Niemann]{Christoph Niemann}
\author[J.~Rambau]{J\"org Rambau}
%
%
%%\author[T.~Eymann]{Torsten Eymann}

\address{Rainer Hegselmann\\Fakult\"at f\"ur Kulturwissenschaften\\Universit\"at Bayreuth\\Germany}
\email{rainer.hegselmann@uni-bayreuth.de}
\address{Sascha Kurz\\Fakult\"at f\"ur Mathematik, Physik und Informatik\\Universit\"at Bayreuth\\Germany}
\email{sascha.kurz@uni-bayreuth.de}
\address{J\"org Rambau\\Fakult\"at f\"ur Mathematik, Physik und Informatik\\Universit\"at Bayreuth\\Germany}
\email{joerg.rambau@uni-bayreuth.de}
%
%\address{Stefan K{\"o}nig\\Fakult\"at f\"ur Rechts- und Wirtschaftswissenschaften\\Universit\"at Bayreuth, Germany}
%\email{stefan.koenig@uni-bayreuth.de}
\address{Stefan K{\"o}nig\\Fiducia IT AG, Aschheim, Germany}
\email{stefan.koenig@fiducia.de}

%\address{Christoph Niemann\\Fakult\"at f\"ur Rechts- und Wirtschaftswissenschaften\\Universit\"at Bayreuth, Germany}
%\email{christoph.niemann@uni-bayreuth.de}
\address{Christoph Niemann\\MaibornWolff GmbH, M\"unchen, Germany}
\email{christoph.niemann@maibornwolff.de}
%
%%\address{Torsten Eymann\\Fakult\"at f\"ur Rechts- und Wirtschaftswissenschaften\\Universit\"at Bayreuth, Germany}
%%\email{torsten.eymann@uni-bayreuth.de}

%%%%%%%%%%%%%%%%%%%%%%%%%%%%%%%%%%%%%%%%%%%%%%%%%%%%%%%%%%%%%%%%%%%%%%%%%%%%%%
%% Abstract:
%%%%%%%%%%%%%%%%%%%%%%%%%%%%%%%%%%%%%%%%%%%%%%%%%%%%%%%%%%%%%%%%%%%%%%%%%%%%%%
\begin{abstract}
  \revR{Opinion dynamics is nowadays a very common field of
    research. In this article we formulate and then study a novel,
    namely strategic perspective on such dynamics: There are the usual
    `normal' agents that update their opinions, for instance according
    the well-known bounded confidence mechanism. But, additionally,
    there is at least one strategic agent. That agent uses opinions as
    freely selectable strategies to get control on the dynamics: The
    strategic agent of our benchmark problem tries, during a campaign
    of a certain length, to influence the ongoing dynamics among
    normal agents with strategically placed opinions (one per period)
    in such a way, that, by the end of the campaign, as much as
    possible normals end up with opinions in a certain interval of the
    opinion space. Structurally, such a problem is an optimal
    control problem. That type of problem is ubiquitous. Resorting to
    advanced \revJ{and partly non-standard} methods for computing
    optimal controls, we solve \revJ{some instances} of the campaign
    problem. But even for a very small number of normal agents, just
    one strategic agent, and a ten-period campaign length, the problem
    turns out to be extremely difficult. Explicitly we discuss
    moral and political concerns that immediately arise, if someone
    starts to analyze the possibilities of an ``optimal opinion
    control''.}
%
%  We propose discrete-time models for the optimal control of
%  continuous opinions of a finite number of individuals.  \revJ{The
%    only allowed control is diplomacy, grounded on the human right of
%    free speech:} Given a model of the dynamics of opinions, one
%  individual can freely choose its advertised opinion in each stage,
%  thereby gaining a controlled influence on the dynamics.  Special
%  attention is devoted to the bounded-confidence model, for which we
%  point out some numerical instabilities, which is important for
%  assessing the outcomes of simulations. We show that even in a small
%  \rev{example
%  %% system  with eleven individuals and ten stages 
%  it is} highly non-trivial to learn
%  something rigorous about optimal controls.  By means of
%  mixed-integer linear programming we can characterize globally
%  optimal controls.  In the example case, we can use this to prove the
%  optimality of some controls generated by a tailor-made
%  heuristics. Investigations about the control population utilizing a
%  genetic algorithm show strong evidence that optimal controls can be
%  extremely hard to find among all controls.  Controls generated by
%  the popular generic feedback-principle of model predictive control
%  can be significantly far from optimum, which underlines that more
%  research is necessary to understand the structural properties of
%  optimal controls.
\end{abstract}

\maketitle

%%%%%%%%%%%%%%%%%%%%%%%%%%%%%%%%%%%%%%%%%%%%%%%%%%%%%%%%%%%%%%%%%%%%%%%%%%%%%%
%% Section: Introduction (modified by Rainer Hegselmann)
%%%%%%%%%%%%%%%%%%%%%%%%%%%%%%%%%%%%%%%%%%%%%%%%%%%%%%%%%%%%%%%%%%%%%%%%%%%%%%
\section{Introduction}
\label{sec:introduction}

Since about 60 years the dynamics of opinions has been studied. Today
it is a standard topic of general conferences on agent-based
modelling. A bunch of models were defined and analyzed.\footnote{For
  \rev{the} \revR{general history of opinion dynamics} 
  see the introduction and for a partial
  classification see ch.~2 and~3 of the \revR{paper} by
  \citet{Hegselmann:OD-BC:2002}.}  In the last 15 years at least
hundreds and probably more than thousands of simulation studies on the
dynamics of opinions were published.\footnote{There are several
  surveys encompassing various models of opinion dynamics
  \cite{Acemoglu+Ozdaglar:OD-Learning:2011,Castellano+Fortunato+Loreto:SocialDynamics-StatisticalPhysics:2009,Liebrand+Hegselmann:SocialProcessModeling:1998,Stauffer:OD:2005,Xia:2011:OD-ReviewPerspective}. \revR{A microfoundation for the 
    evolution of several opinion dynamics mechanisms is proposed and discussed by
    \citet{Groeber+Lorenz+Schweitzer:DissonanceMinimization:2014}. -- 
    In subsection \ref{subsec:underlying_OD} 
we give specific hints to the related opinion dynamics literature 
that is directly relevant in the context of our present article.}}
The studies and their underlying models differ in many details: The
opinions and the underlying time are continuous or discrete, the
updating of opinions is governed by different updating regimes, the
space of possible opinions may have more than one dimension, the
dynamics may run on this or that type of network, various types of
noise may be involved. But despite of all their differences there is a
commonality in all these studies and their models: The agents
influence mutually their opinions, \revR{\emph{but they do not do that
strategically}}.

%\revJ{\footnote{A microfoundation for the dynamic
%    evolution is discussed by
%    \citet{Groeber+Lorenz+Schweitzer:DissonanceMinimization:2014}.}}

What is studied in the huge body of articles is typically focusing on
convergence, dynamical patterns or final structures. Given the
specific parameters of the underlying model, the typical questions
are: Under what conditions does the dynamics stabilize? Does the
dynamics lead to consensus, polarisation, or other interesting types
of clustering? What are the time scales that are involved?  What
remains unasked in these studies are strategic questions like: Which
opinion should an agent pretend to have to drive the whole dynamics in
his or her preferred direction? Where in the opinion space should an
agent `place' an opinion, given that he or she has a certain
preference with regard to the opinions in later periods.  Our article
deals with such strategic questions. We develop a conceptual framework
that allows to answer strategic questions in certain
cases. Additionally, we analyze why it is surprisingly difficult or
impossible to give exact answers to strategic questions even in cases
that look very, very simple.

It is not by accident that strategic questions are normally not raised
in the sort of research that is labeled opinion dynamics. The standard
approach is to describe the dynamics as a more or less complicated
dynamical system: There is a set $I$ of agents
$1, 2, \ldots, i, j, \ldots, n$ and a discrete time
$t = 1, 2, \ldots$. The opinions of the agents are given by an
\revJ{\emph{opinion profile}.  This is a} vector
$\mathbf{x}^t = (x^t_1, \ldots, x^t_n)$ that describes the state of
the system at time~$t$. Even if stated in an informal or semi-formal
way (sufficiently clear to program the process), the dynamics of the
system is basically given by a function $\mathbf{f}^t$ that computes
the state of the system $\mathbf{x}^{t+1}$ as
$\mathbf{x}^{t+1} = \mathbf{f}^t(\mathbf{x}^t)$.

Thus, for each agent $i$ the function $\mathbf{f}^t$ specifies how
$x_i^{t+1}$ depends upon $\mathbf{x}^t$. Depending upon the specific
opinion dynamics model, the vector valued functions $\mathbf{f}^t$
work in very different ways. For the most part they do some kind of
averaging: averaging with a privileged weight for agent $i$'s own
opinion or a weighted averaging with weights $w_{ij}$ that agent $i$
assigns to agents~$j$ and that are meant to express agent~$i$'s
respect for agent~$j$, or some other sort of averaging subject to
constraints, for instance constraints in terms of network distance on
an underlying network on which the opinions dynamics is assumed to
run.

Whatever the `story' about $\mathbf{f}^t$, it is always a reactive
process in which the agents \revR{react on the \emph{last} period $t$. 
In principle the step to $t+1$ might depend upon some more past periods. 
But even then, an answer to the question, where to place an opinion in
order to drive the dynamics in a preferred direction, requires
something very different from looking into the past: It requires} 
\emph{anticipation}, i.e., finding out what the future
effects of placing an opinion here or there in the opinion space
probably are, and then placing it there, where the placement is most
effective to get to a preferred outcome.  

In the following, we assume a setting in which we have two sets of
agents: First, a set of \emph{non-strategic} agents as they are
usually assumed in opinion dynamics. They are driven by the function
$\mathbf{f}^t$ . The function describes a dynamical system in which
the non-strategic agents always reveal their \revR{`true'} actual 
opinion, take the opinions of others as their true actual opinion, and mutually
react on the given opinion profile $\mathbf{x}^t$ according to
$\mathbf{f}^t(\mathbf{x}^t)$. The second set of agents is a set of
\emph{strategic} agents. Whatever their true opinion may actually be,
they can place any opinion strategically in the opinion space where,
then, non-strategic agents take these \rev{opinions} \revR{at their 
face values} and consider them as revealed true opinions of other agents. 
The strategic agents have preferences
%\footnote{We will rigorously
%define a preference relation later.} 
over possible opinion profiles
of non-strategic agents. Therefore, they try to place opinions \rev{in
such} a way that the opinion stream that is generated by
$\mathbf{f}^t$ is driven towards the preferred type of profile. 

%Thus,
%the strategic agents somehow `play a game' with the non-strategic
%agents. Non-strategic agents do not realize that and take the
%strategically placed opinions of strategic agents as if they were
%revealed true opinion of `\revR{normal', i.e. \emph{non}-strategic} agents.

Our setting has a structure as it is often perceived or even
explicitly `conceptualized' by political or commercial campaigners:
There is an ongoing opinion stream, result of and driven by mutual
exchange of opinions between communicating citizens, consumers,
members of some parliamentary body, etc. That opinion stream has its
own dynamics. However, it is possible to intervene: Using channels of
all sorts (TV, radio, print media, mails, posters, adds, calls,
speeches, personal conversation, etc.)\ one can place opinions in the
opinion space. Done in an intelligent way, these interventions should
drive the opinion stream in the direction of an outcome that is
preferred by the principal who pays a campaign. About that will
often be the self-understanding and selling-point of a campaigning
agency.

The number of strategic agents matters: If there are \emph{two or
  more} strategic agents, the setting becomes a \emph{game theoretical
  context} in which the strategic agents have to take into account
that there are others that try to \revR{influence} the opinion
dynamics \emph{as well}. Therefore, the strategic agents do not only
`play a game' with the non-strategic agents. They play also -- and
that now in an exact game theoretical sense of the word -- a
\emph{game} against each other. It is a complicated game for which in
principle usual solution concepts like the (sub-game perfect) Nash
equilibrium can be applied.  But if there is \emph{just one} strategic
agent, then there are no other competing \revR{players. That turns the
  problem of the strategic agent into the following question: How can
  one place opinions in an ongoing opinion stream (governed
  by~$\mathbf{f}^t$) in such a way that the stream moves as much as
  possible in the direction of one's favorite profile?  \revJ{This
    task means to optimize of decisions over time.  Technically
    speaking, problems of this type are \emph{optimal control
      problems}}. As soon as a mathematical formalization is
  available, there are various mathematical methods to find
  solutions.}

\revR{The topic of this paper is the optimal control problem of a
  strategic agent who tries to influence a certain ongoing opinion
  dynamics. As a benchmark problem for such an agent, we define what
  we call the \emph{campaign problem}: The strategic agent tries to
  control an opinion dynamics, such that in a certain period, known in
  advance, there are as much as possible opinions of normal agents in
  a certain interval of the opinion space.  That covers many types of
  voting or buying campaigns.  We will investigate the campaign
  problem by various methods. Our focus will always be on
  understanding basic features of the problem---and that will be
  difficult enough. Some elementary mathematical proofs will give some
  theoretical insights into structural properties. With models from
  mixed-integer linear programming (MILP) we will try to directly
  solve our central optimization problem.  \revJ{This way we can solve
    some instances of the campaign problem, but -- surprisingly -- not
    all}. Therefore we will additionally attack the control problem by
  heuristic methods\revJ{.  An additional investigation using
    genetic algorithms provides evidence for the fact that `good'
    controls are almost impossible to find by randomized
    exploration.}}

\subsection{Moral and political concerns}
\label{subsec:concerns}
\revR{
Our question and approach immediately raises moral and political
concerns: Over the last years we learnt that there are well equipped agencies 
that (among other things) aim at a more or less complete supervision of 
private and public opinions. Isn't the title
of our paper at least a partial confession, that now, 
as a kind of additional thread, basic research on 
strategies for an efficient  
manipulation of public opinion formation is put on the agenda, isn't 
our `control terminology' a tell-tale language?}

\revR{This is a very serious question and a very serious concern. An
  answer requires careful consideration of least \revJ{six} aspects:
  \revJ{\emph{First}, the only means of opinion `control' is
    publically stating an opinion.}  \emph{Second}, campaigning can be
  done for both, good and bad purposes: Information, enlightenment,
  spreading the truth, \emph{or} desinformation, confusion, spreading
  liars.  \emph{Third}, whoever plans a campaign for a good purpose
  will immediately get into serious optimization problems. For an
  example, let's assume that measles vaccination is basically a good
  thing.  From a public health point of view, a certain minimum
  vaccination rate is necessary and a certain upper rate sufficient.
  In the US, Germany, and other countries the actual vaccination rates
  are too low.  Therefore public health institutions design pro
  vaccination campaigns.  In doing so, important questions are: On
  which parts of the networks of vaccination skeptics and enemies
  should one concentrate to what degree, or in what sequence? What
  might convince which parts of the networks to change their opinions?
  Which opinion changes can be induced in the network? -- All these
  questions are questions about an optimal campaign, given one's
  constraints (in terms of budget, time, channels, arguments, their
  effects on whom, and chances to confront people with them etc.).
  \emph{Fourth}, our control terminology is the usual terminology in
  the disciplines, theories, and tool box approaches that can be used
  to solve optimization problems of all sorts.  Our campaign problem
  is just an instance.  The spirit of these approaches is technical:
  One has some constrained means to influence something to some degree
  in a certain, more or less attractive direction. What, then, is an
  optimal use of one's means?---that is the guiding question.  The
  question is neutral as to ends, except for optimality itself (and
  some more formal consistency requirements).  The possibility of a
  `dual use' is inherent to such approaches.  \emph{Fifth}, in what
  follows our strategic agents place their opinions independent of any
  truth considerations. They try to influence the ongoing opinion
  dynamics in their favorite direction by placing an opinion here or
  there in the opinion space.  Only the effects matter. Therefore, one
  might say, our strategic agents are completely opportunistic.  Under
  a less pejorative description we might consider them as perfect
  diplomats that know how to overcome entrenched opinions.  If we
  consider diplomacy or opportunism as morally unacceptable, then we
  could formulate a corresponding \emph{moral} constraint: We could
  require to place only opinions that are close to one's true actual
  opinion, and we specify `close' by a certain threshold for the
  maximal acceptable distance of placed opinions to one's own actual
  opinion.  Probably nobody would advocate a threshold of
  zero.\footnote{ \revJ{Imagine, I use a hypothetical argument that I
      myself do not share.  Nevertheless, in other persons' opinions
      this may induce a change in the direction of what I consider the
      truth.  What is morally wrong with such a discussion strategy?}}
  But, then, one should also recognize, that for any such
  \emph{non}-zero-threshold, there still exists the control problem of
  how to design an optimal campaign (though with an additional
  constraint).  \emph{Finally}, dual use concerns about an optimal
  control approach to opinion dynamics are justified, of course. The
  existence of professional disinformation agencies is not an
  invention of conspiracy theorists.  They are at work, they have
  their \emph{secret} expertise -- and, in all likelihood, they
  learn. But exactly because of that there has to be \emph{public}
  knowledge about what can be done in terms of optimal campaigning --
  and that for both, good \emph{and} bad purposes.  The first type of
  knowledge is supportive, the second protective.  Helpful are both.}

%\section{A Sketch of Optimal Control in Opinion Dynamics}
%\label{sec:Sketch}

\subsection{Our benchmark: The campaign problem}
\label{subsec:benchmark}
To specify the campaign problem  we add
\emph{just one} strategic agent$_0$ to the set of agents. Agent$_0$ is
equipped with the ability to freely choose in any time step what other
agents then perceive as his or her opinion. We call agent$_0$ the
\emph{controller} and his freely chosen opinion the \emph{control}.
Mathematically, this makes the development of the opinion system
dependent on exogenously chosen parameters, namely the control
opinion, and we are faced with a \emph{control system}. If we define
what the controller wants to achieve, we can formulate an
\emph{optimal control problem}, for which the controller tries to find
the controls to get there.  

Our optimal control problem is of a \emph{specific -- seemingly simple
  -- type}: Agent$_0$ can strategically place opinions in a finite
sequence of periods, one and only one opinion per period. There is the
ongoing and underlying opinion dynamics, given by a function
$\mathbf{f}^t$. Period by period agent$_0$ tries to do the placement
of an opinion in such a way that finally in a certain future time
step~$N$ (the \emph{horizon}), known in advance, something is
maximized: the number of normal agents' opinions that are in a certain
part of the opinion space that ex ante was specified by agent$_0$. To
keep it simple, we assume as a special one-dimensional opinion space
the real-valued unit interval $[0,1]$.  As the part of the opinion
space preferred by the controller, we assume a \emph{target interval}
$[\ell, r] \subseteq [0, 1]$ for some prescribed $\ell < r$ in $[0,
1]$ known to the controller.

Both assumptions are much less restrictive than they seem to be:
First, the unit interval can be used to represent opinions about, for
example, tax rates, minimal wages, maximum salaries, political
positions on a left-right spectrum, product quality, or any property
whatsoever that can be expressed by real-valued numbers. If -- and
often that will be the case -- the `really' possible range of
numerical values is different from the unit interval, then some
transformation, approximation, or range regarding guess work is
necessary. But that is an easy and widely accepted step (at least in
principle). Second, suppose there are $m$ fixed alternatives
$\mathbf{a} = (a_1, a_2, \ldots, a_m) \in [0, 1]^m$, sorted such that
$a_1 \le a_2 \le \ldots \le a_m$. Further suppose, our $n$ normal agents
have to choose among the alternatives at the future time step~$N$ and
and will do that by choosing an alternative that is next to their own
opinion in that time step. What, then, is the problem of a controller
with the interest to make as much normal agents as possible choosing a
certain alternative~$a_j$? Obviously the problem is to maximize the
number of agents' opinions that in time step~$N$ are within a certain
interval to the left and to the right of the favored
alternative~$a_j$. The exact bounds of that interval depend upon the
exact positions of the two nearest alternatives~$a_{j-1}$ to the left
and~$a_{j+1}$ to the right of the favored~$a_j$. The exact left and
rights bounds are then $\tfrac{a_{j-1} + a_j}{2}$ and $\tfrac{a_j +
  a_{j+1}}{2}$ respectively.

Therefore, whatever the vector of alternatives may be (e.g., positions
of parties or candidates on a left/right scale, a price or a technical
specification of a certain type of product), whenever there are voting
or buying decisions\footnote{It may even be a buying decision in the
  metaphorical sense of whether or not to `buy' a certain
  assumption: Imagine a committee that after some discussion has to
  decide whether to proceed based on this or that assumption in a
  vector~$\mathbf{a}$ of alternatives.} after a foregoing opinion
dynamics (e.g., about the appropriate position in the political
spectrum, the acceptable price or a desirable technical specification
of some product), our controller agent$_0$ who tries to `sell' a
certain alternative $a_j$ as effectively as possible, has always the
same problem: How to get by period $N$ as many opinions as possible
within a certain target interval $[\ell, r]$, determined by the
closest left and right competitors of~$a_j$? -- Obviously, our
framework and setup is much more general than it looks at a first
glance.

\subsection{The underlying opinion dynamics: A linear and a non-linear
version}
\label{subsec:underlying_OD}
Our problem and approach presupposes an underlying opinion dynamics
given by a function~$\mathbf{f}^t$. But there are many. We will use a
linear and a non-linear variant, each of them probably being the most
prominent variant of their type.  

%
%\section{Two Opinion Dynamics in the Literature}
%\label{sec:RelatedWork}

In the \emph{linear variant}, \revR{the} \revJ{opinion} dynamics is driven by
\emph{weighted averaging}: Weights $w_{ij}$ may express the respect or
competence that an agent~$i$ assigns to an agent~$j$; alternatively, a
weight $w_{ij}$ may express the influence, impact or power that
agent~$j$ has on agent~$i$. The weights that an agent assigns sum up
to one. The opinion space is the unit interval $[0,1]$. -- The history
of this linear system goes back to 
\citet{French:Formal-Theory-Social_Power:1956}, %French [1956],
it has been already presented by
\citet{Harary:Unanimity-French:1959}, %Harary [1959],
it was explicitly stated by
\citet{DeGroot:Reaching-Consensus:1974}, %DeGroot [1974],
and it received a lot of attention, especially in philosophy, through
the book by \citet{Lehrer+Wagner:Rational-Consensus:1981}%
%Lehrer/Wagner [1981]
.\footnote{\citet{Lehrer+Wagner:Rational-Consensus:1981} %Lehrer/Wagner [1981] 
do not interpret the iterated
  weighted averaging as a process in time. As stressed by
  \citet[p.~4]{Hegselmann:Truth-LaborDivision:2006}:  %Hegselmann/Krause [2006, p.~4]: 
  ``Their starting point is a
  `dialectical equilibrium', i.e., a situation after, the group has
  engaged in extended discussion of the issue so that all empirical
  data and theoretical ratiocination has been communicated. `\ldots{}
  the discussion has sufficiently exhausted the scientific information
  available so that further discussion would not change the opinion of
  any member of the group' (\cite[p.~19]{Lehrer+Wagner:Rational-Consensus:1981}). %Lehrer/Wagner 1981, 19). 
  The central
  question for Lehrer and Wagner then is: Once the dialectical
  equilibrium is reached, is there a rational procedure to aggregate
  the normally still divergent opinions in the group (cf.{} \cite[p.~229]{Lehrer:Replies:1981}?  
  %Lehrer 1981b, 229)? 
  Their answer is `Yes.' The basic idea for the procedure
  is to make use of the fact that normally we all do not only have
  opinions but also information on expertise or reliability of
  others. That information can be used to assign weights to other
  individuals. The whole aggregation procedure is then iterated
  weighted averaging with $t \to \infty$ and based on constant
  weights. It is shown that for lots of weight matrices the
  individuals reach a consensus whatever the initial opinions might be
  -- if they only were willing to apply the proposed aggregation
  procedure.''} We will refer to that model as the
\emph{DeGroot-model} (DG-model).\footnote{In philosophy the model is
  often called the \emph{Lehrer-Wagner model}.}

The \emph{non-linear variant} that we will use is the so called
\emph{bounded confidence model} (BC-model). In this model the agents
take seriously those others whose opinion are not too far away from
their own opinion: The agents have a certain confidence
radius~$\epsilon$ and update their opinions -- the opinion space is
again the unit interval $[0,1]$ -- by averaging over all opinions that
are from their own opinion not further away than~$\epsilon$: An
agent~$i$ updates to $x^{t+1}_i$ by averaging over the elements of the
set $\bigl\{ j \in \{1, 2, \ldots n \} \, \big\vert \, \abs{x^t_i -
  x^t_j} \le \epsilon \bigr\}$, i.e., over all the opinions that are
within what is called his or her confidence interval. The model was
defined by \citet{Krause:SozDyn:1997}, in 1998 it was coined \emph{bounded
confidence model} \citep{krause2000discrete}, and then for the
first time, to a certain extent, comprehensively analyzed, by both
simulations and rigorous analytical means, by
\citet{Hegselmann:OD-BC:2002}.

The model looks extremely simple. However, there are several warnings
in the literature on the BC-model, among them a very recent one: ``The
update rule \ldots{} is certainly simple to formulate, though the
simplicity is deceptive'' \cite[p.~2]{Wedin+Hegarty:BC-OD-ContinuousNoConsensus:2014}. %[Wedin/Hegarty 2014: p. 4]. 
The authors'
warning is well founded: The simple update rule generates a
complicated dynamics that still is only partially understood. The main
reason for that is this: The updating rule of the BC-model can be
described as assigning weights to other agents. All agents with
opinions out of the confidence interval get a weight of~$0$; agents
within get a weight of~$1$ divided by the number of agents that are
within the interval. Therefore, the BC-dynamics is weighted averaging
as well. However, there is a crucial difference to the linear
DG-model: The weights of the BC-model are \emph{time-dependent} and,
even worse, \emph{discontinuously dependent on the current
  profile}. That causes a lot of trouble -- and, at the same time,
generates many of interesting effects. As a consequence, the BC-model
became a subject for all sorts of analytical or computational analysis
and a starting point for extensions of all sorts.  The body of
literature on the BC-model is correspondingly huge.

%% The BC-model was invented by Krause (without that name) in
%% \cite{Krause:SozDyn:1997} and later thoroughly analyzed by Hegselmann
%% and Krause in \cite{Hegselmann:OD-BC:2002} and immediately
%% investigated further in~\cite{Dittmer:OD-BC:2001}.  
Structural results in the BC-model were obtained with respect to
convergence and its rate
\cite{Dittmer:OD-BC:2001,Krause:ArithmeticGeometricDiscreteSystems:2006,Lorenz:OD-BC-Continuous:2006},
thresholds for the confidence radius
\cite{Fortunato:OD-BC-Threshold:2005}, the identification of the
really crucial topological and metric structures
\cite{Krause:ArithmeticGeometricDiscreteSystems:2006}, or the
influence of the underlying network
\cite{Weisbuch:OD-BC-Networks:2004}.  The influence of a `true'
opinion, to which (some of) the individuals are attracted, received
special attention
\cite{Hegselmann:Truth-LaborDivision:2006,Malarz:OD-TruthSeekers:2006,Douven+Riegler:ExtendingHK:2010,Douven+Kelp:TruthApprox:2011,Kurz+Rambau:HegselmannKrauseConjecture:2011,Wenmackers+Vanpoucke+Douven:Inconsistencies:2012,Wenmackers+Vanpoucke+Douven:Rationality:2014}.
With a grain of salt, the true opinion can also be interpreted as a
control that is constant over time and that is contained in each
individual's confidence interval.

Many variants of the original BC-model (discrete time, finitely many
individuals, continuous opinion space) have been proposed, among them
pairwise sequential updating
\cite{Defuant+Neau+Amblard+Weisbuch:MixingBeliefs:2000}, a discrete
opinion space \cite{Fortunato:OD-BC-Discrete:2004}, a
multi-dimensional opinion space
\cite{Fortunato+Latora+Luchino+Rapisarda:OD-BC-Dim2:2005,Krause:OD-TimeVariant-HighDim:2005},
a noisy opinion space
\cite{bhattacharyya2013convergence,pineda2013noisy}, a continuous
distribution instead of a finite-dimensional vector of opinions
\cite{Lorenz:OD-Stabilization:2005,Lorenz:OD-BC-Survey:2007},
\rev{continuous} time
\rev{\cite{Blondel+Hendricks+Tsitsiklis:OD:2010}, and a continuum of
  agents \cite{Wedin+Hegarty:BC-OD-ContinuousNoConsensus:2014}}.
Alternative dynamics have been enhanced with BC-type features by
several authors
\cite{Stauffer:Sznajd-Model-LimitedPersuasion:2002,Stauffer:Sznajd-Model-Simulation:2003,Fortunato:SznajdConsensus:2005,Stauffer+Sahimi:OD-Simulation:2006,Rodrigues+DaFCosta:SznajdNetworks:2005}
in order to make the resulting emergent effects more interesting.

It is interesting to note that simulations play an important role in
the investigation of the innocent-looking models for opinion dynamics
\cite{Stauffer:Sznajd-Model-Simulation:2003,Hegselmann:Truth-LaborDivision:2006,Fortunato+Stauffer:OD-Simulations:2006,Stauffer+Sahimi:OD-Simulation:2006}
-- a hint that some aspects of opinion dynamics are very hard to
deduce purely theoretically. The arguably
most general account of theoretical aspects was contributed by
Chazelle on the basis of a completely new methodology around
\rev{function theoretical arguments} %(!)
\cite{Chazelle:TotalsEnergy:2011}.  Moreover, opinion dynamics in
multi-agent systems can be seen as an instance of \emph{influence
  system} -- this broader context is also described
by~\cite{Chazelle:InfluenceSystemsFOCS:2012}.

\revJ{\citet{Carletti+Fanelli+Grolli+Guarino:EfficientPropaganda:2006}
  study the influence of a given, fixed exogenous propaganda opinion
  on the Deffuant-Weisbuch opinion dynamics.  However, the problem of
  optimally choosing a propaganda opinion is not considered.}
\revJ{Implicitly, the control of opinion dynamics started out from
  gaining control over the communication structure
  \citep{Lorenz+Urbig:EnforceConsensus:2007} or additional system
  parameters \citep{Lorens:FosteringConsensus:2008}.}  Recently, the
effect of time-varying exogenous influences on the consensus process
has been studied by \citet{mirtabatabaei2014eulerian}.
\revJ{\citet{Kurz:BCControlledConvergence:2015} investigated how the
  time to stabilization changes if controlled opinions enter the
  scene.}  For the first time to the best of our knowledge, the notion
of opinion control appeared literally only in the title of a paper by
\citet{Zuev+Feyanin:OC:2012}, based on a different dynamical model,
though. For continuous time optimal \rev{control techniques} have been
applied to opinion consensus, exemplarily also for the BC-dynamics, in
the preprint by \citet{albi2014kinetic}.  Closest to our research is,
up to now, probably the investigation by
\citet{Fortunato:OD-DamageSpreading:2005} about so-called damage
spreading: what happens if some of the opinions in a BC-model change
erratically and possibly drastically?  The setting in that paper,
however, has not been utilized to find in some sense `optimal
damages', \revJ{i.e., some that lead to an outcome that is most
  desirable among all possible outcomes.}

\subsection{Specification of the campaign problem for the DG- and 
the BC-model}
\label{subsec:specification}

In what follows, we analyze our optimal control problem with the
linear DG- and the non-linear BC-model as the underlying opinion
dynamics that agent$_0$ tries to control: Given the DG- or BC-dynamics
and using the controls, i.e., placing one opinion per period here or
there in the opinion space, agent$_0$ tries to maximize the number of
agents with opinions that in a future period~$N$ are in a certain
target interval $[\ell, r]$ -- and therefore would `buy' the
corresponding alternative $a_j$ that agent$_0$ is actually campaigning
for. -- 
In the following, we will switch to a more vivid language fitting this
interpretation: agents are called \emph{voters}, the special agent$_0$
is called the \emph{controller}, the target interval $[\ell, r]$
is called the \emph{conviction interval}, and voters in the target
interval are called \emph{convinced voters}.

%\section{The Benchmark Campaign Problem}
%\label{sec:Benchmark}

As an example, we will analyze a specific instance of the campaign
problem: There are $11$ voters,
governed by a DG- or BC-version of the (reactive) function
$\mathbf{f}^t$. At $t = 0$, the $11$ voters start with the opinions
$0$, $0.1$, $0.2$, \ldots, $0.9$, $1$.  The \rev{confidence radius is
  given by $\epsilon=0.15$ and the} conviction interval is
$[0.375, 0.625]$. That conviction interval would be the target in a
campaign in which the alternatives
$\frac{1}{4}, \frac{1}{2}, \frac{3}{4}$ are competing and
$\frac{1}{2}$ is the preferred alternative.  The goal is to maximize
the number of convinced voters, i.e., those with opinions in the
conviction interval, in a certain future period $N$ with
$N = 1, 2, \ldots, 10$.  The benchmark problem looks like a baby
problem. But it is a monster: It will turn out that for higher one
digit values of $N$ we could not solve it by the most sophisticated
methods available.\revJ{\footnote{These values were chosen in a research
  seminar because they looked very simple and came to our mind
  first. We were convinced that we would solve all related problems
  very fast by whatever method and could start to worry about more
  serious problem sizes.  We were wrong.  It is still this
  innocent-looking benchmark campaign problem that determines the
  agenda, as we will see.}}

More specifically, we will see, that even for this innocent-looking
example we were not able to find the optimal number of convinced
voters for all numbers of stages between $1$ and~$10$.  In
Table~\ref{tab:allresults} we summarize our results on the benchmark
campaign problem in this paper.  In the rest of the paper we will
explain in detail how we obtained that knowledge.
\begin{table}[h]
  \centering\footnotesize
  \begin{tabular}{r@{\hspace*{1cm}}rr}
    \toprule
    \# stages  & \multicolumn{2}{c}{\# convincable voters}  \\
    & lower bound$^*$ & upper bound$^{**}$ \\
    \midrule 
    0 &  3 & 3\\
    1 &  3 & 3\\
    2 &  4 & 4\\
    3 &  5 & 5\\
    4 &  5 & 5\\
    5 &  6 & 6\\
    6 &  6 & 6\\
    7 &  8 & 11\\
    8 &  8 & 11\\
    9 &  8 & 11\\
    10 & 11& 11\\
    \bottomrule
  \end{tabular}
  \caption{The optimal number of convinced voters in the benchmark
    example lies between the given lower and upper bounds, depending
    on the number of stages; we will explain in 
    in Sections~\ref{sec:Heuristics}$^*$
    and~\ref{sec:computational-info}$^{**}$ how we obtained this information.}
  \label{tab:allresults}
\end{table}

\subsection{Stucture of this Paper}
\label{subsec:structure}

In what follows, we
\begin{itemize}
\item \revR{define exactly} the new problem \emph{optimal opinion
    control}\footnote{During finalizing work on this paper, it came to
    our attention that optimal opinion control was independently
    introduced and investigated with a different objective in
    continuous time by
    \citet{Wongkaew+Caponigro+Borzi:LeadershipFlocking:2014}} and an
  example instance of it, the benchmark campaign problem, for
  illustration and test purposes,
\item develop two exact mathematical models based on mixed integer
  linear programming to characterize optimal controls\footnote{Another
    application of mixed integer linear programming techniques in
    modeling social sciences was, e.g., presented by
    \citet{Kurz:InversePowerIndexProblem:2012}.},
\item devise three classes of primal heuristics (combinatorial
  tailor-made, meta-heuristics, model predictive control) to find good
  controls,
\item present computational results on all methods applied to our
  benchmark problem,
\item sketch possible lines of future research.
\end{itemize}

%% \subsection{Outline of the paper}
%% \label{sec:introduction:outline}

In Section~\ref{sec:Opinion-Dynamics} we formally introduce the
optimal opinion control problem.  Section~\ref{sec:Numerics} shows
what can happen if the pitfalls of numerical mathematics are ignored
in computer simulations.  Section~\ref{sec:Structures} presents some
structural knowledge about optimal controls in small special cases,
thereby indicating that a general theoretical solution of the optimal
control problem is not very likely. Computational results are
presented in Section~\ref{sec:computational-info}. The detailed
presentation of the underlying exact mathematical models is postponed
to Appendix~\ref{sec:Mathematical-Model}, followed by information on
the parameter settings for the commercial solver we used in
Appendix~\ref{sec:param-sett-milp}.  Our heuristics are introduced in
Section~\ref{sec:Heuristics}. In Section~\ref{sec:interpr-results} we
\rev{interpret} the obtained results. Conclusions and further directions
can be found in Section~\ref{sec:conclusion-and-outlook}.

%%%%%%%%%%%%%%%%%%%%%%%%%%%%%%%%%%%%%%%%%%%%%%%%%%%%%%%%%%%%%%%%%%%%%%%%%%%%%%
%% Section: Modeling the Dynamics of Opinions and Their Control
%%%%%%%%%%%%%%%%%%%%%%%%%%%%%%%%%%%%%%%%%%%%%%%%%%%%%%%%%%%%%%%%%%%%%%%%%%%%%%
\section{Modeling the Dynamics of Opinions and Their Control}
\label{sec:Opinion-Dynamics}

We will now formalize the ideas \revJ{sketched so far}.  
%% There may be
%% many ways to transform opinions about every conceivable topic into
%% more or less complicated mathematical structures.  
In this paper, we
will restrict ourselves to the arguably simplest case where an opinion
of a voter $i \in I$ can be represented by a real number $x_i$ in the
unit interval $[0,1]$.
%% This can be an
%% opinion about a real quantity measured in percent of a reasonably
%% restricted range (e.g., where is the economic growth of a country next
%% year between $-2$\,\% and~$6$\,\%), or it may be an opinion about the
%% range between extreme positions (e.g., where should be the balance
%% between free markets and public control).

The opinions may change over time subject to a certain \emph{system
  dynamics}: We assume that time is discretized into \emph{stages} $T
:= \{0, 1, 2, \dots, N\}$.  The opinion of voter~$i \in I$ in stage~$t
\in T$ is denoted by~$x^t_i$.  We call, as usual, the vector
$\mathbf{x}^t := (x^t_i)_{i \in I}$ the \emph{state} of the system in
stage~$t$.  The \emph{system dynamics} $\mathbf{f}^t$ is a vector
valued function that computes the state of the system
$\mathbf{x}^{t+1}$ as $\mathbf{x}^{t+1} :=
\mathbf{f}^t(\mathbf{x}^t)$.  We assume a given \emph{start value}
$x^{\text{start}}_i$ for the current opinion of each voter~$i \in I$.
Thus, $x^0_i = x^{\text{start}}_i$ holds for all $i \in I$.

Depending on how $\mathbf{f}^t$ is defined, we obtain different models
of opinion dynamics.  In this paper, we will only consider so-called
\emph{stationary} models, where $\mathbf{f}^t$ does not depend on the
stage~$t$.  Therefore, from now on, we will drop the superscript $t$
from the notation and write $\mathbf{f}$ for the system dynamics.

%% \subsection{The Average Model}
%% \label{sec:Average}

%% The motivation for this model is that each voter is in contact with each
%% other in every stage, and each opinion is influenced by each other (including
%% itself) by the same amount.  The mathematical model for this is to define
%% $\mathbf{f}$ as the arithmetic mean of all opinions.

%% This is boring because after the first stage all opinions are equal
%% (consensus).

\subsection{The DeGroot Model}
\label{sec:DeGroot}

In this model, each voter is again in contact with each other in every
stage.  The strengths of the influences of opinions on other opinions
are given by \revJ{non-negative} weights $w_{ij}$ with $\sum_{j \in I}
w_{ij} = 1$ for all $i \in I$, with the meaning that the opinion of
voter~$i$ is influenced by the opinion of voter~$j$ with
weight~$w_{ij}$.  The mathematical formulation of this is to define
$\mathbf{f} = (f_i)_{i \in I}$ as a weighted arithmetic mean in the
following way:
\begin{equation}
  \label{eq:DeGroot}
  f_i(x_1, \dots, x_n) := \sum_{j \in I} w_{ij} x_j.
\end{equation}

It can be shown that this, in the limit, leads to
consensus.\footnote{This is an easy consequence of the Banach Fixed
  Point Theorem, since this dynamics is a contraction.} It leads, as
we will see below, still to an interesting optimal control problem.

\subsection{The Bounded-Confidence Model}
\label{sec:Bounded-Confidence}

The motivation for this model is that our voters ignore too distant
opinions of others.  Formally, we fix once and for all \rev{voters and stages} an $\epsilon
\in (0,1)$, and each voter is influenced only by opinions that are no
more than $\epsilon$ away from his or her own opinion.  We call
$[x^t_i - \epsilon, x^t_i + \epsilon]\cap[0,1]$ the \emph{confidence interval
  of voter~$i$ in \revR{stage}~$t$}.  Let the \emph{confidence set} $I_i(x_1,
\dots, x_n)$ of voter $i \in I$ in state $\mathbf{x} = (x_1, \dots,
x_n)$ be defined as
\begin{equation}
  \label{eq:ConfidenceSet}
  I_i(x_1, \dots, x_n) := \setof{j \in I}{\abs{x_j - x_i} \le \epsilon}.
\end{equation}
Observe that $I_i(x_1, \dots, x_n)\neq\emptyset$ due to $i\in I_i(x_1, \dots, x_n)$.

Then the system dynamics of the BC-model is given as follows:
\begin{equation}
  \label{eq:Bounded-Confidence}
  f_i(x_1, \dots, x_n) := \frac{1}{\abs{I_i(x_1, \dots, x_n)}} \sum_{j \in I_i(x_1, \dots, x_n)} x_j.
\end{equation}

%% This system dynamics is mathematically extremely interesting because
%% it is not even continuous.  It can be argued about whether it makes
%% sense that confidence is dropped completely at a very sharp point in
%% opinion space or whether a more continuous transition should be
%% chosen.  Our personal experience is that it actually happens that all
%% of a sudden, because of some event, we lose our faith in some person;
%% this supports the non-continuous characteristics of the
%% bounded-confidence model. 
A possible extension might be a stochastic
disturbance on $\epsilon$, but, as we will see, bounded-confidence is
still far from being completely understood.  Therefore, in this paper
bounded confidence will be in the main focus.

\subsection{A New Opinion Control Model}
\label{sec:Control}

Given a dynamical system as above, we can of course think about the
possibility of a control that can influence the system dynamics.
Formally, this means that the system dynamics $\mathbf{f}$ depends
also on some additional exogenously provided data $\mathbf{u}$, the
\emph{control}.

%% The easiest way to think about opinion control is to carefully state
%% an opinion in front of all voters so that the new opinion takes part
%% in influencing all the voters' opinions.  Formally, the controller can
%% place one or more additional opinions in the opinion space in order to
%% guide the voters' opinions in a specified direction.  One possible
%% interpretation of this to place suitable statements in the stages of a
%% marketing campaign in order to convince as many customers as possible
%% to buy the product rather than the competition; another is to present
%% political speeches with consciously designed opinions during the
%% stages of an election campaign in order to get as many votes as
%% possible, i.\,e., the opinions of as many as possible voters are
%% closer to the party than to the competition.

Formally, this means in the simplest case (and we will restrict to
this case) that the controller can present an additional opinion $u^t$
in every stage that takes part in the opinion dynamics.  The
corresponding system dynamics, taking the control as an additional
argument, are then given as follows (with $x_0 := u$ and $I_0 := I
\cup \{0\}$ as well as $w_{ij}$ this time with $\sum_{j \in I_0}
w_{ij} = 1$ for easier notation):

\begin{align}
  \label{eq:Opinion-Control-DeGroot}
  f_i(x_0; x_1, \dots, x_n)
  &:=
  \sum_{j \in I_0} w_{ij} x_j,\tag{DeGroot-Control}\\
  \label{eq:Opinion-Control-Bounded-Confidence}
  f_i(x_0; x_1, \dots, x_n)
  &:=
  \frac{1}{\abs{I_i(x_0, x_1, \dots, x_n)}} \sum_{j \in I_i(x_0, x_1, \dots, x_n)} x_j \tag{Bounded-Confidence-Control}.
\end{align}

We can interpret this as a usual model of opinion dynamics with an
additional opinion~$x_0$ that can be positioned freely in every stage
by the controller.
%% We still have to formalize what the goal of the controller is.  In
%% marketing or politics it would be desirable that after a certain
%% number of stages as many opinions as possible are closer to a target
%% opinion (favor the product over the competition, vote for the party)
%% than to competing opinions.  In the case of fixed target opinions and
%% fixed competing opinions this can always be expressed
The aim of the controller is as follows:
Control opinions in a way such that after $N$ stages there are as many
opinions as possible in a given interval $[\ell, r] \subseteq [0,1]$.

To formalize this, fix an interval $[\ell, r]$ (the \emph{conviction
  interval}), and let \emph{the conviction set} $J(x_1, \dots, x_n)$ denote
the set of all voters $j \in I$ with $x_j \in [\ell, r]$.  We want to
maximize the number of convinced voters.  Thus, the problem we want to
investigate is the following deterministic discrete-time optimal control
problem:
\begin{align}
  \label{eq:Optimal-Control}
  \max_{x^0_0, x^1_0, \dots, x^{N-1}_0} &\abs{J(x^N_1, \dots, x^N_n)}\notag\\
  \text{subject to}\notag\\
  x^{0}_i &= x^{\text{start}}_i 
  & \forall i \in I, %t = 0, 1, \dots, N-1,
  \tag{Start Configuration}\\
  x^{t+1}_i &= f_i(x^t_0; x^t_1, \dots, x^t_n)
  & \forall i \in I, t = 0, 1, \dots, N-1,\tag{System Dynamics}\\
  x^t_0 &\in [0,1]
  & \forall t = 0, 1, \dots, N-1, \tag{Control Restrictions}
\end{align}
where $\mathbf{f} = (f_i)_{i \in I}$ is one of the controlled system dynamics
in Equations~\eqref{eq:Opinion-Control-DeGroot}
and~\eqref{eq:Opinion-Control-Bounded-Confidence}, resp.

\section{Simulation and Pitfalls from Numerical Mathematics}
\label{sec:Numerics}

%% In this section we want to convince the reader from the fact that the
%% numerical inaccuracies even in a simple simulation of the bounded confidence
%% model (as opposed to the DG-model) have drastic effects on the
%% results observed.  The only way for us to cope with this problem is to resort
%% to exact rational arithmetics throughout, although there may be more
%% sophisticated methods to improve efficiency.  This numerical instability has
%% the more serious consequence that of-the-shelf optimization algorithms with
%% floating point arithmetic's can not be used without checking the results for
%% correctness in exact arithmetics.
\revJ{No matter what we do: a computer can represent only finitely
  many distinct numbers.  Thus, it is impossible that a computer can
  distinguish infinitely many numbers like there exist in $[0,1]$.
  Even worse: if a number type is used for computer programming that
  uses a fixed number of bits, i.e., zeros and ones in the binary
  representation, like ``float'' or ``double'', then distinct real
  numbers that are very close can be interpreted as identical numbers
  by the computer.  This has a serious influence on our ability to
  check correctly whether or not one opinion is in the confidence
  interval of another whenever this is a cutting-edge decision.
%% Before we can start with our benchmark problem we have to realize a
%% very inconvenient fact: 
  More specifically: }If we try to solve our benchmark problem for the
BC-model, the computer has -- again and again -- to decide the
question whether or not $\abs{x_i - x_j} \le \epsilon$, where $x_i$,
$x_j$, and $\epsilon$ are, in particular, \emph{real} numbers. For a
human being with a tiny bit of math training it is easy to answer the
question whether $\abs{0.6 - 0.4} \le 0.2$. If a computer has to
answer that simple question and uses what in some programming
languages is called the data format ``real'' or ``float'' or even
``double'' (float with double precision), then the computer might get
it wrong. In Figure~1, left, one can see that effect: We start the
dynamics with $6$ voters that are regularly distributed at the
positions $0, 0.2, 0.4, \ldots, 1$. The confidence radius $\epsilon$
is~$0.2$.\revJ{\footnote{The number $0.2$ is tricky because its binary
    representation as a floating point number is determined by
    $0.2 = \frac{1}{8} + \frac{1}{16} + \frac{1}{128} + \frac{1}{256}
    + \dots {} = 0.\overline{0011}_2$,
    i.e., its exact representation would need infinitely many bits.
    For example, a cutoff bit representation would be identical to a
    smaller number than~$0.2$.  Thus, checking whether or not some
    number is smaller or larger than~$0.2$ checks this question not
    for $0.2$ but for a smaller number.  How exactly $0.2$ is finitely
    represented in a computer is usually defined by an IEEE
    standard.}}  Obviously the computer (using a program written in
Delphi, but in NetLogo the analogous mistake would happen, possibly
somewhere else) answers the question whether $\abs{0.6 - 0.4} \le 0.2$
in a wrong way. As a consequence, from the first update onwards the
dynamics is corrupted: Given our start distribution and the
homogeneous, constant, and symmetric confidence radius, the opinion
stream should be mirror symmetric with respect to a horizontal line at
$y=0.5$.  That symmetry is forever destroyed by the very first update.
What happens here is no accident. It is the necessary consequence of
the floating point arithmetic that computers use to approximate
\revR{real} numbers.  Using floating point arithmetic each number is
represented with a finite number of binary digits, so a small error is
possibly made. For a hard decision like $\abs{x_i - x_j} \le \epsilon$
or $\abs{x_i - x_j} > \epsilon$ a small error is sufficient to draw
the wrong conclusion, whenever $\abs{x_i - x_j}$ equals or is rather
close to~$\epsilon$.  The only way for us to cope with this problem is
to resort to exact rational arithmetics throughout, although there may
be more sophisticated methods to improve efficiency. This numerical
instability has the more serious consequence that of-the-shelf
optimization algorithms with floating point arithmetics can not be
used without checking the results for correctness in exact
arithmetics.\footnote{\revR{\citet{Polhill+Izquierdo+Gotts:FloatingPointGhost:2005}
    demonstrate that, based upon floating-point numbers, in several
    agent-based models branching statements lead to severe numerical
    artefacts.}}

%% Let us start with an example of $6$ voters with opinions being
%% regularly distributed at the positions $0.0$, $0.2$, $\dots$, $1.0$
%% with $\epsilon=0.2$.  For a moment we forget about the control and
%% focus on the consensus process. Since there is a mirror symmetry
%% around opinion $0.5$ and we assume no external control that might
%% destroy this symmetry there should be such a symmetry in each stage.

Using exact arithmetic we obtain that the opinions of our voters are given by
\begin{eqnarray*}
 x^0&=&(0.0,0.2,0.4,0.6,0.8,1.0),\\
 x^1&=&(0.1,0.2,0.4,0.6,0.8,0.9),\\
 x^2&=&(0.15,0.2\overline{3},0.4,0.6,0.7\overline{6},0.85),\\
 x^3&=&(0.191\overline{6},0.26\overline{1},0.4\overline{1},0.5\overline{8},0.73\overline{8},
      0.808\overline{3})\\
    &=&\left(\frac{23}{120},\frac{47}{180},\frac{37}{90},\frac{53}{90},\frac{133}{180},
      \frac{97}{120}\right),\\
 x^4&=&\left(\frac{163}{720},\frac{311}{1080},\frac{227}{540},\frac{313}{540},\frac{769}{1080},
      \frac{557}{720}\right),\\
 x^5&=&\left(\frac{673}{2160},\frac{673}{2160},\frac{3271}{8640},\frac{5369}{8640},\frac{1487}{2160},
      \frac{1487}{2160}\right),\\
 x^6&=&\left(\frac{577}{1728},\frac{577}{1728},\frac{577}{1728},\frac{1151}{1728},\frac{1151}{1728},
      \frac{1151}{1728}\right)\\
    &=&\Big(0.333912\overline{037},0.333912\overline{037},0.333912\overline{037},0.666087\overline{962},\\
    & &0.666087\overline{962},0.666087\overline{962}\Big).
\end{eqnarray*}
The corresponding correct trajectory is drawn on the right hand side
of Figure~\ref{fig_numerical_disaster}.

%% \begin{figure}[htp]
%%   \begin{center}
%%     \includegraphics[width=12cm]{Disaster.pdf}
%%     \caption{A computational disaster.}
%%     \label{fig_numerical_disaster}
%%   \end{center}
%% \end{figure}

\begin{figure}[htp]
  \centering
  \includegraphics[width=0.35\linewidth]{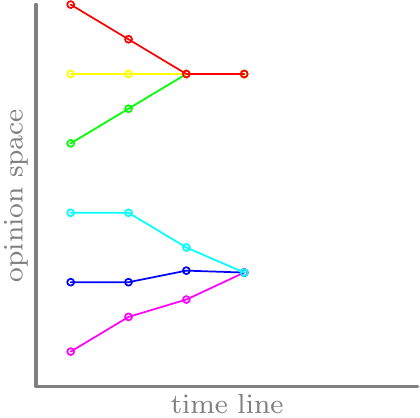}\quad\quad\quad\quad\quad
  \includegraphics[width=0.35\linewidth]{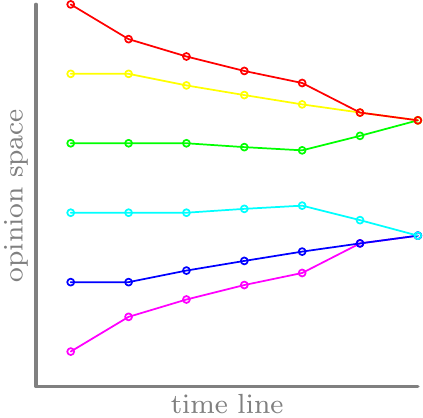}%
  \caption{A computational disaster caused by a tiny inaccuracy (left:
    numerical artefact; right: correct result)}
  \label{fig_numerical_disaster}
\end{figure}

As we have seen, a small error in the positions of the
voters, computational or by observation, can have a drastic
effect. We mention that this effect can not happen for numerical
stable dynamics like, e.g., the DG-model. The only patch that came
to our mind which was capable of dealing with the numerical
instability was to use exact arithmetic. This means that we represent
all numbers as fractions where the numerator and the denominator are
integers with unlimited accuracy. We remark that we have used the
Class Library of Numbers (CLN) a C++-package, but similar packages
should be available also for \revR{other} programming languages.

There are quite a lot of articles dealing with the simulation of the
BC-model. To our knowledge none of these mentioned the use of exact
arithmetic. So one could assume that the authors have used ordinary
floating point numbers with limited precision for there
considerations. It is an interesting question whether all of these
obtained results remain more or less the same if being recalculated
with exact arithmetic.  For no existing publication, however, we found
any evidence that the conclusions are only artefacts of numerical
trouble.  For results using randomized starting configurations the
probability is zero that the distance between agents equals the
confidence radius.  In those experiments numerical artifacts are much
less likely (though not impossible) than in simulations starting from
equidistant configurations.

We have to admit, that in the starting 
phase of our investigation in optimal control of opinion dynamics,  
we have also used floating point arithmetic. We heuristically found controls
achieving $10$ voters after $10$ stages. Using exact arithmetic it turned out
that the computed control yields only $4$ convinced voters, which is a really 
bad control, as we will see later on.

\section{Basic Structural Properties of Optimal Controls}
\label{sec:Structures}
In this section we collect some basic facts about structural
properties of optimal controls, mainly for the BC-model. While 
generally the DG-model has nicer theoretical properties, there is an
exception when considering the ordering of the voters over time.

\begin{Lemma}(Cf. \cite[Lemma 2]{krause2000discrete}.)
  \label{lem:ordering}
  Consider an instance of the BC-model (with control or without).
  \begin{enumerate}
    \item[(1)] If $x_i^t=x_j^t$, then $x_i^{t+1}=x_j^{t+1}$.
    \item[(2)] If $x_i^t\le x_j^t$, then $x_i^{t+1}\le x_j^{t+1}$.
  \end{enumerate}
\end{Lemma}
\begin{proof}
  If the positions of voter $i$ and voter $j$ coincide at stage $t$, then they have identical confidence 
  sets and the system dynamics yields the same positions for $i$ and $j$ at stage $t+1$. For (2), we assume 
  $x_i^t<x_j^t$ w.l.o.g. We set $C=I_i(x_1^t,\dots,x_n^t)\cap I_j(x_1^t,\dots,x_n^t)$, $L=I_i(x_1^t,\dots,x_n^t)\backslash C$, 
  and $R=I_j(x_1^t,\dots,x_n^t)\backslash C$. Due to $x_i^t<x_j^t$ we
  have $x_l^t<x_c^t<x_r^t$ for all $l\in L$, $c\in C$, and $r\in R$. 
  Thus $x_i^{t+1}\le x_j^{t+1}$.
\end{proof}

The analogous statement for the DG-model is wrong in general.

Next we observe that one or usually a whole set of optimal
controls exists. The number of convinced voters is in any stage
trivially bounded from above by the total number of voters
$|I|$. Hence, to every control there corresponds a bounded integer
valued number of convinced voters. With some technical
effort an explicit bound can be computed.

  First, we observe that there are some boundary effects. Consider a
  single voter with start value $x_1^0=\frac{1}{2}$ and
  $\epsilon=\frac{1}{2}$. Further suppose that the conviction interval
  is given by $[0,\delta]$, where $\delta$ is small. The most
  effective way to move the opinion of voter $1$ towards $0$ is to
  place a control at $0$ at each stage. With this we obtain
  $x_1^t=\frac{1}{2^{t+1}}$ for all $t$. Thus the time when voters~$1$
  can be convinced depends on the length $\delta$ of the conviction
  interval. This is due to the fact that we can not place the control
  at $x_1^t-\epsilon$ if $x_1^t$ is too close to the boundary. The
 same reasoning applies for the other boundary at~$1$.  In order to
  ease the exposition and the technical notation we assume that no
  opinion is near to the boundaries in the following lemma.

\begin{Lemma}
 Consider an instance of the BC-model such that the start values and the conviction interval
 $[l,r]$ are contained in $[\epsilon,1-\epsilon]$. It is possible to select suitable 
 controls at each stage such that after at most $\frac{2n+1}{\epsilon}+2$ stages all $|I|$ voters are convinced.
\end{Lemma}
\begin{proof}
  We will proceed in two steps. In the first step we ensure that all voters have the same 
  opinion after a certain amount of stages. In the second step we will move the coinciding 
  opinions inside the conviction interval.

    Without loss of generality, we assume the ordering
    $x_1^0\le\dots\le x_n^0$ and observe $x_n^0-x_1^0\le 1-2\epsilon$.
    If $x_n^t-x_1^t>\epsilon$, we place a control at
    $\rev{x_1^t}+\epsilon<x_n^t$. \rev{As an abbreviation we set 
    $R=I_1\!\left(x_1^t,\dots,x_n^t\right)$.}
    At most $n-1$ of the $n$ voters can be inside the confidence set of 
    voter~$1$, \rev{i.e., $|R|\le n-1$,} and \rev{we have $x_i^t\ge x_1^t$ 
    for all $i\in R$,} see Lemma~\ref{lem:ordering}. 
    %% Thus we have $x_1^{t+1}\ge x_1^t+\frac{\epsilon}{n}$. 
    \rev{With this we conclude
    $$
      x_1^{t+1}=\frac{1}{|R|+1}\cdot\left(x_0^t+\sum_{i\in R}x_i^t \right)
      \ge \frac{1}{|R|+1}\cdot\left( x_1^t+\epsilon+\sum_{i\in R}x_1^t \right)=x_1^t+\frac{\epsilon}{|R|+1}
      \ge x_1^t+\frac{\epsilon}{n}.
    $$}
    After at most $n\cdot
    (\left\lceil\frac{1}{\epsilon}\right\rceil-3)\le
    \frac{n}{\epsilon}$ stages we can achieve
    $x_n^t-x_1^t\le\epsilon$. Then placing the control at
    $\frac{x_1^t+x_n^t}{2}$ yields the same opinion, which is also
    inside $[\epsilon,1-\epsilon]$, for all voters after at most
    $\frac{n}{\epsilon}+1$ stages, i.e., the first step is completed.
 
  In Step 2 we proceed as follows. If $x_1^t\in[l,r]$ nothing needs to be done. Due to symmetry 
  we assume $x_1^t<l$ in the following. If $l-x_1^t\ge \frac{\epsilon}{n+1}$ we place a control at
  $x_1^t+\epsilon$ so that $x_1^{t+1}=x_1^t+\frac{\epsilon}{n+1}$, since all voters influence voter $1$.
  After at most $(n+1)\epsilon$ stages we have $l-x_1^t< \frac{\epsilon}{n+1}$. In such a situation we 
  set the control to $x_1^t+(n+1)(l-x_1^t)$ such that $x_1^{t+1}=l$. 

  Applying Lemma~\ref{lem:ordering} again, we conclude that we can achieve $x_i^t=l$ after at most
  $\frac{2n+1}{\epsilon}+2$ for all $i\in I$. Taking the control as $x_i^t$ we can clearly preserve the 
  configuration to later stages.
\end{proof}

Thus, given enough time (number of stages) we could always achieve the upper bound of 
$|I|$ convinced voters. By setting $[l,r]=[1-\epsilon,1-\epsilon]$ and $x_i^0=\epsilon$, 
we see that the stated estimation gives the right order of magnitude in the worst case.

For the DG-model the upper bound on the time needed to convince 
all voters depends on the influence $w_{i0}$ of the control for each voter.
To this end we define $\omega=\min_{i\in I} w_{i0}$, i.e., the tightest possible 
lower bound on the influences. Since it may happen that no stable state is reached 
after a finite number of stages, we can only navigate the voters into an interval of 
length greater than zero. 

\begin{Lemma}
  Consider an instance of the DG-model with $0<\omega=\min_{i\in I} w_{i0}<1$, $\delta\in(0,1)$, and a position $p\in[0,1]$. It is possible to select suitable 
  controls at each stage such that after at most $\frac{\log(\delta)}{\log(1-\omega)}$ stages all $|I|$ voters have an opinion 
  within the interval $[p-\delta,p+\delta]$.
\end{Lemma}
\begin{proof}
  By induction over $t$ we prove that we have $\left|x_i^t-p\right|\le (1-\omega)^t$ for all $i\in I$ and $t\in\mathbb{N}$, if we place the control
  at position $p$ at all stages. Since $x_i^t,p\in [0,1]$ we have $\left|x_i^0-p\right|\le (1-\omega)^0=1$ for all $i\in I$. For $t\ge 1$ we have
  \begin{eqnarray*}
    x_i^{t}&=&\sum_{j\in I} w_{ij}\cdot x_j^{t-1}\,+\,w_{i0}\cdot p
           \,\,\ge\,\, \sum_{j\in I} w_{ij}\cdot \left(p-(1-\omega)^{t-1}\right)\,+\,w_{i0}\cdot p\\
           &\ge& (1-\omega)\cdot \left(p-(1-\omega)^{t-1}\right)+\omega\cdot p\,\,=\,\, p-(1-\omega)^t
  \end{eqnarray*}
  for all $i\in I$. Similarly we conclude $x_i^{t}\le p+(1-\omega)^t$.
\end{proof}

Thus, given enough time (number of stages) we could always achieve the
upper bound of $|I|$ convinced voters if the conviction interval has a
length greater than zero. By setting $p=1$ and $x_i^0=0$, we see that
the stated estimation gives the right order of magnitude in the worst
case.  Using the Taylor expansion of $\log(1-\omega)$ and having an
influence that decreases with the number of voters in mind, we remark
that
\begin{equation*}
  \frac{\log(\delta)}{\log\!\left(1-\frac{1}{n+1}\right)}\le
  -(n+1)\cdot\log(\delta).
\end{equation*}

\section{Computational Information on Optimal Controls}
\label{sec:computational-info}

%% In this section we present findings derived from mathematical models,
%% namely mixed integer linear programming models (MILP), for our optimal
%% control problems.  Such models serve two purposes:
%% \begin{itemize}
%% \item By reducing the number of periods in the model, we can exactly
%%   solve the model.  This allows us to employ a receding-horizon
%%   heuristics to the original problem (see Section
%%   \ref{sec:MPC} for results).
%% \item The model allows for the computation of performance bounds, in
%%   our case upper bounds on the number of convinced voters achievable
%%   by a control.
%% \end{itemize}

How can one find reliable information on optimal controls and their
resulting optimal achievements?  That is, for a special instance like
our benchmark instance, how can we find out, how many convinced voters
are achievable for a given horizon~$N$?  It is certainly not possible
to try all possible controls and pick the best -- there are uncountably
infinitely many feasible controls, because all elements in $[0, 1]^N$
constitute feasible controls.  On the other hand, some logical
constraints are immediate without enumeration: it is impossible to
achieve more convinced voters than there are voters.  

A common technique to supersede such a trivial statement without
complete enumeration is to devise a \emph{mathematical model} and find
solutions to it by \emph{exact methods}.  Exact methods depend on
generic logical reasoning or provably correct computational
information about the solution of a mathematical model.\footnote{An
example of generic logical reasoning can be seen in the previous
section.}  In this section, we use \emph{mixed integer linear
programming (MILP)} for modeling the DG and the BC optimal control
problems \rev{and \texttt{cplex} for solving them, see 
Table~\ref{tab:MILP-cplex-parameters} for the precise parameter settings}.

While the models are generically correct, concrete computational
results will only be given for our benchmark problem and related data.
One big advantage of MILP models is that there is commercial software
of-the-shelf that can provide solutions to such models regardless of
what they represent.  There is even academic free software that is
able to provide solutions stunningly fast.

In this section, we will not spell out the formulae of our models
explicitly.\footnote{Mathematically explicit descriptions, suitable
  for replicating our results, can be found in
  Appendix~\ref{sec:Mathematical-Model} in the appendix.}  Instead, we
try to emphasize the features of our approach.

First, an optimal solution to an MILP model is \emph{globally}
optimal.  That is, no better solution exists \emph{anywhere} in the
solution space.  Second, if an optimal solution to an MILP model was
reported by the solver software, we are \emph{sure} (within the bounds
of numerical accuracy) that it is an optimal solution, i.e., the
method is \emph{exact}.  Third, if an optimal solution to an MILP
could not be found in reasonable time, then very often we still obtain
\emph{bounds} on the value of an (otherwise unknown) optimal solution.
And fourth, as usual the process of constructing an MILP model is
\emph{non-unique}, i.e., usually there are many, substantially
different options to build an MILP model for the same problem, and one
may provide solutions faster or for larger instances than another.

We built an MILP model for the DG optimal control problem and two MILP
models for the (much more difficult) BC optimal control problem.

\subsection{Principle Ideas of the MILP models}
\label{sec:MILP-models}

The MILP model for the DG optimal control problem can be classified as
a \emph{straight-forward} MILP: the system dynamics is linear and fits
therefore the modeling paradigm of MILP very well.  The only little
complication is to model the number of convinced voters, which is a
non-linear, non-continuous function of the voters' opinions.  Since
binary variables are allowed in MILP, we can construct such functions
in many cases using the so-called ``Big-$M$ method.''  Details are
described in the Appendix~\ref{sec:MILP-DeGroot}.

An MILP model for the BC optimal control problem is \emph{not
  straight-forward} at all.  Since the system dynamics depends on
whether or not some voter is in the confidence interval of another, we
have to decide at some point whether or not two voters' distance is
either $\le \epsilon$ or $> \epsilon$.  It is another general feature
that strict inequalities cause trouble for MILP modeling, and the
distinction would be numerically unstable anyway (most MILP solvers
use floating-point arithmetic, see Section~\ref{sec:Numerics}).  Thus,
we refrained from trying to build a correct MILP model for the BC
optimal control problem.  Instead, we built two complementary MILP
models.  Without referring to the details, we again explain only the
features.  In the first MILP model, the \emph{lower-bound model}, any
feasible solution defines a feasible control, which achieves, when
checked with exact arithmetic, at least as many convinced voters as
predicted in the MILP model.  The second MILP model, the \emph{dual
  model}, is some kind of \emph{relaxation}: No control can convince
more voters than the number of convinced voters predicted for any of
its optimal solutions.  

This is implemented by using a \emph{safety margin} $\hat{\epsilon}$
for the confidence interval.  In the models, it is now required for
any feasible control that it leads, at all times, to differences
between any pair of voters that are either $\le \epsilon$ or
$\ge \epsilon + \hat{\epsilon}$.  If $\hat{\epsilon} > 0$, we obtain a
lower-bound model, since some originally feasible controls are
excluded because they lead to differences between voters that are too
close to the confidence radius~$\epsilon$.  If $\hat{\epsilon} \le 0$,
we obtain an \emph{upper-bound model} where the requirements for ``are
within distance~$\epsilon$'' and ``are at distance at
least~$\epsilon + \hat{\epsilon}$'' overlap so that the solution with
better objective function value can be chosen by the solver software.

Now, if we put together the information from \emph{both} models then
we can achieve more: If the optimal numbers of convinced voters
coincide in both models, then we have found the provably optimal
achievable number of convinced voters although we had no single model
for it.  Otherwise, we obtain at least upper and lower bounds.
Moreover, any lower bound for the number of convinced voters predicted
by the lower-bound model is a lower bound on the number of achievable
convinced voters, and any upper bound on the number of convinced
voters predicted by the upper-bound model is an upper bound on the number of
achievable convinced voters.

The first MILP model for the BC-model is a \emph{basic} model with a
compact number of variables along the lines of the DG MILP.  However,
the system dynamics is discontinuous this time, which requires a much
heavier use of the Big-$M$ method.  MILP-experience tells us that too
much use of the Big-$M$ method leads to difficulties in the solution
process.  Since the basic model did indeed not scale well to a larger
number of rounds, we engineered alternative models.

The resulting \emph{advanced} MILP model for
the BC optimal control problem has substantially more variables but
not so many Big-$M$ constructions.  Moreover, the advanced model uses
a perturbed objective function: Its integral part is the predicted
number of convinced voters, and one minus its fractional part
represents the average distance of the unconvinced voters to the
conviction interval.  This perturbation was introduced in order to
better guide the solution process.  The problem with the unperturbed
objective function is that many feasible controls unavoidably achieve
identical numbers of convinced voters because there are simply much
fewer distinct objective function values than controls; this is a
\emph{degenerate situation}, which is difficult to handle for the MILP
solver.  The perturbation guarantees that two distinct controls are
very likely to have distinct objective function values.

Our hypothesis was that the advanced model would be easier to solve,
which, to a certain extent, turned out to be true.  We know of no
other method to date that yields more \emph{provable and global}
information about optimal controls for the BC optimal control problem.

\subsection{Computational Results on the Benchmark Campaign Problem}
\label{sec:comp-results-campaign}

In the following we report on our computational results.
\revJ{We first compare the effectiveness of the two modeling approaches
before we restrict ourselves to the more successful
model.\footnote{Because of its special characteristics the advanced
  model was accepted for the benchmark suite
  MIPLIB~2010~\cite{Koch+Achterberg:MIPLIB2010:2011}.  Thus, the new
  model will automatically receive some attention by the developers of
  MILP solver software and MILP researchers, which may help to clarify
  the yet open cases.}}

%% \subsection{Results for the MILP for the DG Optimal Control}
%% \label{sec:DG-model}

The MILP for the DG-model was very effective.  It could be solved in
seconds for the benchmark problem with homogeneous weights.  Eleven
convinced voters are possible for any horizon, and, of course, no
more.  The control is non-crucial here, because homogeneous weights
lead to an immediate consensus in the conviction interval.  But also
for other weights, optimal solutions can be found for all horizons
very fast.  The real conclusion is that solving the DG optimal control
problem on the scale of our benchmark problem is easy, but there are
no mind-blowing observations about optimal controls.

%% \subsection{Results for the Basic MILP for the BC Optimal Control}
%% \label{sec:basic-BC-model}

Using the basic MILP revealed that the BC-model is in an
all different ball-park.
\begin{table}[h]
  \centering\footnotesize
  \begin{tabular}{rrrrrr}
    \toprule
    \# stages & optimal value & CPU time [s] & \# variables/binary/integer & \# constraints & \# non-zeroes\\
    \midrule
    1 & 3     &       0.01 &  \hphantom{0\,}2454/\hphantom{0\,}1738/\hphantom{0}572  &     3487 &     11\,330\\
    2 & 4     &       1.42 &  \hphantom{0\,}4776/\hphantom{0\,}3377/1122 &     6809 &     22\,319\\
    3 & 5     &     355.33 &  \hphantom{0\,}7098/\hphantom{0\,}5016/1672 &  10\,131 &     33\,308\\
    4 & 5--11 &    3600.00 &  \hphantom{0\,}9420/\hphantom{0\,}6655/2222 &  13\,453 &     44\,297\\
    5 & 4--11 &    3600.00 &  11\,742/\hphantom{0\,}8294/2772    &  16\,775 &     55\,286\\
    6 & 5--11 &    3600.00 &  14\,064/\hphantom{0\,}9933/3322    &  20\,097 &     66\,275\\
    7 & 3--11 &    3600.00 &  16\,386/11\,572/3872 &  23\,419 &     77\,264\\
    8 & 4--11 &    3600.00 &  18\,708/13\,211/4422 &  26\,741 &     88\,253\\
    9 & 0--11 &    3600.00 &  21\,030/14\,850/4972 &  30\,063 &     99\,242\\
    10& 0--11 &    3600.00 &  23\,352/16\,489/5522 &  33\,385 &    110\,231\\
    \bottomrule    
  \end{tabular}
  \caption{Results of the basic MILP model for the benchmark problem
    with a positive~$\hat{\epsilon}$ (provably feasible configurations);
    number of variables/constraints/non-zeroes for original problem
    before preprocessing; the time limit was 1h = 3600s;
    MacBook Pro 2013, 2.6\,GHz Intel Core i7, 16\,GB 1600\,MHz DDR3
    RAM, OS~X~10.9.2, \texttt{zimpl} 3.3.1, \texttt{cplex} 12.5
    (Academic Initiative License), branching
    priorities given according to stage structure.} 
  \label{tab:MILP-basic-results}
\end{table}
Table~\ref{tab:MILP-basic-results} shows the computational results for
our basic \revJ{lower-bound} model, \revJ{in particular},
with~$\hat{\epsilon} = 10^{-5} > 0$ yielding provably feasible
controls.\revJ{\footnote{We have added for each instance the typical
  information about the MILP problem scale, which is characterized by
  the number of variables, integrality requirements, and non-zero
  constraint coefficients in the model.  The higher these values the
  larger and harder the problem is considered in the MILP world.  Note
  that in addition to the obvious decision variables modeling the
  actual control opinion there is a large number of auxiliary
  variables necessary to achieve the correct logic. Conventional MILP
  solvers fall back into a guided enumeration (Branch-and-Bound) of
  discrete variables whenever they assess that no other information
  can be computed to their advantage. Branching priorities tell the
  MILP solver which integral variables should be enumerated first,
  second, etc.}} In order to really prove that no solutions with a
better objective (above the upper bound) exist, we would have to rerun
the computations with an $\hat{\epsilon} \le 0$.  We skipped this for
the basic model and did this only for the advanced model below, since
the information obtained by the advanced model is superior anyway.

\begin{figure}[htp]
  \centering
  \includegraphics[width=0.20\linewidth]{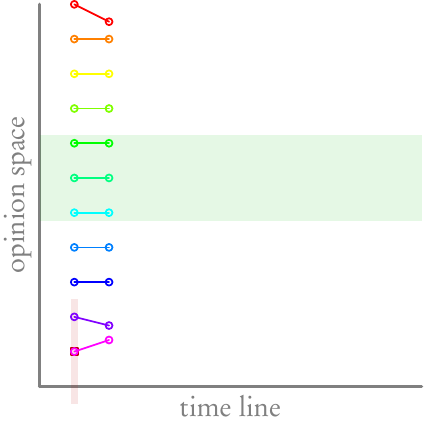}%
  \includegraphics[width=0.20\linewidth]{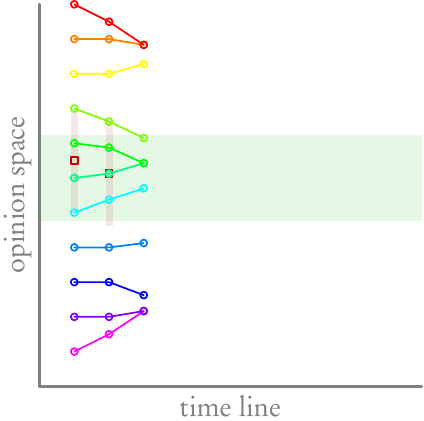}%
  \includegraphics[width=0.20\linewidth]{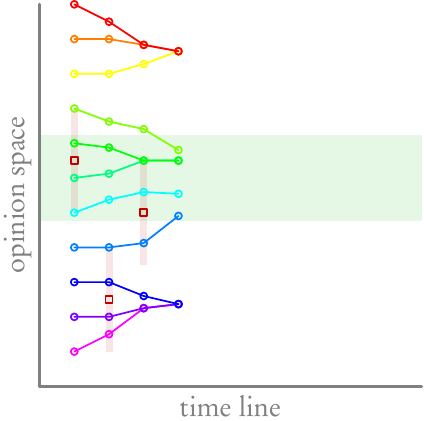}%
  \includegraphics[width=0.20\linewidth]{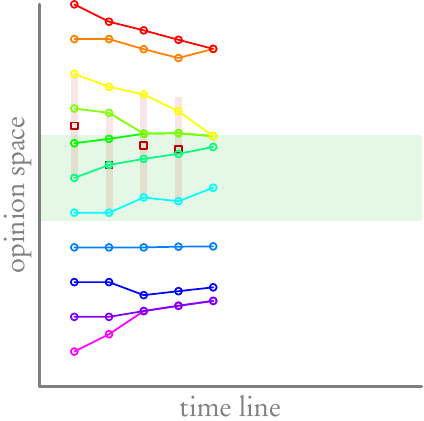}%
  \includegraphics[width=0.20\linewidth]{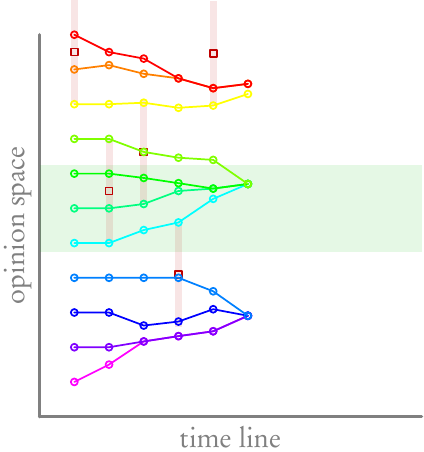}\\%
  \includegraphics[width=0.20\linewidth]{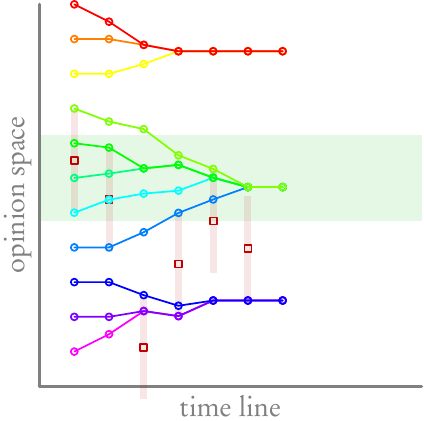}%
  \includegraphics[width=0.20\linewidth]{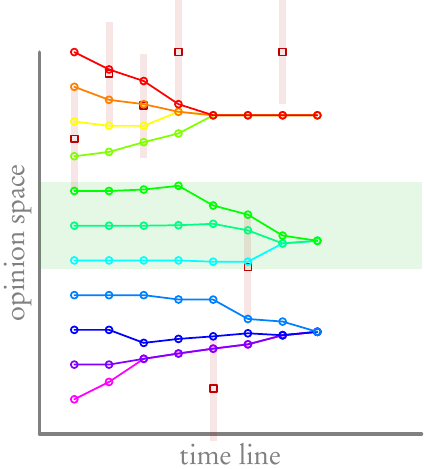}%
  \includegraphics[width=0.20\linewidth]{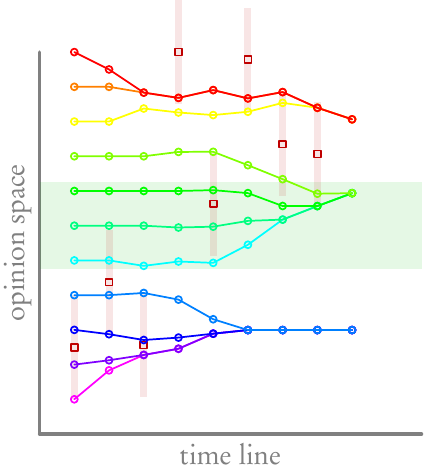}%
  \caption{Trajectories produced by solutions of the basic MILP (for
    $N \ge \revJ{4}$ we show the incumbent solution when the time
    limit of 1h was exceeded\revJ{, i.e., they only show the lower
      bounds on the optimal values reported in
      Table~\ref{tab:MILP-basic-results}}; for $N \ge 9$ no feasible
    solution was found in 1h\revJ{, thus, in those cases we learned no
      trajectories from the basic model}).}
  \label{fig:MILP-basic-trajectories}
\end{figure}
\revJ{The solutions of the basic lower-bound model determine for all
  stages, via the setting of the control and opinion variables,
  trajectories of opinions.  See Appendix~\ref{sec:MILP-DeGroot} for a
  detailed description of all variables. All feasible solutions of our
  lower-bound model correspond to the data of a BC-dynamics induced by
  the controls.  Thus, these trajectories are exactly the trajectories
  that would appear in a BC-simulation with the same controls.  The
  numbers of convinced voters induced by the trajectories always
  correspond to the lower values in the optimal value column of
  Table~\ref{tab:MILP-basic-results}.} \rev{The corresponding
  trajectories of the solutions of our basic lower-bound model are
  depicted in Figure~\ref{fig:MILP-basic-trajectories}. Here an empty
  red square represents the opinion of the control at that stage. The
  circles represent the opinions of the non-strategic agents. The
  range of influence for each control and the conviction interval are
  hinted by shaded regions.  One can clearly see that the structure of
  the optimal strategy heavily varies with~$N$. No conspicuous pattern
  is identifiable -- at least for us.}

%% \subsection{The Advanced MILP for the BC Optimal Control}
%% \label{sec:advanced-BC-model}

\begin{table}[h]
  \centering\footnotesize
  \begin{tabular}{rrrrrr}
    \toprule
    \# stages & optimal value & CPU time [s] & \# variables/binary/integer & \# constraints & \# non-zeroes\\
    \midrule
    1 & 3.600               &       0.02 &    2367/\hphantom{1\,}2321/0 &    6261 &  62\,551\\ 
    2 & 4.615               &       0.18 &    4656/\hphantom{1\,}4576/0 & 11\,420 & 122\,736\\
    3 & 5.640               &       2.40 &    6945/\hphantom{1\,}6831/0 & 16\,579 & 182\,921\\
    4 & 6.653               &     113.91 &    9234/\hphantom{1\,}9086/0 & 21\,738 & 243\,106\\
    5 & 6.676               &    1420.28 & 11\,523/11\,341/0 & 26\,897 & 303\,291\\
    6 & 6.691--11.785       &    3600.00 & 13\,812/13\,596/0 & 32\,056 & 363\,476\\
    7 & 5.703--11.826       &    3600.00 & 16\,101/15\,851/0 & 37\,215 & 423\,661\\
    8 & 8.725--11.853       &    3600.00 & 18\,390/18\,106/0 & 42\,374 & 483\,846\\
    9 & 7.746--11.872       &    3600.00 & 20\,679/20\,361/0 & 47\,533 & 544\,031\\
   10 & 8.761--11.893       &    3600.00 & 22\,968/22\,616/0 & 52\,692 & 604\,216\\
    \bottomrule    
  \end{tabular}
  \caption{Results of the advanced MILP model for the benchmark
    problem with an $\hat{\epsilon} = 10^{-5} > 0$ (provably feasible
    configurations); 
    number of variables/constraints/non-zeroes for original problem
    before preprocessing; the time limit was 1h = 3600s;
    MacBook Pro 2013, 2.6\,GHz Intel Core i7, 16\,GB 1600\,MHz DDR3
    RAM, OS~X~10.9.2, \texttt{zimpl} 3.3.1, \texttt{cplex} 12.5
    (Academic Initiative License), branching priorities given
    according to stage structure. \emph{Remark:} For a time limit of
    24h, $6$ stages can be solved to optimality (optimal value $6.697$)
    and for $10$ stages we obtain a solution with $11$ convinced voters
    (optimal value $11.800--11.840$).}
  \label{tab:MILP-advanced-results}
\end{table}
Table~\ref{tab:MILP-advanced-results} shows the computational results
for our advanced \revJ{lower-bound }model, \revJ{in particular, }with an
$\hat{\epsilon} = 10^{-5} > 0$, i.e., all obtained configurations are
feasible, i.e., the given controls provably produce this objective
function (within the numerical accuracy).  For $N = 6$ and larger,
\texttt{cplex} \cite{cplex:2014} could not find an upper bound with
fewer than $11$ convinced voters in one hour.  For a time limit of
$24h$, however, the optimum was determined for $N = 6$, and the
optimal number $11$ of convinced voters for $N = 10$ could be found.
Moreover, a better feasible control with objective $6.707$ for $N = 7$
and of $8.767$ for $N = 9$, respectively, could be computed in less
than 1h = 3600s on a faster computer\footnote{Mac Pro 2008
  2$\times$2.8\,GHz Quad-Core Intel Xeon, 21\,GB 800\,MHz DDR2 RAM,
  OS~X~10.9.5, \texttt{zimpl} 3.3.1, \texttt{cplex} 12.5 (Academic
  Initiative License}

The instance for $N =
10$ is special in the sense that the trivial bound of $11$ convinced
voters is sufficient to prove the ``voter''-optimality of a solution
with $11$ convinced voters.  It is yet an open problem to find a
configuration that provably maximizes the perturbed objective function
of the advanced model.
\begin{table}[h]
  \centering\footnotesize
  \begin{tabular}{rrrrrr}
    \toprule
    \# stages & optimal value & CPU time [s] & \# variables/binary/integer & \# constraints & \# non-zeroes\\
    \midrule
    1 & 3.600               &       0.06 &    2367/\hphantom{1\,}2321/0 &    6261 &  62\,551\\ 
    2 & 4.615               &       0.44 &    4656/\hphantom{1\,}4576/0 & 11\,420 & 122\,736\\
    3 & 5.640               &       7.84 &    6945/\hphantom{1\,}6831/0 & 16\,579 & 182\,921\\
    4 & 6.657               &     303.81 &    9234/\hphantom{1\,}9086/0 & 21\,738 & 243\,106\\
    5 & 6.668--11.743       &    3600.00 & 11\,523/11\,341/0 & 26\,897 & 303\,291\\
    6 & 5.697--11.810       &    3600.00 & 13\,812/13\,596/0 & 32\,056 & 363\,476\\
    7 & 5.701--11.864       &    3600.00 & 16\,101/15\,851/0 & 37\,215 & 423\,661\\
    8 & 7.729--11.882       &    3600.00 & 18\,390/18\,106/0 & 42\,374 & 483\,846\\
    9 & 8.741--11.896       &    3600.00 & 20\,679/20\,361/0 & 47\,533 & 544\,031\\
   10 & 8.762--11.908       &    3600.00 & 22\,968/22\,616/0 & 52\,692 & 604\,216\\
    \bottomrule    
  \end{tabular}
  \caption{Results of the advanced MILP model for the benchmark
    problem with an $\hat{\epsilon} = -10^{-5} < 0$ (capturing all
    feasible and possibly some infeasible configurations); 
    number of variables/constraints/non-zeroes for original problem
    before preprocessing; the time limit was 1h = 3600s;
    MacBook Pro 2013, 2.6\,GHz Intel Core i7, 16\,GB 1600\,MHz DDR3
    RAM, OS~X~10.9.2, \texttt{zimpl} 3.3.1, \texttt{cplex} 12.5
    (Academic Initiative License), branching priorities given
    according to stage structure.}
  \label{tab:MILP-advanced-results_DB}
\end{table}
Table~\ref{tab:MILP-advanced-results_DB} shows the computational
results for our advanced upper-bound model, in particular, with an
$\hat{\epsilon} = -10^{-5} < 0$, i.e., the obtained configurations may
be infeasible, i.e., the listed objective function values may be
different than the results of the true BC-dynamics applied to the
computed controls.  However, the set-up guarantees that no feasible
configurations exist that produce better objective function values.
It is apparent that this time not even for $N=5$ the optimal value of
the upper-bound model could be found in 1h.  This can be explained:
Since in the upper-bound model a solution has more freedom to classify
``inside confidence interval: yes or no'', it is harder for the solver
to prove that high objective function values are impossible.  On a
faster computer\footnote{Mac Pro 2008 2$\times$2.8\,GHz Quad-Core
  Intel Xeon, 21\,GB 800\,MHz DDR2 RAM, OS~X~10.9.5, \texttt{zimpl}
  3.3.1, \texttt{cplex} 12.5 (Academic Initiative License} we obtained
for the upper-bound model optimal values of $6.676$ for $N=5$ in 2649s
and $6.700$ for $N=6$ in 117\,597s, respectively.  Thus, we know that
more than six voters are neither possible in five nor in six stages.

\begin{table}[h]
  \centering\footnotesize
    \begin{tabular}{rrr}
      \toprule
      \# stages & \multicolumn{2}{c}{objective of an optimal BC-control} \\
      &           after 1h CPU time & all we know from MILP\\
      \midrule
      1     &3.600                   &3.600         \\ 
      2     &4.615                   &4.615         \\
      3     &5.640                   &5.640         \\
      4     &6.653--\hphantom{1}6.657&6.653--\hphantom{1}6.657 \\
      5     &6.676--11.743           &\textcolor{blue}{6.676} \\
      6     &6.691--11.810           &6.691--\hphantom{1}\textcolor{blue}{6.700} \\
      7     &5.703--11.864           &\textcolor{blue}{6.707}--11.864 \\
      8     &8.725--11.882           &8.725--11.882 \\
      9     &7.746--11.896           &\textcolor{blue}{8.767}--11.896 \\ 
      10    &8.761--11.908           &\textcolor{blue}{11.800}--11.908 \\ 
      \bottomrule    
    \end{tabular}
    \caption{Summary of the knowledge that the advanced MILP model
      could collect for the benchmark problem: first, what we know
      after 1h time limit per computation and second, what we know at
      all (selected extra computations in \textcolor{blue}{blue};
      tightening the upper bound for $7 \le N \le 9$ would have been
      of paramount interest but was out of reach
      for us, thus no new figures there).
      \label{tab:MILP-advanced-results_summary}}
\end{table}
Table~\ref{tab:MILP-advanced-results_summary} summarizes the results
that the advanced model can generate in 1h per computation by
combining the results of the lower-bound and upper-bound models.  For
up to $3$ stages the optimal value of an optimal BC-control was
found. Regarding the number of achievable voters only (the integral
part of the objective function), we know that in $4$ stages six voters
are possible, but no more.  There is, however, a small gap between the
optimal values of the lower-bound and the upper-bound model.  For $5$
and more stages, the upper bound could not pushed below the trivial
bound $11$ in 1h computation time.

\begin{figure}[htp]
  \centering
  \includegraphics[width=0.20\linewidth]{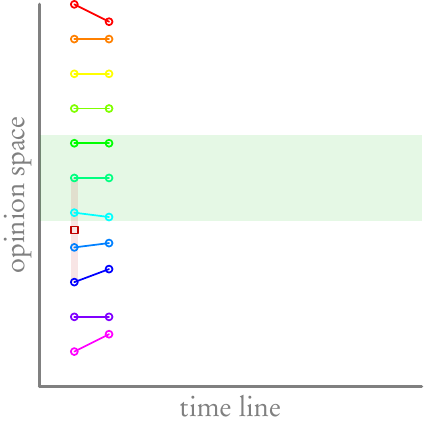}%
  \includegraphics[width=0.20\linewidth]{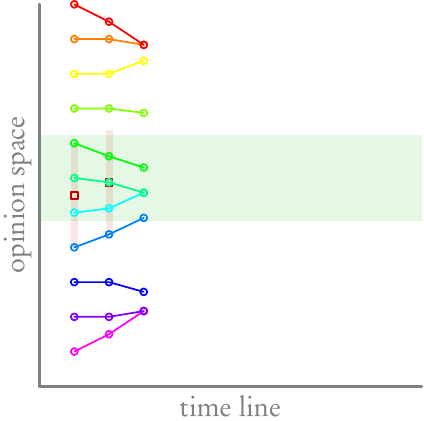}%
  \includegraphics[width=0.20\linewidth]{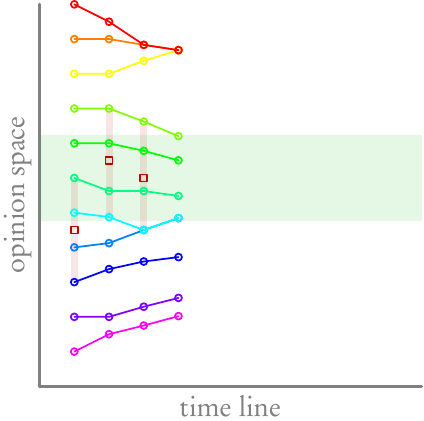}%
  \includegraphics[width=0.20\linewidth]{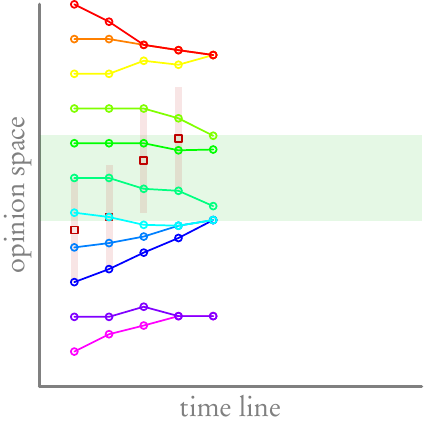}%
  \includegraphics[width=0.20\linewidth]{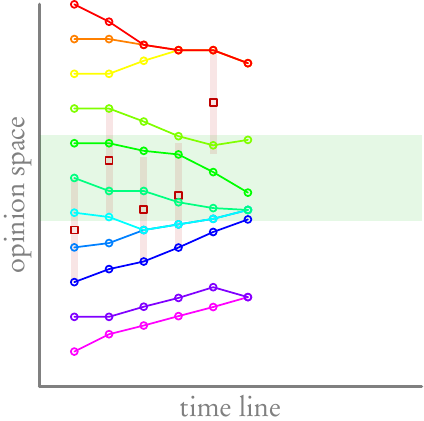}\\%
  \includegraphics[width=0.20\linewidth]{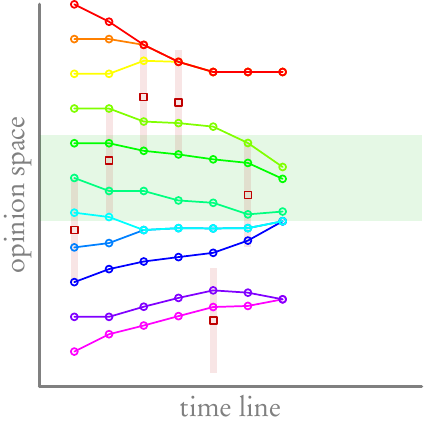}%
  \includegraphics[width=0.20\linewidth]{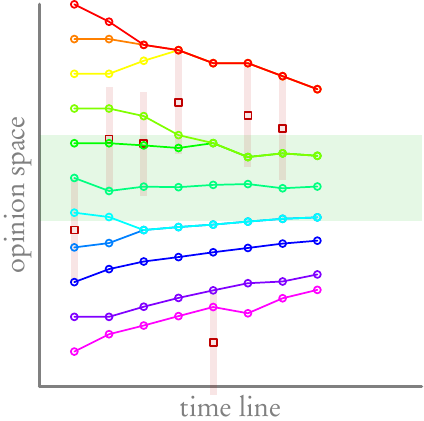}%
  \includegraphics[width=0.20\linewidth]{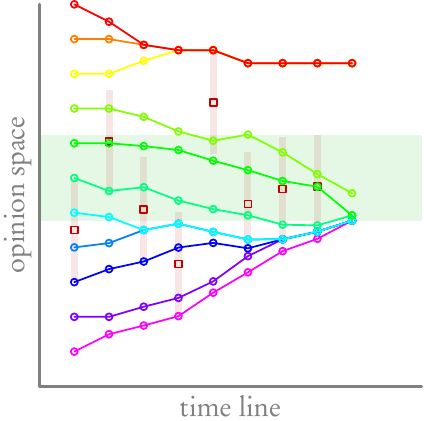}%
  \includegraphics[width=0.20\linewidth]{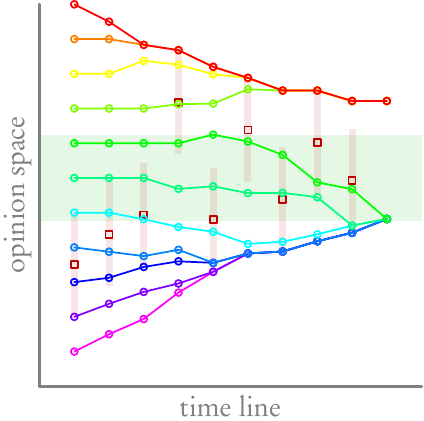}%
  \includegraphics[width=0.20\linewidth]{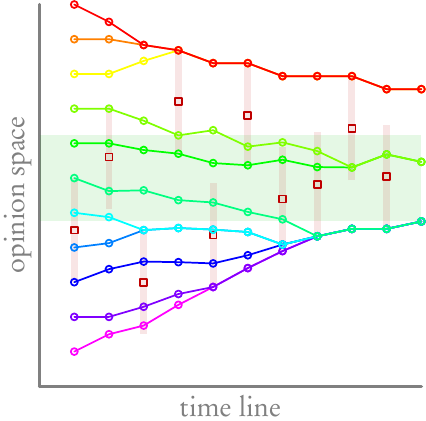}%
  \caption{Trajectories produced by solutions of the advanced
    lower-bound MILP for $\hat{\epsilon} = 10^{-5} > 0$ (provably
    feasible configurations); for $N \ge 6$ we show the incumbent
    solutions when the time limit of 1h was exceeded}
  \label{fig:MILP-advanced-trajectories}
\end{figure}

\begin{figure}[tp]
  \centering
  \includegraphics[width=0.20\linewidth]{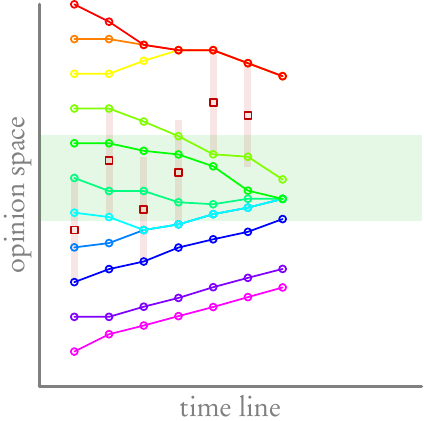}%
  \includegraphics[width=0.20\linewidth]{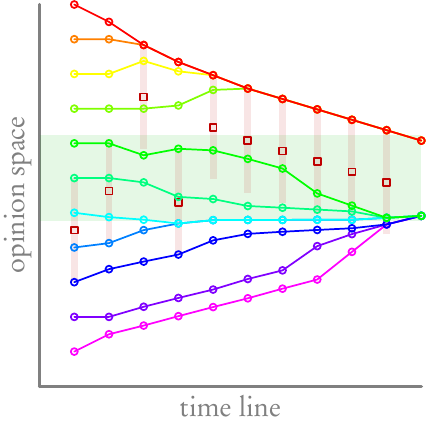}%
  \caption{Trajectories produced by solutions of the advanced
    lower-bound MILP for $\hat{\epsilon} = 10^{-5} > 0$ (provably
    feasible configurations) for 6 (optimal objective) and 10 (optimal
    number of convinced voters) stages in a time limit of 24h.}
  \label{fig:MILP-advanced-trajectories-extra}
\end{figure}

\begin{figure}[tp]
  \centering
  \includegraphics[width=0.20\linewidth]{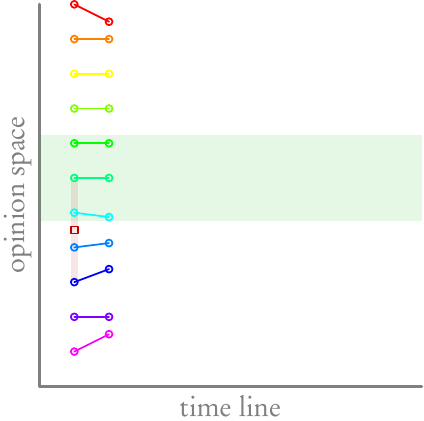}%
  \includegraphics[width=0.20\linewidth]{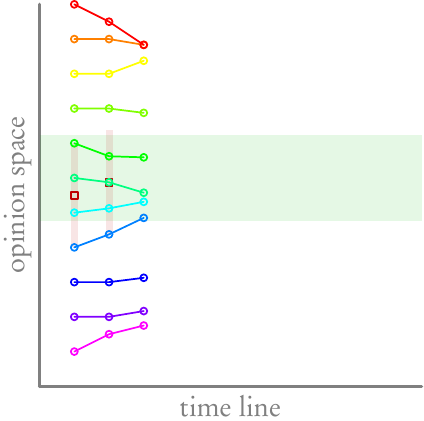}%
  \includegraphics[width=0.20\linewidth]{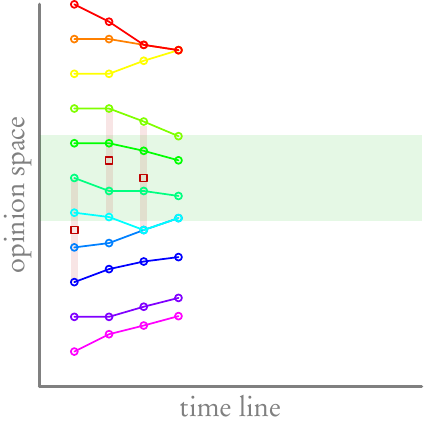}%
  \includegraphics[width=0.20\linewidth]{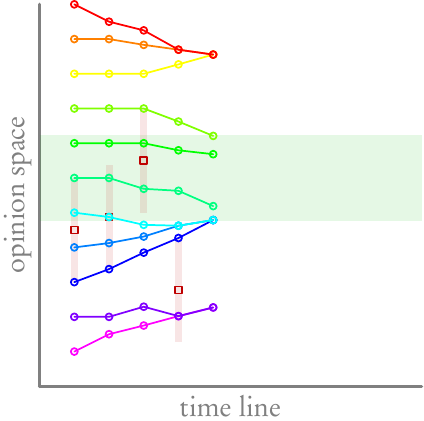}%
  \includegraphics[width=0.20\linewidth]{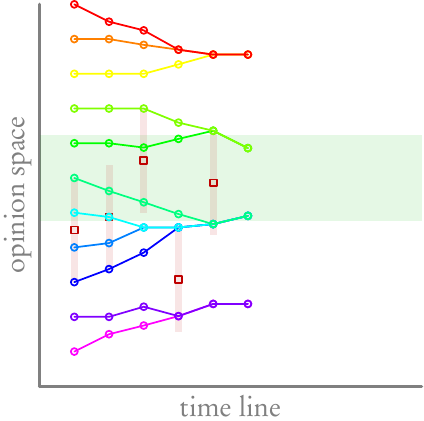}\\%
  \includegraphics[width=0.20\linewidth]{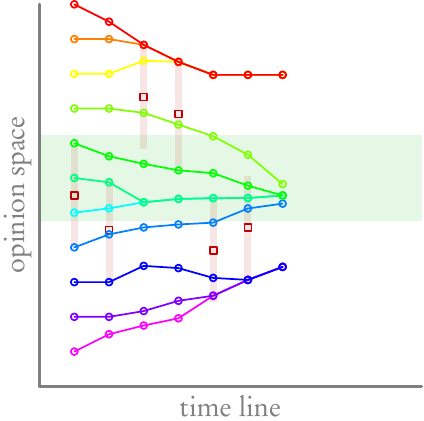}%
  \includegraphics[width=0.20\linewidth]{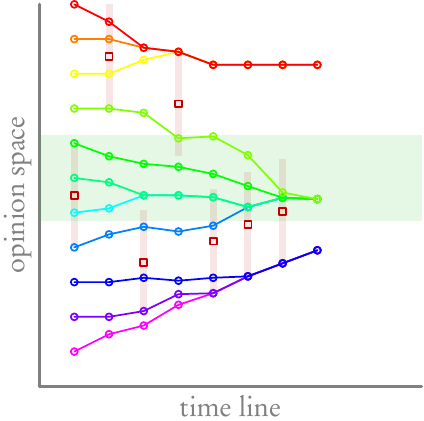}%
  \includegraphics[width=0.20\linewidth]{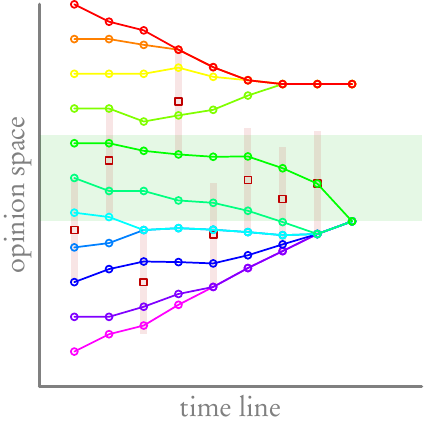}%
  \includegraphics[width=0.20\linewidth]{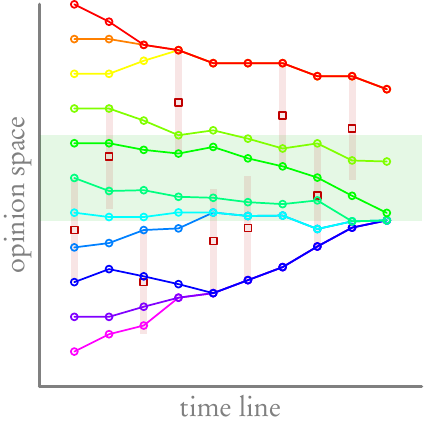}%
  \includegraphics[width=0.20\linewidth]{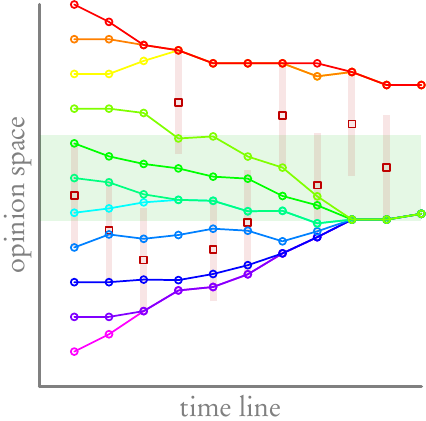}%
  \caption{Trajectories produced by solutions of the advanced
    upper-bound MILP for $\hat{\epsilon} = -10^{-5} < 0$ (capturing
    all feasible and possibly some infeasible configurations); for
    $N \ge 5$ we show the incumbent solution when the time limit of 1h
    was exceeded; note the obviously infeasibly splitting trajectories
    for $N = 10$ stages for voters $1$ and~$2$ in stage~$7$.}
  \label{fig:MILP-advanced-trajectories_DB}
\end{figure}

We see that MILP modeling requires a lot of effort.  We know, however,
of no other method to date that can prove global optimality of a
control for $N = 6$ or larger.

Figures~\ref{fig:MILP-advanced-trajectories}
and~\ref{fig:MILP-advanced-trajectories-extra} as well as
Figure~\ref{fig:MILP-advanced-trajectories_DB} visualize trajectories.
\revJ{As with the basic model, the optimal solution found by the MILP
  solver or, in the case of a timeout, the best feasible incumbent
  solution, determine, via the values of corresponding variables,
  control values and opinion values for all voters in all stages.  See
  Appendix~\ref{sec:MILP-Bounded-Confidence} for a detailed
  description of all variables.\footnote{Note that the voters' opinion
    values have \emph{not} been taken from the results of a
    BC-simulation on the basis of the control values; they have rather
    been taken from the MILP solution vector directly.  This way we
    can better see the differences between solutions to the lower- and
    the upper-bound model: for the indicated control, the former
    always produces trajectories that a BC-simulation would have
    produced, whereas the latter may produce trajectories that are
    impossible in a BC-simulation.}  Again, the numbers of convinced
  voters induced by the trajectories always correspond to the lower
  values in the optimal value column of Tab
  Table~\ref{tab:MILP-advanced-results_DB} for the upper-bound model.
  The upper values in those tables are produced by involved
  mathematical computations from MILP solving technology (e.g., formal
  model relaxations). Therefore, they do not correspond to any
  underlying trajectories.\footnote{Recall that in the summarizing
    table~\ref{tab:MILP-advanced-results_summary} the \emph{lower}
    values stem from \emph{lower} bounds (derived from feasible
    solutions) to the \emph{lower}-bound model and the \emph{upper}
    values stem from \emph{upper} bounds (derived by MILP machinery)
    for the \emph{upper}-bound model.  Thus, the numbers of convinced
    voters induced by the trajectories in
    Figure~\ref{fig:MILP-advanced-trajectories_DB} that stem from
    \emph{lower} bounds to the \emph{upper}-bound model do not appear
    at all in Table~\ref{tab:MILP-advanced-results_summary}.}}
Figures~\ref{fig:MILP-advanced-trajectories}
and~\ref{fig:MILP-advanced-trajectories-extra} show trajectories
induced by the provably feasible solutions of the advanced lower-bound
model with $\hat{\epsilon} = 10^{-5} > 0$.
Figure~\ref{fig:MILP-advanced-trajectories_DB} shows the (possibly
infeasible) solutions of the advanced upper-bound model for
$\hat{\epsilon} = -10^{-5} < 0$; infeasibility is best seen for 10
stages where trajectories split, which is impossible for feasible
trajectories.

\revJ{Splitting trajectories in the upper-bound model illustrate the
  fact that opinion trajectories that are feasible solutions for the
  upper-bound model\footnote{Recall, this is characterized by a
    \emph{negative} safety margin $\hat{\epsilon} < 0$, see
    Section~\ref{sec:MILP-models} for the idea
    and~\ref{sec:MILP-Bounded-Confidence} for the full details.}  are
  not necessarily feasible trajectories of opinions under the
  BC-dynamics.  Trajectories coming from the upper-bound model can,
  e.g., split: For identical opinions $i$ and~$j$ whose distance from
  opinion~$k$ is between $\epsilon + \hat{\epsilon} < \epsilon$ and
  $\epsilon$ the optimization algorithm may choose that $i$ is in the
  confidence interval of~$k$ but $j$ is not.  This can lead to
  different opinions of $i$ and~$j$ in the upcoming stage. This
  contrasts the situation in the lower-bound model.\footnote{Recall,
    this is characterized by a \emph{positive} safety margin
    $\hat{\epsilon} > 0$.}  The lower-bound model is designed in such
  a way that \emph{all} solutions with opinions at a distance between
  $\epsilon$ and $\epsilon + \hat{\epsilon} > \epsilon$ are infeasible
  for it.\revJ{\footnote{Recall, that this renders all `cutting-edge'
      solutions infeasible in the lower-bound model although in fact
      some of them may still be feasible trajectories in the
      BC-dynamics.}}}

What can we learn from the provably optimal solutions?  We see that
\begin{itemize}
\item some controls near the end of the horizon just decrease the
  average distance of non-convinced voters to the conviction
  interval, i.e., they are induced by the perturbation of the
  objective function
\item the controls of the earlier stages, however, always try to pull
  a voter as far as possible, i.e., they are exactly at confidence
  distance to some selected voter
\end{itemize}
The second observation will be exactly what is used in
Section~\ref{sec:Heuristics} as an idea for a clever heuristics.

\subsection{Computational Results on Random Instances}
\label{sec:comp-results-rand}

Of course, the choice of the particular benchmark problem is
completely arbitrary.  We want to provide additional evidence for the
observation that DG control is easier than BC control and that the
advanced model performs better than the basic one. To this end, we
tested all models on random instances.  We chose a uniform
distribution of start opinions in the unit interval.  The conviction
interval was constructed by choosing a \emph{party opinion~$o$} uniformly
at random in the unit interval and a \emph{conviction distance~$\Delta$}
uniformly at random between 0.1 and~0.2.  The conviction interval then
contained all values with a distance at most the conviction distance
from the party opinion, cut-off at zero and one, respectively.  The
confidence distance $\epsilon$ was chosen uniformly at random between
0.1 and~0.2 as well.  We drew five samples of random data and ran
\texttt{cplex} \cite{cplex:2014} with a time limit of one hour on the
resulting instances with one up to ten stages.  \revJ{See
Table~\ref{tab:SampledData} for the data realizations of the five
samples used.}

\revJ{
\begin{table}[h]
  \centering\tiny
  \begin{tabular}{rlllll}
    \toprule
    Sample & 1 & 2 & 3 & 4 & 5\\
    \midrule
    \multicolumn{6}{l}{Start Opinions}\\
    \midrule
    1
    &\textsf{0.0001143810805199624}&\textsf{0.02592622792020585}&\textsf{0.0707248803392809}&\textsf{0.1726953250292445}&\textsf{0.05518012076038404}\\
    2
    &\textsf{0.09233859556083068}&\textsf{0.1850820815155939}&\textsf{0.121328579290148}&\textsf{0.2160895006302953}&\textsf{0.08982103413199563}\\
    3
    &\textsf{0.1281244478021107}&\textsf{0.3205364363548663}&\textsf{0.2909047436646429}&\textsf{0.5472322533715591}&\textsf{0.2067191540744898}\\
    4
    &\textsf{0.146755892584742}&\textsf{0.3303348203958791}&\textsf{0.4370619401887669}&\textsf{0.5975562058383497}&\textsf{0.2219931714287943}\\
    5
    &\textsf{0.2360889763189687}&\textsf{0.4203678016598261}&\textsf{0.5108276010748994}&\textsf{0.609035598255935}&\textsf{0.3637368954633681}\\
    6
    &\textsf{0.3023325678199372}&\textsf{0.4353223932989226}&\textsf{0.5507979045041831}&\textsf{0.6977288244985344}&\textsf{0.4884111901019726}\\
    7
    &\textsf{0.417021998534217}&\textsf{0.4359949027271929}&\textsf{0.5693113258037044}&\textsf{0.7148159946582316}&\textsf{0.6117438617189749}\\
    8
    &\textsf{0.7203244894557456}&\textsf{0.4847490963257731}&\textsf{0.7081478223456413}&\textsf{0.8556209450717133}&\textsf{0.831327840180911}\\
    9
    &\textsf{0.9325573614175797}&\textsf{0.5496624760678183}&\textsf{0.8399490424990534}&\textsf{0.9006214549067946}&\textsf{0.8707323036786941}\\
    10
    &\textsf{0.9971848083653452}&\textsf{0.9315408638053435}&\textsf{0.8929469580512835}&\textsf{0.9670298385822749}&\textsf{0.9186109045796587}\\
    11
    &\textsf{0.9990405156274885}&\textsf{0.9477306110662712}&\textsf{0.8962930913307454}&\textsf{0.9726843540493129}&\textsf{0.9794449978460197}\\
    \midrule
    \multicolumn{6}{l}{Parameters}\\
    \midrule
    $\epsilon$
    &\textsf{0.1387910740307511}&\textsf{0.1698862689477127}&\textsf{0.1040630737561879}&\textsf{0.1224505926767482}&\textsf{0.1354138042860231}\\
    $o$
    &\textsf{0.3965807262334462}&\textsf{0.1544266755586552}&\textsf{0.01874801028025057}&\textsf{0.1414641728954073}&\textsf{0.3967366065822394}\\
    $\Delta$
    &\textsf{0.1186260211324845}&\textsf{0.1204648636096308}&\textsf{0.1040630737561879}&\textsf{0.1976274454960663}&\textsf{0.1765907860306536}\\
    \bottomrule
  \end{tabular}
  \caption{Random data, sampled with 16 significant decimal digits precision.}
  \label{tab:SampledData}
\end{table}
}

Table~\ref{tab:results-random} shows the results for an
$\hat{\epsilon} = 10^{-5} > 0$; we again skip the upper-bound computation
with $\hat{\epsilon} = -10^{-5} < 0$.  One very interesting phenomenon
can be identified in Sample~5: the number of convincible voters is
\emph{not} monotonically increasing with the number of stages
available for control.  There can be some unavoidable `distraction'
caused by other voters.  While in two stages we can achieve five
convinced voters, in three stages no more than three are possible.
This is some more evidence for the assessment that bounded-confidence
dynamics can lead to the emergence of various counter-intuitive
structures that make the dynamic system appear random and erratic
though it actually is deterministic.

\begin{table}[p]
  \centering\footnotesize
  \begin{tabular}{r@{\hspace*{10mm}}rr@{\hspace*{10mm}}rr@{\hspace*{10mm}}rr}
    \toprule
    \# stages & \multicolumn{2}{c@{\hspace*{10mm}}}{DG} &
    \multicolumn{2}{c@{\hspace*{10mm}}}{BC (basic)} &
    \multicolumn{2}{c}{BC (advanced)} \\
    & opt.~val. & CPU~[s] & opt.~val. & CPU~[s] &
    opt.~val. & CPU~[s]\\
    \midrule
    \multicolumn{7}{l}{Sample 1}\\
    \midrule
    1 & 11 & $<$0.01 & 2        &    0.01 & 2.48         & 0.03\\
    2 & 11 & $<$0.01 & 2        &    1.77 & 2.49         & 0.22\\
    3 & 11 & $<$0.01 & 2        &  190.56 & 2.51         & 3.62\\
    4 & 11 & 0.01    & 3--11    & 3600.00 & 3.49         & 114.40\\
    5 & 11 & 0.01    & 3--11    & 3600.00 & 3.51         & 1754.67\\
    6 & 11 & 0.01    & 1--11    & 3600.00 & 3.52--\hphantom{1}8.62  & 3600.00\\
    7 & 11 & 0.01    & 3--11    & 3600.00 & 3.53--11.60 & 3600.00\\
    8 & 11 & 0.01    & $-\infty$--11    & 3600.00 & 3.54--11.76 & 3600.00\\
    9 & 11 & 0.01    & $-\infty$--11    & 3600.00 & 3.55--11.78 & 3600.00\\
    10 &11 & 0.02    & $-\infty$--11    & 3600.00 & 2.56--11.81 & 3600.00\\
    \midrule
    \multicolumn{7}{l}{Sample 2}\\
    \midrule
    1 & 0 & $<$0.01 & 2         &    0.01 & 2.72         &    0.02\\
    2 & 0 & $<$0.01 & 2         &    3.30 & 2.73         &    0.18\\
    3 & 0 & $<$0.01 & 2         &  378.88 & 2.73         &    2.02\\
    4 & 0 & $<$0.01 & 2--11     & 3600.00 & 2.73         &   87.14\\
    5 & 0 & $<$0.01 & 1--11     & 3600.00 & 2.74         & 1558.18\\
    6 & 0 & $<$0.01 & 0--11     & 3600.00 & 2.75--\hphantom{1}9.79  & 3600.00\\
    7 & 0 & $<$0.01 & 1--11     & 3600.00 & 2.76--\hphantom{1}9.83  & 3600.00\\
    8 & 0 & $<$0.01 & $-\infty$--11     & 3600.00 & 2.76--11.85 & 3600.00\\
    9 & 0 & $<$0.01 & 0--11     & 3600.00 & 9.80--11.87 & 3600.00\\
    10 &0 & $<$0.01 & $-\infty$--11     & 3600.00 & 9.81--11.88 & 3600.00\\
    \midrule
    \multicolumn{7}{l}{Sample 3}\\
    \midrule
    1 & 0 & $<$0.01 & 2         &    0.01 & 2.53 & 0.01\\
    2 & 0 & $<$0.01 & 2         &    2.04 & 2.54 & 0.05\\
    3 & 0 & $<$0.01 & 3         &  233.37 & 3.54 & 0.30\\
    4 & 0 & $<$0.01 & 2--11     & 3600.00 & 3.54 & 2.30\\
    5 & 0 & $<$0.01 & 3--11     & 3600.00 & 3.55 & 53.41\\
    6 & 0 & $<$0.01 & 3--11     & 3600.00 & 3.55 & 985.32\\
    7 & 0 & $<$0.01 & 3--11     & 3600.00 & 3.56--\hphantom{1}3.61 & 3600.00 \\
    8 & 0 & $<$0.01 & $-\infty$--11     & 3600.00 & 3.56--\hphantom{1}8.69 & 3600.00 \\
    9 & 0 & $<$0.01 & 3--11     & 3600.00 & 3.56--\hphantom{1}8.67 & 3600.00 \\
    10 &0 & $<$0.01 & $-\infty$--11     & 3600.00 & 3.57--\hphantom{1}8.75 & 3600.00 \\
    \midrule
    \multicolumn{7}{l}{Sample 4}\\
    \midrule
    1 & 0 & $<$0.01 & 2         &    0.01 & 2.49 & 0.01\\
    2 & 0 & $<$0.01 & 2         &    1.67 & 2.49 & 0.05\\
    3 & 0 & $<$0.01 & 2         &   89.49 & 2.50 & 0.30\\
    4 & 0 & $<$0.01 & 2--11     & 3600.00 & 2.51 & 2.30\\
    5 & 0 & $<$0.01 & 2--11     & 3600.00 & 2.51 & 24.45\\
    6 & 0 & $<$0.01 & 2--11     & 3600.00 & 2.52 & 278.52\\
    7 & 0 & $<$0.01 & 2--11     & 3600.00 & 2.53 & 1281.73\\
    8 & 0 & $<$0.01 & 2--11     & 3600.00 & 2.54--\hphantom{1}7.59 & 3600.00\\
    9 & 0 & $<$0.01 & $-\infty$--11     & 3600.00 & 2.54--\hphantom{1}7.67 & 3600.00\\
    10 &0 & $<$0.01 & $-\infty$--11     & 3600.00 & 2.55--11.71 & 3600.00\\
    \midrule
    \multicolumn{7}{l}{Sample 5}\\
    \midrule
    1 &  8 & $<$0.01 & 3        &    0.01 & 3.63         &    0.02\\
    2 & 11 & $<$0.01 & 5        &    1.57 & 5.64         &    0.17\\
    3 & 11 & $<$0.01 & 3        &   89.57 & 3.66         &    1.51\\
    4 & 11 &    0.01 & 7--11    & 3600.00 & 7.67         &    5.39\\
    5 & 11 &    0.01 & 7--11    & 3600.00 & 7.69         &   36.32\\
    6 & 11 &    0.01 & 3--11    & 3600.00 & 7.70         &  456.57\\
    7 & 11 &    0.01 & 7--11    & 3600.00 & 7.72         & 2590.83\\
    8 & 11 &    0.01 & 7--11    & 3600.00 & 7.73--11.79  & 3600.00\\
    9 & 11 &    0.01 & 7--11    & 3600.00 & 7.74--11.83  & 3600.00\\
    10 &11 &    0.01 & $-\infty$--11    & 3600.00 & 7.75--11.86  & 3600.00\\
    \bottomrule   
  \end{tabular}
  \caption{Results of the lower-bound MILP models on random instances;
    the time limit was 1h = 3600.00s; 
    MacBook Pro 2013, 2.6\,GHz Intel Core i7, 16\,GB 1600\,MHz DDR3
    RAM, OS~X~10.9.2, \texttt{zimpl} 3.3.1, \texttt{cplex} 12.5
    (Academic Initiative License), branching priorities given
    according to stage structure.} 
  \label{tab:results-random}
\end{table}

%%%%%%%%%%%%%%%%%%%%%%%%%%%%%%%%%%%%%%%%%%%%%%%%%%%%%%%%%%%%%%%%%%%%%%%%%%%%%%
%% Section: Heuristics to Find Good Controls
%%%%%%%%%%%%%%%%%%%%%%%%%%%%%%%%%%%%%%%%%%%%%%%%%%%%%%%%%%%%%%%%%%%%%%%%%%%%%%
\section{Heuristics to Find Good Controls}
\label{sec:Heuristics}

In the previous section we have described a MILP-formulation of our
problem. So in principle one could solve every problem instance by
standard of-the-shelf software like \texttt{cplex}
\cite{cplex:2014}. In contrast to the DG-model, where we could solve
our benchmark problem for any number of stages between between $1$ and
$10$ without any difficulty, the instances from the BC-model are
harder. Using the MILP-formulation of the previous section we were
only able to determine the optimal control up to $6$ stages using
\texttt{cplex}.

The approach using a MILP-formulation has the great advantage that we
receive dual, i.e., upper, bounds for the optimal control. For the
lower-bound direction (find feasible controls) one can equally well
employ heuristics that can determine good controls more
efficiently. Note that also the MILP-approach benefits from good
feasible solutions, especially if they respect fixed variables in the
branch-and-bound search tree.\revJ{\footnote{A branch-and-bound
    search tree is built by conventional MILP solvers in order to
    organize the enumeration of integral variables.  In each node of
    that tree, some of the integral variables have been fixed to a
    value.  A heuristic is only useful in the branch-and-bound
    solution process if it can find solutions in all nodes of the
    branch-and-bound tree, i.e., solutions that are consistent with
    the variable fixings in that node.}} So in the next
subsections we give three heuristics to find good controls.

\subsection{The Strongest-Guy Heuristics}
What makes the problem hard, apart from the discontinuous dynamics and
numerical instabilities, is the fact, that the control $x^t_0$ is a
continuous variable in all stages~$t$. Thus, at first sight the
problem is not a finite one. Let us relax our problem a bit by
allowing only a finite number of possibilities for $x_0^t$
at any stage and have a closer look at the situation.

By placing a control $x_0^t$ at stage~$t$ some
voters are influenced by the control while others are not
influenced by the control.  We notice that the magnitude
of influence rises with the distance between the voters opinion
$x_i^t$ and the control $x_0^t$ as long as
their distance remains below $\epsilon$. So the idea is, even though
we are not knowing what we are doing, we will do it with full
strength. Let $c=\frac{l+r}{2}$ be the center of the
conviction interval $[l,r]$ \rev{and 
$$
  \mu(i,t)=\left\{
  \begin{array}{rcl}
  \max(x_i^t+\epsilon,1) & \text{if} & i\in I\text{ and }x_i^t\le c,\\
  \min(x_i^t-\epsilon,0) & \text{if} & i\in I\text{ and } x_i^t>c,\\
  c                      & \text{if} & i=0
  \end{array}
  \right.
$$
be a mapping from $\Big(I\cup\{0\}\Big)\times\mathbb{N}_{\ge 0}$ to $[0,1]$. Pulling 
an agent~$i\in I$ at stage~$t$ with full strength towards the center 
of the conviction interval means to place the control at $\mu(i,t)$. The special 
choice of the function for $i=0$ attracts the agents towards the center of the 
conviction interval and aims to avoid oversteering and oscillations.}
%% then the set of possible positions of
%% $x_0^t$ at stage~$t$ is given by
%% $$
%%   \left\{
%%   \begin{array}{rcl}
%%   \max(x_i^t+\epsilon,1) & \text{if} & x_i^t\le c,\\
%%   \min(x_i^t-\epsilon,0) & \text{if} & x_i^t>c\\
%%   \end{array}
%%   \right.
%% $$
%% for all $i\in I$. In order to avoid oscillations we
%%   additionally consider the control $x_0^t=c$, which we label by $0$.
We call this relaxation of the problem the strongest-guy heuristics.

\begin{table}[h]
  \centering\footnotesize
  \begin{tabular}{rrl}
    \toprule
    \# stages  & \# convinced voters & control sequence\\
    \midrule 0 &  3 & $[]$                     \\
    1 &  3 & $[4]$                             \\
    2 &  4 & $[4,4]$                           \\
    3 &  5 & $[3,8,0]$             \\
    4 &  5 & $[3,3,8,0]$ \\
    5 &  6 & $[3,3,8,10,3]$        \\
    6 &  6 & $[3,8,8,2,7,2]$       \\
    7 &  8 & $[3,8,1,0,1,1,1]$     \\
    8 &  8 & $[3,6,0,3,1,9,9,9]$   \\
    9 &  8 & $[3,8,6,0,9,9,9,9,9]$ \\ 
    10 & 11 & $[3,0,0,10,6,9,3,4,7,0]$\\
    \bottomrule
  \end{tabular}
  \caption{Results of the strongest-guy heuristics on the benchmark example.}
  \label{table_strongest_guy}
\end{table}

%% Actually we had used the positions $x_i+\epsilon-\delta$ and $x_i-\epsilon+\delta$ with a small number $\delta>0$
%% to avoid numerical instabilities.
%% 
%% \begin{table}[h]
%%   \centering\footnotesize
%%   \begin{tabular}{rrl}
%%     \toprule
%%     \# stages  & \# convinced voters & control sequence\\
%%     \midrule 0 &  3 & $[]$                  \\
%%     1 &  3 & $[4]$                 \\
%%     2 &  4 & $[4,4]$               \\
%%     3 &  5 & $[4,8,3]$             \\
%%     4 &  6 & $[3,3,8,6]$           \\
%%     5 &  6 & $[3,3,8,7,9]$         \\
%%     6 &  6 & $[3,7,3,8,3,3]$       \\
%%     7 &  8 & $[3,8,1,2,4,1,1]$     \\
%%     8 &  8 & $[4,5,10,9,9,1,1,1]$  \\
%%     9 &  8 & $[3,9,3,8,9,8,8,8,9]$ \\
%%     10 & 11 & $[3,11,4,6,9,8,8,8,6,1]$\\
%%     \bottomrule
%%   \end{tabular}
%%   \caption{Results of the strongest-guy heuristics on the benchmark example.}
%%   \label{table_strongest_guy}
%% \end{table}

Instead of giving the exact values of $x^t_0$ for all stages~$t$ we
can also give a sequence of \rev{indices~$[i_0, i_1, \dots, i_{N-1}]$} of those
voters that are pulled at maximum strength.  The number of such
sequences is finite.  Therefore, the controls arising this way can be
enumerated, in principle, \rev{see Algorithm~\ref{alg_strongest_guy}}.

%% For our benchmark problem, %% (a clever implementation of) 
%% enumeration could be done, \rev{see Algorithm~\ref{alg_stringest_guy}}. 

%% \rev{
\begin{algorithm}[htbp]
\label{alg_strongest_guy}
\SetKw{In}{in}
\SetKwInOut{Input}{Input}
\SetKwInOut{Output}{Output}
\SetKwData{InfluenceRadius}{$\epsilon$}
\SetKwData{Voters}{$I$}
\SetKwData{NumStages}{$N$}
\SetKwData{StartingOpinions}{$x_i^0$ for all $i\in I$}
\SetKwFunction{ConvincedVoters}{NumberOfConvincedVoters}
\Input{\InfluenceRadius, \Voters, \NumStages, \StartingOpinions}
\Output{best configuration for the strategic agent}
\DontPrintSemicolon
%\dontprintsemicolon
\BlankLine
$\mathit{champion}$ $\leftarrow$ $-1$\;
\ForEach{$\left[i_0,i_1\dots,i_{N-1}\right]$ \In $\Big(I\cup\{0\}\Big)^N$}
{
  $\kappa$ $\leftarrow$ \ConvincedVoters($\mu(i_0,0),\mu(i_1,1),\dots,\mu(i_{N-1},N-1)$)\;
  \If(\tcc*[f]{better solution found?}){$\kappa>\mathit{champion}$}
  {
    $\mathit{champion}$ $\leftarrow$ $\kappa$\;
    \For(\tcc*[f]{update best solution}){$j\leftarrow 0$ \KwTo $N-1$}
    {
      $i^\star_j$ $\leftarrow$ $i_j$\;
    }
  }
}
\Return{best solution}\;
\caption{Strongest-Guy Heuristics}
\end{algorithm}

%% \medskip
%% 
%% 
%% \noindent
%% \rev{Given $\epsilon$, $I$, $N$, and $x_i^0$ for all $i\in I$:\\
%% set $champion=-1$\\[-1mm]
%% loop over all $\left[i_0,i_1\dots,i_{N-1}\right]\in\Big(I\cup\{0\}\Big)^N$\\
%% \hspace*{5mm} determine the number of convinced voters $\kappa$ for the sequence of \\
%% \hspace*{5mm} controls $\mu(i_0,0),\mu(i_1,1),\dots,\mu(i_{N-1},N-1)$\\  
%% \hspace*{5mm} if $\kappa>champion$ then\\
%% \hspace*{10mm} $champion=\kappa$\\
%% \hspace*{10mm} $i^\star_j=i_j$ for all $0\le j\le N-1$\\
%% \hspace*{5mm} end if\\
%% output $champion$ and $i^\star_j$ for all $0\le j\le N-1$
%% }
%% 
%% \medskip

In Table~\ref{table_strongest_guy} we give for our benchmark problem the
maximum number of convinced voters that can be achieved by using the
strongest guy heuristics together with the corresponding
index-vector.\footnote{We have used an implementation with
exact arithmetic to avoid numerical inaccuracies.} The
corresponding trajectories are drawn in
Figure~\ref{fig:Strongest-guy-trajectories}. We observe
  that the strongest-guy heuristics improves the best found solution
  of the ILP approach, given 1h of computation time, for $N=7$ stages
  by two additional convinced voters.  For $N=4$ stages the heuristics
  misses the optimum of $6$ convinced voters by one. Given the upper
  bounds from Table~\ref{tab:MILP-advanced-results_summary} we can
  conclude that the strongest-guy heuristics found an optimal solution
  for $N\in\{0,1,2,3,5,6,10\}$.

\begin{figure}[htp]
  \centering
  \includegraphics[width=0.2\linewidth]{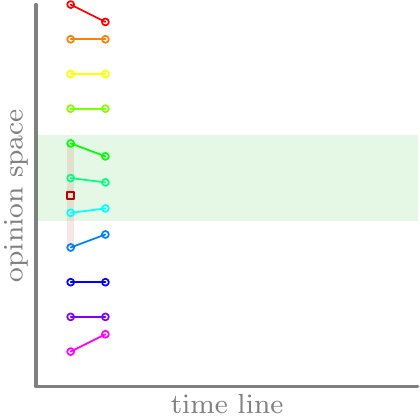}%
  \includegraphics[width=0.2\linewidth]{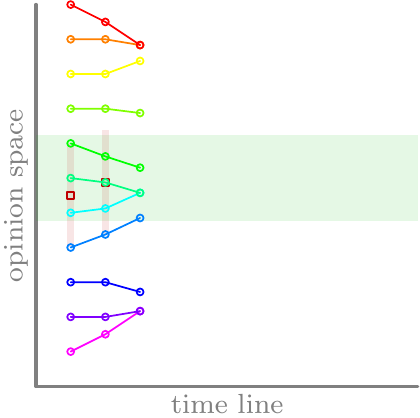}%
  \includegraphics[width=0.2\linewidth]{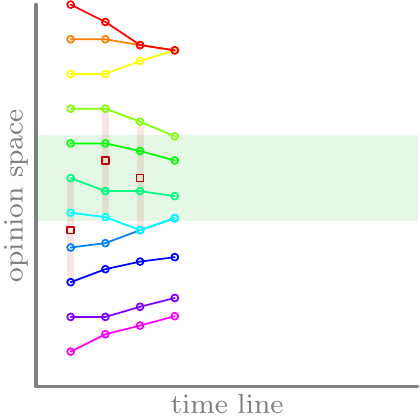}%
  \includegraphics[width=0.2\linewidth]{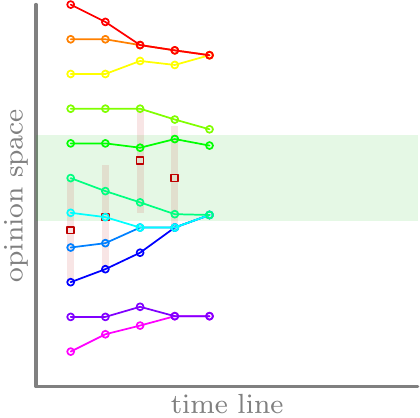}%
  \includegraphics[width=0.2\linewidth]{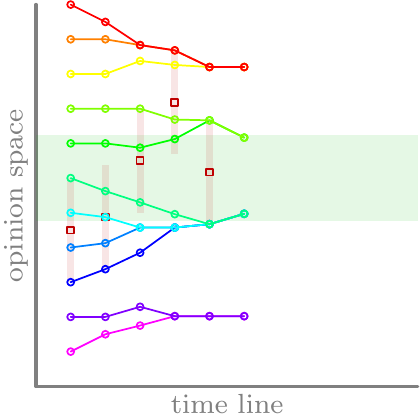}\\%
  \includegraphics[width=0.2\linewidth]{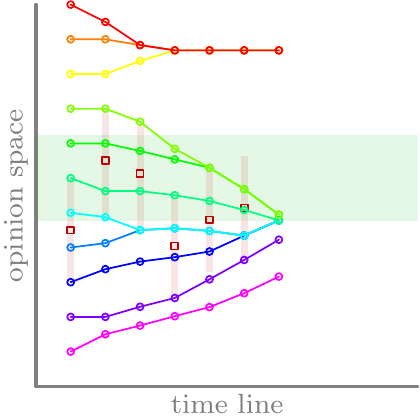}%
  \includegraphics[width=0.2\linewidth]{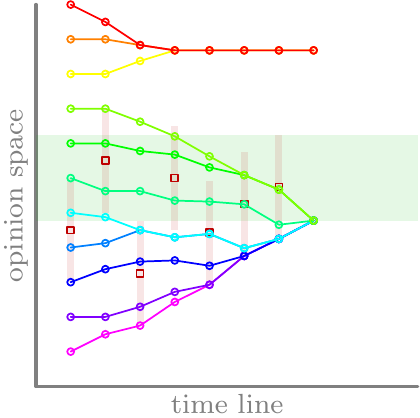}%
  \includegraphics[width=0.2\linewidth]{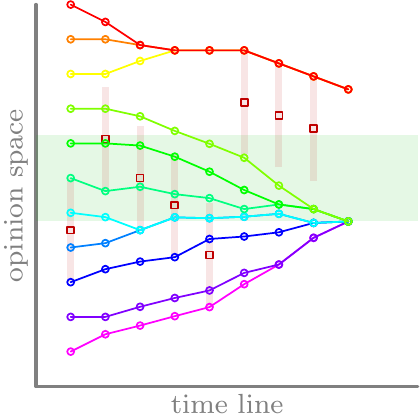}%
  \includegraphics[width=0.2\linewidth]{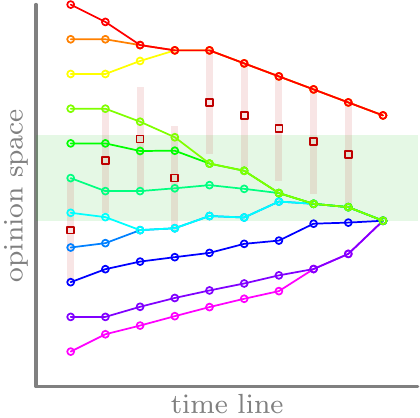}% 
  \includegraphics[width=0.2\linewidth]{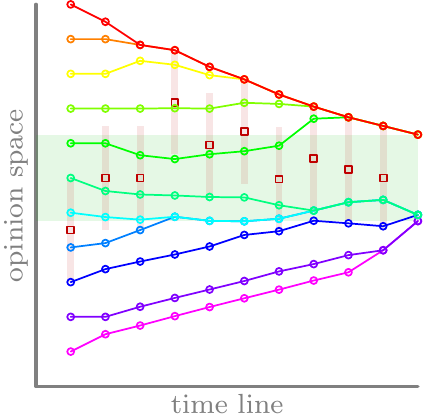}%
  \caption{Trajectories produced by solutions of the strongest-guy heuristics}
  \label{fig:Strongest-guy-trajectories}
\end{figure}

We can improve upon these findings by slightly modifying 
the strongest-guy heuristics. In order \rev{to improve the running time 
of our implementation, we had replaced the exact arithmetic by floating 
point numbers.} To avoid numerical \rev{instabilities we} had experimented with the 
positions $x_i^t+\epsilon-\delta$ and $x_i^t-\epsilon+\delta$ for $\delta=10^{-6}$, 
i.e.\ almost full strength. Curiously enough, we have obtained better solutions 
in some cases. For $N=4$ stages the $\delta$-modified control sequence 
$[3,3,8,6]$, corresponding to the control
$$
x_0=\left(\frac{349999}{1000000},\frac{309999}{800000},\frac{550001}{1000000},\frac{122599789}{200000000}
\right)\approx\left(
0.349999,0.387499,0.550001,0.612999
\right),
$$
results in six convinced voters. 
%% For $N=10$ stages the control sequence 
%% $[3,11,4,6,9,8,8,8,6,1]$ modified with $\delta=10^{-6}$, corresponding to the control
%% \begin{eqnarray*}
%% x_0&=&\Big(\frac{349999}{1000000},\frac{800001}{1000000},\frac{369999}{800000},\frac{8056649}{12800000},\frac{1054401961}{1600000000},
%% \frac{24872737789}{40000000000},\frac{23672745789}{40000000000},\\
%% &&\frac{22472753789}{40000000000},\frac{3170054515817321}{5832000000000000},\frac{432268514903721157}{857304000000000000}
%% \Big)\approx(0.349999,0.800001,\\
%% &&0.462499,0.629426,0.659001,0.621818,0.591819,0.561819,0.543562,0.504218),
%% \end{eqnarray*}
%% results in eleven convinced voters. This value is indeed optimal due to the trivial upper bound of 
%% $|I|=11$. 
The corresponding trajectory is drawn in Figure~\ref{fig:Modified-strongest-guy-trajectories}. 
We remark that we are not aware of any further improvements. 

\begin{figure}[htp]
  \centering
  \includegraphics[width=0.4\linewidth]{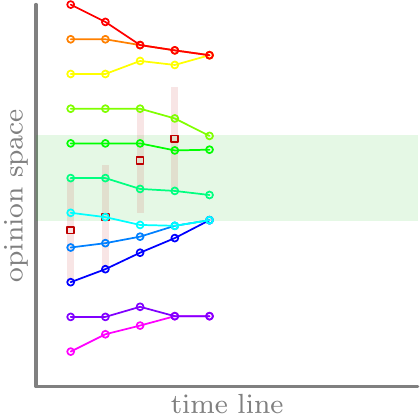}
  \caption{Trajectories produced by solution of the modified strongest-guy heuristics for $N=4$}
  \label{fig:Modified-strongest-guy-trajectories}
\end{figure}

%Although being a heuristic, we can conclude that if the
%number of stages is at least $10$ then the maximum number of convinced
%voters is $11$. Clearly, this is only possible due to the trivial
%upper bound of $11=|I|$. The values of Table \ref{table_strongest_guy}
%are proven to be optimal for up to $5$ stages by using the MILP
%approach from the previous section. We remark that we are not aware of
%any improvements to Table \ref{table_strongest_guy}.

Having a look at the optimal controls of Table
\ref{table_strongest_guy} one gets an impression of the hardness of
our problem. There seems to be no obvious pattern in the optimal
controls. Who would have guessed a control like
$[3,0,0,10,6,9,3,4,7,0]$ or $[3,8,8,2,7,2]$? 

%We remark that starting with $x_1$ or $x_2$
%at the first stage does not lead to $11$ convinced voters at the
%end. I.e. starting with $x_2$ and continuing with $x_{11}$ does lead
%to only $7$ convinced voters using the strongest-guy heuristic.

We remark that the strongest-guy heuristics can be easily adopted to
the situation where some of the $0$-$1$ variables from the MILP
formulation (see Appendix~\ref{sec:Mathematical-Model}) are fixed to
either $0$ or~$1$.  Thus, it is possible to install a call-back to the
strongest-guy heuristics inside an MILP solution process.
However, in our experience the MILP solution needs most of the time to
tighten the dual bound.

%\begin{figure}[htp]
%  \begin{center}
%    \includegraphics[width=10cm]{genAlgorithmus3.jpg}
%    \caption{A control yielding $8$ convinced voters.}
%    \label{fig_geneticalgoritm_1}
%  \end
%\end{figure}

\subsection{A Genetic Algorithm}
\label{sub:genetic_algorithm}
To get a better idea about the solution space, we implemented a
genetic algorithm (GA) to \rev{heuristically} search for optimal solutions. 
%% The GA used standard GA methods to evolve good solutions from a randomly chosen
%% set of starting strategies. To use GA for \rev{the problem,} the problem
%% has to be formulated to suit the GA terminology. A GA instance
%% consists of \emph{genes} that form a \emph{chromosome}. Each
%% chromosome can be evaluated using a \emph{fitness function}, which
%% serves to determine the chromosome's quality with respect to the
%% original problem. The GA uses alternating steps of \emph{evolution}
%% and \emph{selection} to modify the chromosomes and moving the entire
%% \emph{population} of chromosomes to increase the number of high
%% quality chromosomes. By means of the \emph{survival of the fittest},
%% the selection process sorts out weak chromosomes with low fitness
%% values, while it retains chromosomes with high fitness values. The
%% remaining chromosomes evolve to the next round of the GA. The
%% following subsection explain the different parts of the GA
%% individually.

\rev{
GAs work on the analogy of biological evolution and the Darwinian principle of \emph{survival 
of the fittest}. They test many sets of parameters automatically to find good solutions to the 
problem. GAs work based on rounds (so called \emph{generations}). Parameter sets resulting in 
good solutions advance to the next generation, bad ones are discarded. By modifying the good 
parameter sets before testing them again, new sets are created and tested. Over the course of 
many rounds the quality of the solutions gradually increases. Eventually one hopes to find a 
very good solution without having to brute-force all possible parameter combinations.
}

\rev{
Because a new parameter set is created by modifying existing sets that are already showing useful 
results, one of the main assumptions of GA is the idea that the solution space is somewhat smooth 
and that the optimal solution is found close to an already good solution. If the solution space is 
heavily fragmented and good solutions are just slightly removed from very bad solutions GAs tend not 
to work. During the execution of a GA this would manifest in a non-increasing overall solution quality.
}

\subsubsection{Setup of the GA}
%% \subsubsection{Set-Up of the GA}
\label{subsub:ga_setup}
%% The only free variables in our example problem are the ten different
%% positions of our freely selectable voter: The ten different positions
%% make up one strategy to control the remaining voters. In the GA, one
%% strategy is encoded as one chromosome with the ten different positions
%% each occupying one gene. Because of the numerical inaccuracies (see
%% Section \ref{sec:Numerics}) in standard floating point
%% implementations, the GA had to use fractions to specify the
%% strategies. Throughout the GA exact arithmetics had to be used, which,
%% since the GA has been implemented in Java, made use of the
%% \texttt{java.math.BigInteger} class.
%% 
%% The GA itself has been set up to run for 250 rounds or until an
%% optimal control had been reached, whichever would occur first. The
%% size of the population has been set to 500 voters. Both values could
%% be increased to get a higher chance to reach an optimal
%% solution. However, the computational cost of using exact arithmetics
%% is quite high, which led to this acceptable compromise.
\rev{
To find a solution for the voting problem, the only parameters that can be influenced are the 
position of the strategic agent for each of the discrete time events. In GAs each parameter set is 
called a \emph{chromosome}, a single parameter is a \emph{gene}. In the voter's game, a chromosome 
consists of ten genes, which each encode one position of the strategic agent in the opinion space. 
}

\rev{
To prevent any rounding problems inherent to the \texttt{float} and \texttt{double} data types in Java (see
Section \ref{sec:Numerics}), we chose to use fractions instead. The numerator and the denominator are represented 
as Java's \texttt{BigInteger}. Thus, the calculation is exact and does not suffer from the usual rounding problems 
that Java encounters with the primitive data types.\revJ{\footnote{See
Section~\ref{sec:Numerics} for the relevance of this.}} 
}

\rev{
One of the most popular packages to implement GAs in Java is \emph{JGAP -- Java Genetic Algorithm 
Package}\footnote{\url{http://jgap.sf.net}}. It provides a framework for the GA that must be extended 
with the problem specific fitness function and the chromosome implementation. Then it can be used to 
execute selection, cross over, mutation and evolution on the chromosomes. We used JGAP version 3.3.3.
}

\rev{
Each run of the GA start with a \emph{population} of 500 randomly generated chromosomes. Furthermore, 
we had to determine how many generations the GA should evolve. We set this value to 250 rounds. After 
250 rounds, the overall quality of the solutions did not increase, but reached a plateau from which the 
GA did not find any better solutions. 
}

\subsubsection{Chromosome And Fitness Function}
%% \subsubsection{Fitness Function}
\label{subsub:ga_fitness_function}
%% In the example we used the conviction interval of $[0.375, 0.625]$,
%% the most obvious fitness function to use is to count all the convinced
%% voters. This fitness function is denoted as \emph{MaxVotes
%%   (MV)}. Figure \ref{fig:ga_mv} depicts a typical run of the GA with
%% the MV fitness function.
\rev{
Each chromosome is a list of the strategic agents position in that particular configuration. To assess the 
quality of the chromosome, we have to calculate how many of the voting agents are convinced by the configuration 
after ten rounds of the game. Analogous to the biological \emph{survival of the fittest} GAs evaluate the quality 
of a solution with a \emph{fitness function}. This function is highly specific to the problem.
}
 
\rev{
To assess the fitness of each chromosome, we started out with the simplest possible fitness function: At the end
of the game we counted how many voters were actually convinced. Whether or not a voting agent had been convinced was 
determined by the party that the agent was closest to. Because we tried to convince voting agents to vote for party 2, 
any voting agent that was closer to party 2 than to any other party was considered a voter for that party. The parties 
have fixed positions in the opinion space. They are located at positions ${0.25, 0.5, 0.75}$. Effectively this results 
in the conviction interval of $[0.375, 0.625]$. If a voting agent ended up in that interval, it was considered a voter. 
}

\begin{figure}[htp]
  \begin{center}
    \includegraphics[width=10cm]{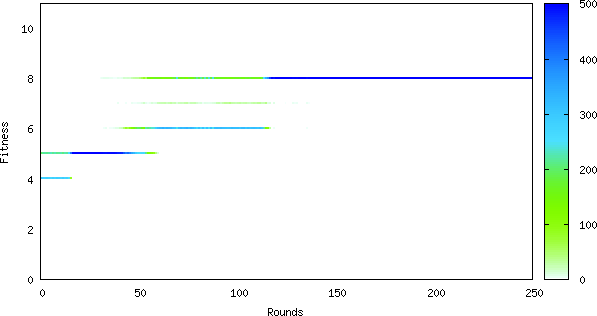}
    \caption{GA with the \emph{MaxVotes} fitness function yielding at most eight voters}
    \label{fig:ga_mv}
  \end{center}
\end{figure}

\rev{
This fitness function (called \emph{MaxVoter} (MV)), see Figure \ref{fig:ga_mv} for a typical run of the GA with
the MV fitness function, suffers several drawbacks. The main one is that it is a discrete function: A voting agent 
is either convinced or it is not convinced. If there are two chromosomes that both evaluate to \emph{nine} voters 
one cannot tell whether one of the chromosomes is closer to \emph{ten} voters than the other. For instance, a chromosome 
$A$ could yield the positions 
\begin{equation*}
  A = \{0, 0.38, 0.39, 0.4, 0.45, 0.5, 0.55, 0.6, 0.61, 0.62, 1\},
\end{equation*}
chromosome B could yield 
\begin{equation*}
  B = \{0.37, 0.38, 0.39, 0.4, 0.45, 0.5, 0.55, 0.6, 0.61, 0.62, 0.63\}.
\end{equation*}
Both chromosomes result in nine convinced voters. 
However, chromosome~$B$ has the two non-convinced voters (at positions $0.37$ and $0.63$) much closer to the conviction 
interval than chromosome~$A$ (non-convinced voters at position $0$ and $1$). Probably chromosome~$B$ is closer to an optimal 
solution than chromosome~$A$. Therefore it should have a higher fitness value. To account for these shortcomings, we used 
eight different fitness functions with varying results. 
}

%% While the MV fitness function provides an exact mapping of the
%% original problem to the formulation of the GA's fitness function, it
%% does not provide a very good example of a fitness function for a GA
%% because it is discontinuous and has huge gaps in its range. Because
%% the fitness function is just a count of the convinced voters, the
%% function can take integer values only. However, the GA's fitness
%% function is defined on real values. As a consequence, the fitness
%% function does not provide enough information about a particular
%% strategy. In the figure, one can see that, approximately after round
%% 120, almost all remaining strategies yield eight voters, but the GA
%% does not progress further to yield more voters. This indicates a lack
%% of information on the strategies. If all chromosomes evaluate to the
%% same fitness value, the GA has no way to select the fittest
%% chromosomes and advance them to the next generation.
%% 
%% Therefore, we came up with some different fitness functions that span
%% the entire range and try to provide additional information about
%% strategies that yield the same amount of voters. These strategies,
%% however, deviate from the original problem and create a slightly
%% different problem for the GA to solve. Thus, all fitness functions
%% that do not map the original problem directly, have to be evaluated
%% with respect to their mapping ability.

The fitness functions fall into three different categories:

\begin{description}
\item[Weighted Sum] This category of fitness functions calculates the
  weighted sum of all the voters final positions. The weight to be
  used is computed with a given partially defined function that maps a
  position to a weight. The MV fitness function is a special case of
  the Weighted Sum class of functions, because it assigns weight $1$
  to all positions in the conviction interval and weight $0$
  otherwise. The other functions used in this class are
  \emph{DistanceToParty2 (D2P2)} and \emph{BorderDistanceToAll
    (BD2A)}. Both of them differ from MV in that they assign values
  between $0$ and $1$ to positions that (a) are not in the conviction
  interval with higher values the closer the position is to the
  interval or (b) decreasing values within the interval, the closer
  the position is to the center of the interval. This leads to
  positions of voters on the very edges of the interval to be the most
  favorable. The idea behind evaluating positions within the interval
  differently is that voters sitting on the edges of the interval have
  the greatest effect on voters that are not in the range yet.
\item[Last Remaining] The class of the last remaining fitness
  functions does not evaluate every voter but restricts itself to
  convinced voters (counted with weight $1$) and the nearest voter
  that has not reached the conviction interval yet (weighted according
  to the function). All other voters are assigned weight $0$. The
  fitness function in this category is \emph{BorderDistanceToMin
    (BD2M)}, which has the same form as BD2A from the \emph{Weighted
    Sum} category.
\item[Minimum Distance] The last class of fitness functions does not
  evaluate all voters positions but takes into account the distance
  between the two outmost voters (i.\,e. the voter with the highest
  opinion and the voter with the lowest opinion). Because of the order
  preserving characteristics of the model, these two voters do not
  change throughout one run, which means, we can just use the distance
  between the voter that started with opinion $0$ and the one that
  started with opinion $1$. If the distance is in the range of $[0,
  0.25]$, which is the size of the conviction interval, the fitness
  evaluates to the maximum fitness value of $10$. If the distance is
  greater than $0.25$, the function computes a value that is
  decreasing to $0$ with increasing distance. The two functions used
  in this class are \emph{MinimumDistanceBetweenFirstAndLast (MDBFL)}
  and \emph{MinimumDistanceBetweenFirstAndLastSquare (MDBFLS)} that
  have a linear or quadratic slope respectively. A third function in
  this class is the
  \emph{MinimumDistanceBetweenFirstAndLastToCenterSquare (MDBFL2CS)}
  that accounts for the fact that the group of voters with opinions
  below $0.5$ may not behave symmetrically to the group with opinions
  above $0.5$. This asymmetry can result in the position range not to
  be centered around $0.5$ but deviate from that midpoint. Such
  behavior is undesirable, since the original goal is to get as many
  voters close to $0.5$ as possible. MDBFL2CS accounts for this and
  evaluates the positions of the two outmost voters with respect to
  their distance to the desired midpoint. The two values are added.
\end{description}

\rev{
\subsubsection{Selection}
At the end of each round, the fitness value for each chromosome is calculated. Based on the fitness value, 
the core idea of the GA takes place: Those chromosomes with a high fitness value have a larger chance to ``survive'' 
and to advance into the next round. Specifically, the GA framework provides two different selection routines. One 
selection (\emph{weighted roulette selector}, WRS) assigns each chromosome a probability of survival based on its 
fitness value. The population of chromosomes for the next generation is picked according to the probabilities. The 
other selector (\emph{Best Chromosomes Selector}, BCS) discards the chromosomes with the lowest fitness values and 
lets the other chromosomes survive. This selector converges quicker than the WRS. However, its disadvantage is the 
possibility that it gets stuck in a local optimum, which is not the global optimum.
}

\rev{
\subsubsection{Cross Over And Mutation}
Finally, the new population is crossed over and mutated. The \emph{Cross over} operator picks two chromosomes of 
the new population at random. Then it picks a random number $n$ between 1 and the number of genes in the chromosome 
(i.\,e. in the voter's game a number between 1 and 10). It splits both chromosomes at the $n$-th gene, resulting in 
four halves: chromosome 1, part 1 (C1.1), chromosome 1, part 2 (C1.2), chromosome 2, part 1 (C2.1) and chromosome 2, 
part 2 (C2.2). It ``crosses over'' the second parts of the chromosomes, resulting in two new chromosomes, made up of 
C1.1 and C2.2 (chromosome A) and C1.2 and C2.1 (chromosome B). Table \ref{tab:crossover} shows an example for a cross 
over between two chromosomes.
}

\rev{
\begin{table}
\centering
\begin{tabular}{|c|c|c|c|c|}
%\hline
\multicolumn{2}{c}{original chromosomes} & \multicolumn{1}{c}{} & \multicolumn{2}{c}{new chromosomes} \\
%\hline
\cline{1-2}\cline{4-5}
\texttt{1 2 3 4 5 6} &\texttt{7 8 9 0}  & $\rightarrow$ & \texttt{1 2 3 4 5 6} & \texttt{G H I J} \\
\cline{1-2}\cline{4-5}
\texttt{A B C D E F} & \texttt{G H I J}  & $\rightarrow$ & \texttt{A B C D E F} & \texttt{7 8 9 0} \\
\cline{1-2}\cline{4-5}
\end{tabular}
\caption{Example for a cross over at gene 7. Letters and numbers depict individual genes.}
\label{tab:crossover}
\end{table}
}

\rev{
The rate how often cross over happens depends on the population size. It is $r = \frac{\mbox{population size}}{2}$. 
Every time the operator is invoked cross over happens $\frac{\mbox{population size}}{r}$. Thus, with a population of 
500, cross over happens two times per generation.
}

\rev{
The \emph{mutation} operator picks one gene of a random chromosome. It replaces the gene with a random value. The mutation 
rate was fixed at $\frac{1}{15}$, resulting in 33.33 mutation per generation.
}

\subsubsection{Evolution}
\label{subsub:ga_evolution}
\rev{
After each round the GA determines the fitness of each chromosome. Then it selects the fittest chromosomes 
according to the selector and puts them into the next population. Once the candidate set has been established, 
the cross over and mutation happen on the candidate set. After both genetic operators have finished, the GA has 
established the population for the next generation of chromosomes and starts to evaluate them again. This happens 
until either a satisfactory fitness value has been reached by one chromosome (ie. 11 voters for the voter game) 
or until the maximum number of generations has been reached (ie. 250). Finally the GA outputs the best solution. 
}
%% After the evaluation phase of the GA, the fittest chromosomes are
%%chosen to advance to the next generation. 
There are different
selection algorithms available. We used two different ones, which are
among the standard selection algorithms:

\begin{description}
\item[Weighted Roulette Selector (WRS)] Each chromosome is assigned a
  probability to advance to the next round proportional to its
  fitness. Then the population for the next round is chosen by
  randomly picking a chromosome from the so called `roulette wheel'
  as often as desired. This selection method allows for some
  chromosomes with low fitness values to advance to the next round,
  which results in a lower chance to reach a local optimum too
  quickly.
\item[Best Chromosomes Selector (BCS)] The BCS sorts the population
  according to the fitness values and discards the fraction with the
  lowest fitness values. The ration of chromosomes to retain is
  configurable. BCS fosters depth search with the danger of reaching a
  local optimum. As an advantage, it progresses much quicker than WRS.
\end{description}

After the chromosomes for the next round have been selected, the GA
performs the crossover and mutation operations, whose parameters
(percentage of the mutation, point of crossover) are configurable.

\rev{
\subsubsection{Pseudo-Code}
To provide an explanation that is closer to the actual code, the
pseudo code for the core GA routines looks like this: }

\rev{
\begin{algorithm}[htbp]
\SetKw{In}{in}
\SetKwInOut{Input}{Input}
\SetKwInOut{Output}{Output}
\SetKwData{PopSize}{populationSize}
\SetKwData{MaxRounds}{maxRounds}
\SetKwData{Chromosome}{chromosome}
\SetKwData{Population}{population}
\SetKwData{PopNew}{newPopulation}
\SetKwFunction{Selector}{selector}
\SetKwFunction{CrossOver}{crossOver}
\SetKwFunction{Mutate}{Mutate}
\SetKwFunction{SelectionOperator}{selectionOperator}
\SetKwFunction{InitPop}{InitializeRandomPopulation}
\SetKwFunction{Fitness}{EvaluateFitness}
\SetKwFunction{Selection}{Selection}
\Input{\PopSize, \MaxRounds, \Selector}
\Output{best configuration for the strategic agent}
\DontPrintSemicolon
%\dontprintsemicolon
\BlankLine
%\PopSize $\leftarrow$ 500\;
%\MaxRounds $\leftarrow$ 250\;
%\Selector $\leftarrow$ Weighted Roulette Selector or Best Chromosomes Selector\;
\Population $\leftarrow$ \InitPop{\PopSize}\;
\BlankLine
\For(\tcc*[f]{fitness evaluation}){$i\leftarrow 1$ \KwTo \MaxRounds}{
	\ForEach{\Chromosome \In \Population}{
		\Fitness(\Chromosome)\;
	}
	\If(\tcc*[f]{stop evolution}){maximum fitness has been reached}{
		\Return{best solution}\;
	}
	\BlankLine
	\PopNew $\leftarrow$ \Selector{\Population}\tcc*[r]{selection}
	\BlankLine
	\For(\tcc*[f]{cross over}){$j\leftarrow 1$ \KwTo $\frac{\mbox{\PopSize}\times 2}{1\times \mbox{\PopSize}}$}{
		\Chromosome 1 $\leftarrow$ random \Chromosome from \PopNew\;
		\Chromosome 2 $\leftarrow$ random \Chromosome from \PopNew\;
		remove \Chromosome 1 and \Chromosome 2 from \PopNew\;
		\CrossOver{\Chromosome 1, \Chromosome 2}\;
		\PopNew $\leftarrow$ add modified \Chromosome 1 and \Chromosome 2\;
	}
	\BlankLine
	\For(\tcc*[f]{mutation}){$j\leftarrow 1$ \KwTo $\frac{1}{15}$ \PopSize}{
		\Chromosome $\leftarrow$ select random \Chromosome from \PopNew\;
		remove \Chromosome from \PopNew\;
		\Mutate{random gene in \Chromosome}\;
		\PopNew $\leftarrow$ add mutated \Chromosome\;
	}
	\Population$\leftarrow$ \PopNew\;
}
\Return{best solution}\;
\caption{Genetic Algorithm}
\end{algorithm}
}

\subsubsection{Results}
\label{subsub:ga_results}
Figure \ref{fig:ga_strategy_comparison} shows the performance of the
different fitness functions. One notable observation is the step like
behavior of the MV fitness function: Already around round 30, it
reaches eight voters, but does not advance from there. All the other
fitness functions show a much smoother behavior. However, with the
exception of BD2A and MDBFLS, all functions seem to have reached a
plateau around round 150.

\begin{figure}[htp]
  \begin{center}
    \includegraphics[width=10cm]{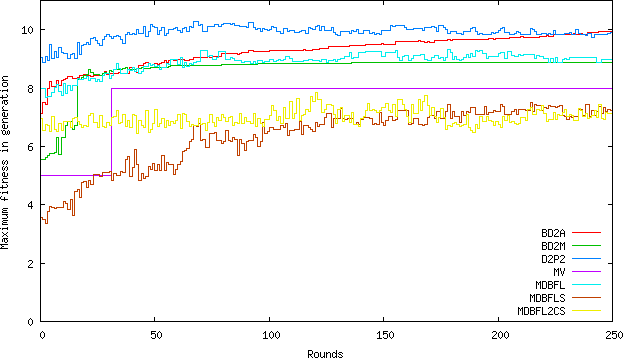}
    \caption{Different fitness functions reach different results}
    \label{fig:ga_strategy_comparison}
  \end{center}
\end{figure}

Judged from the performance of the fitness functions, BD2A promises
the best results as it still progresses at round 250 and also reached
a high fitness value. However, as BD2A optimizes a slightly different
problem than the original problem. Therefore, the performance with
respect to the fitness function has to be compared with the
performance of the fitness function with respect to the voters.

\begin{figure}[htp]
  \begin{center}
    \includegraphics[width=10cm]{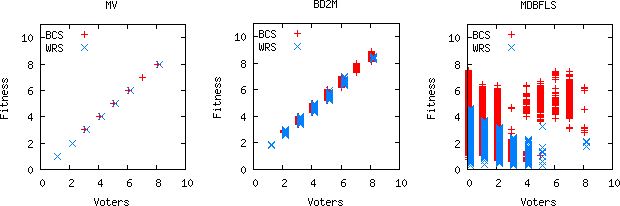}
    \caption{Performance of three fitness functions with respect to the convinced voters}
    \label{fig:ga_fitness_quality}
  \end{center}
\end{figure}

Figure \ref{fig:ga_fitness_quality} depicts the performance of the
fitness functions MV, BD2M, and MDBFLS. While MV maps the original
problem exactly, only BD2M provides a good mapping. MDBFLS (and all
other fitness functions alike) does not provide a good mapping of a
fitness value to a certain number of voters. This, of course, poses a
problem, since the only way for the GA of evaluating a certain
strategy is the fitness function. The graph suggest that, apart from
MV, only BD2M should be used.

Of the two different selectors available, BCS proved to be the most
useful of the two. While the WRS worked, the evolution of the
population happened very slowly regardless of the fitness function
used. Figure \ref{fig:ga_selector_comparison} shows the result of the
two selectors while using the same fitness function (BD2A). Similar
results hold for all other fitness functions as well.

\begin{figure}[htp]
  \begin{center}
    \includegraphics[width=10cm]{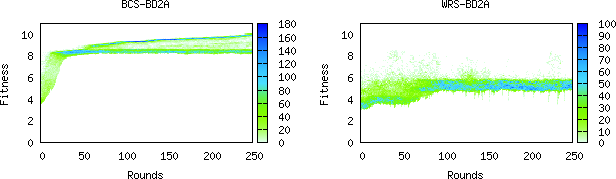}
    \caption{Comparison of two different selectors}
    \label{fig:ga_selector_comparison}
  \end{center}
\end{figure}

Because of these findings, the BD2M fitness function has been tested
with different values for the best performing selector BCS. The
selector allows to configure the fraction of voters that advances from
one generation to the next. With a value of $50\%$ only the better
half of chromosomes advances. This leads to an extremely narrow search
that runs into high risks of lingering at a local optimum. To increase
the chances of leaving a local optimum again, the percentage should be
increased. In the simulation, runs with ratios from the set $\{0.5,
0.6, 0.7, 0.75, 0.80, 0.85, 0.9, 0.91, 0.92, 0.93, 0.94, 0.95, 0.96,
0.97, 0.98, 0.99\}$ have been used. From these runs, only ratios above
$0.75$ resulted in stable evolution patterns, while rates below had
their fitness values alternate between very low and very high values
but did not converge.

Above 75\% all runs converged with an optimal rate at around 95\%. The
runs with ratios of $0.95$ and $0.96$ were the ones that produced
strategies that at least gave nine convinced voters. Figure
\ref{fig:ga_bcs_comparison} shows these two runs and the number of
voters that each chromosome generated.

\begin{figure}[htp]
  \begin{center}
    \includegraphics[width=10cm]{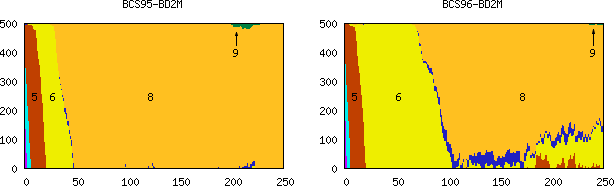}
    \caption{The two runs for different values for the BCS survival rate}
    \label{fig:ga_bcs_comparison}
  \end{center}
\end{figure}

Altogether the GA provides a heuristics to find optimal (or near
optimal) solutions. However, because of the problem structure the GA
did not find an optimal solution for the control problem. The most
convinced voters that the GA could find a strategy for were nine in
the example setting. Strategies yielding nine voters were extremely
rare and could only be obtained in settings with highly tuned
parameters. This result seems to indicate that the solution space has
a very sparse population that could possibly occupy a very restricted
region in the space.

%We have implemented a genetic algorithm as a heuristics for determining a good control.
%
%In Figure \ref{fig_geneticalgoritm_1} we have depicted the results of a run that ended
%in a maximum number of $8$ convinced voters after $10$ stages. It is interesting
%that the remaining $3$ voters have opinions being very close to the conviction interval $[l,r]$.
%
%\begin{figure}[htp]
%  \begin{center}
%    \includegraphics[width=10cm]{genAlgorithmus.jpg}
%    \caption{A control yielding $9$ convinced voters.}
%    \label{fig_geneticalgoritm_2}
%  \end{center}
%\end{figure}
%
%In Figure \ref{fig_geneticalgoritm_2} we have depicted the results of another run of our program.
%Here we have found a control which achieves $9$ convinced voters after $10$ stages.
%The remaining two voters have opinions being very far away from the conviction interval $[l,r]$.

\subsection{The Model Predictive Control Heuristics}
\label{sec:MPC}

The observation that our benchmark problem can be solved to optimality
if restricted to just a few stages suggests the following
receding-horizon control heuristics (RHC), sometimes also known under
the name \emph{model predictive control (MPC)}.  \revJ{See
Algorithm~\ref{alg_MPC} for a pseudocode describing MPC.}

\begin{algorithm}[h]
\SetKwInOut{Input}{Input}
\SetKwInOut{Output}{Output}
\SetKwData{InfluenceRadius}{$\epsilon$}
\SetKwData{Voters}{$I$}
\SetKwData{NumStages}{$N$}
\SetKwData{StartingOpinions}{$\mathbf{x}^0$}
\SetKwData{MPCStartingOpinions}{$\bar{\mathbf{x}}$}
\SetKwData{MPCParameter}{$\check{N}$}
\SetKwData{MPCHorizon}{$\bar{N}$}
\SetKwData{Stage}{$t$}
\SetKwData{MPCControls}{$(\bar{u}^s)_{0 \le s < \MPCHorizon}$}
\SetKwData{MPCFirstControl}{$\bar{u}^0$}
\SetKwData{Control}{$u^t$}
\SetKwFunction{MPCSolution}{OptimalControl(\MPCStartingOpinions, \MPCHorizon, \InfluenceRadius)}
\SetKwFunction{BCDynamics}{BC-dynamics(\Control, \MPCStartingOpinions, \InfluenceRadius)}
\SetKwFunction{ConvincedVoters}{NumberOfConvincedVoters}
\Input{a confidence radius \InfluenceRadius,\\
  a set of agents \Voters,\\
  a horizon \NumStages,\\
  start opinions \StartingOpinions,\\
  an MPC horizon $\MPCParameter \in \mathbb{N}$\\
  a solver \MPCSolution for start opinions \MPCStartingOpinions on horizon \MPCHorizon}
\Output{a control sequence \Control, $0 \le \Stage < \NumStages$, for the strategic agent}
\DontPrintSemicolon
%\dontprintsemicolon
\BlankLine
\revJ{
\MPCStartingOpinions $\leftarrow$ \StartingOpinions \tcc*[f]{initialize opinions}\\
\Stage $\leftarrow$ 0 \tcc*[f]{initialize stage counter}\\
\While(\tcc*[f]{receding-horizon loop}){$\Stage < \NumStages$}{
  \MPCHorizon $\leftarrow$ $\min(\MPCParameter, \NumStages - \Stage)$
  \tcc*[f]{cut the MPC horizon to no.\ of stages left}\\
  \MPCControls $\leftarrow$ \MPCSolution \tcc*[f]{solve MPC-auxiliary problem}\\
  \Control $\leftarrow$ \MPCFirstControl \tcc*[f]{set next control to first MPC control}\\
  \MPCStartingOpinions $\leftarrow$ \BCDynamics \tcc*[f]{update opinions}
}
\Return{\Control}, $0 \le \Stage < \NumStages$\;
\caption{\label{alg_MPC}The MPC Heuristics}
}
\end{algorithm}

%% \begin{enumerate}
%% \item Choose the \emph{MPC-horizon} $\check{N} \le N$, for which the
%%   optimal control problem (the \emph{MPC-auxiliary problem}) can be
%%   solved to optimality.
%% \item \label{itm:receding-horizon:start} Set the state of Stage~$0$ to
%%   the given start configuration of opinions.
%% \item Compute an optimal control for~$\check{N}$ stages.
%% \item Apply the first control.
%% \item Change the start configuration to the resulting configuration in
%%   Stage~$1$.
%% \item Go to \ref{itm:receding-horizon:start}.
%% \end{enumerate}

Although this method uses exact optima of a related optimal control
problem, it is in general a suboptimal control.  The hope is that its
performance is in many cases not too far away from the optimal
objective value.

For our computational results we used the advanced MILP to solve the
upcoming MPC-auxiliary problems.  In order to simplify the setup, we
implemented MPC in the following form: Compute the MPC-auxiliary
problem, then fix the first control to the solution value, finally
increase the number of stages for the MPC-auxiliary problem by one and
repeat.  This way, the opinions in past stages are computed from
scratch in each iteration based on fixed controls.  This does not make
a difference in the resulting number of convinced voters.  Because of
the usage of rounded intermediate results (the fixed controls for
earlier stages), we observed some infeasible problems.  This was cured
by changing the control values by $\pm 10^{-6}$, which cured the
problem in all cases.

\begin{table}[h]
  \centering\footnotesize
  \begin{tabular}{rrrrrr}
    \toprule
    MPC-horizon & \# convinced voters & CPU time [s] for last
    MPC-auxiliary problem \\
    \midrule
    3 & 6       &    6.27\\ 
    4 & 8       &   60.41\\
    5 & 8       &  163.43\\
    \bottomrule    
  \end{tabular}
  \caption{Results of the MPC heuristics for various shorter horizons
    based on the advanced MILP model applied to the benchmark problem
    with $10$ stages;
    MacBook Pro 2013, 2.6\,GHz Intel Core i7, 16\,GB 1600\,MHz DDR3
    RAM, OS~X~10.9.2, \texttt{zimpl} 3.3.1, \texttt{cplex} 12.5
    (Academic Initiative License), branching priorities given
    according to stage structure.} 
  \label{tab:MPC-results}
\end{table}
Table~\ref{tab:MPC-results} shows the computational results for the
MPC heuristics.  We see that it is unable to find the optimal control
with $11$ convinced voters in $10$ stages.  This indicates that the
control leading to $11$ convinced voters must use control
sub-sequences of at least length five that are suboptimal on the
shortened horizon.  This is another hint that optimal controls 
exhibit non-trivial structures.

The conclusion of this section on heuristics is that an intelligent
tailor-made combinatorial heuristics works best on our very special
benchmark problem.  The reason for this is most probably that in the
fitness landscape of this benchmark problem solutions with more than
nine convinced voters are rare and hidden.

For a more general assessment, all methods should compete on a
complete benchmark suite with random perturbations of parameters.
This, however, goes beyond the purpose of this paper, in which we
wanted to advertise the field of optimal opinion control and show some
possible directions of future research.

%%%%%%%%%%%%%%%%%%%%%%%%%%%%%%%%%%%%%%%%%%%%%%%%%%%%%%%%%%%%%%%%%%%%%%%%%%%%%%
%% Section: Interpretation of Results
%%%%%%%%%%%%%%%%%%%%%%%%%%%%%%%%%%%%%%%%%%%%%%%%%%%%%%%%%%%%%%%%%%%%%%%%%%%%%%

\section{Interpretation of  Results}
\label{sec:interpr-results}

The computational results are twofold: first we have collected results
about \emph{optimal controls} including their performance (how many
convinced voters are possible?) and about
\emph{models and solution methods} including their effectivity (how
tight is the information they compute?), and efficiency (how quickly
do they compute that information?).

The results about \emph{models and solution methods} confirm that the
BC optimal control problem is much more \revR{difficult} than the DG optimal
control problem. BC optimal control must be modeled with care: the
advanced model performs much better in all tests than the basic model.
It is the method that, so far, provides the tightest information on
the performance of optimal controls in all our benchmark
problems. Still, our ability to find provably optimal controls is
limited.  Using the fast strongest-guy heuristic, we could find upper
bounds for the number of achievable convinced voters that are tight
for our benchmark problem in all known cases.

The results about \emph{optimal controls} themselves confirm that \revR{the} 
BC dynamics causes many interesting effects: Neither is the number of
achievable convinced voters monotone in time, nor does every optimal
campaign for $N$ stages constitute an optimal campaign for any smaller
number of stages.  In other words: Enlarging the time horizon in an
otherwise unchanged campaign problem requires a \revR{completely} new solution.
Though this sounds plausible in the real-world, we do not claim that
this conclusion can safely be transferred to real-world
campaign-planning -- what we rather state is that it takes no more
than the mechanism of BC opinion dynamics to let this effect emerge.

%%%%%%%%%%%%%%%%%%%%%%%%%%%%%%%%%%%%%%%%%%%%%%%%%%%%%%%%%%%%%%%%%%%%%%%%%%%%%%
%% Section: Conclusion and outlook
%%%%%%%%%%%%%%%%%%%%%%%%%%%%%%%%%%%%%%%%%%%%%%%%%%%%%%%%%%%%%%%%%%%%%%%%%%%%%%
\section{Conclusion and Outlook}
\label{sec:conclusion-and-outlook}

We have introduced the problem of optimal opinion control by simply
allowing one individual to freely choose in each stage its
advertised opinion.  All efforts to find optimal controls
in a small example instance showed that the structure of optimal
controls is complicated.  Modeling with MILP-techniques is possible,
but even sophisticated models are hard to solve.  An optimal control
for the campaign problem with eleven voters and one through
ten stages remains open for seven through nine stages.  Ten stages
could be solved by the strongest-guy heuristics, which is able to
convince all eleven voters in ten stages.  Popular meta-heuristics
like genetic algorithms and model predictive control could not find
this solution. 

The fact that the small campaign problem is still largely a mystery,
make its investigation interesting for further research. \revJ{More
  specifically:
  \begin{itemize}
  \item What happens if the number of voters increases? How do methods
    to compute optimal controls scale with this paramter?  If there
    are really many voters (like 1000), what is a plausible control
    concept to gain a significant influence?
  \item How do optimal controls change with respect to certain
    properties of the uncontrolled dynamics? Are situations with
    uncontrolled consensus easier to control optimally?  Are
    situations easier to control optimally where uncontrolled opinions
    automatically would end up in the conviction interval?
  \item Are there more counter-intuitive effects? Can it happen that
    constantly controlling with the central opinion of the conviction
    interval yields fewer voters than no control at all?
  \item What happens if we have control restrictions?  If each control
    has to be inside an interval around the true opinion of the
    controller, how do optimal controls change and are the methods to
    compute them still applicable?  What if the true opinion inducing
    these restrictions is itself subject to opinion dynamics?
  \end{itemize}
}
Other directions are the generalization to more
than one controller (game theory) and multi-dimensional opinion
spaces.  \revJ{That is:
  \begin{itemize}
  \item How should a suitable game be defined in the first place?
    Should it be one game with one strategy consisting of ten controls?
    Should it be an iterated game with only one control as a strategy?
    Are there equilibria in one and/or the other model?
  \item Do all equilibrium strategies lead to the same `election
    winner' (= party who `convinced' most voters)?
  \item Do three or more controllers make a structural difference?
  \item Is it `easier' (in a sense yet to be defined) to `win an
    election' with a `radical' (close to zero or one) or a `moderate'
    (close to 0.5 or in the middle of the competitors) party opinion?
  \item Does any of these change if the opinion space has a higher
    dimension?
  \end{itemize}
  We think that optimal opinion control opens the door to a wealth of
  questions posing mathematical problems that are interesting in their
  own right.  Moreover, these questions and also possible answers
  trigger equally exciting challenges for their interpretations.
  After all, formally we are investigating only an artificial society,
  and what we can learn from the results is no output of mathematics
  -- it has to be discussed very carefully.}

\section*{Acknowledgements}

We thank the anonymous referees for very valuable suggestions to
improve the presentation of the paper.

%\bibliographystyle{amsplain}
%%\bibliographystyle{jtbnew}
%%\bibliography{PS}

\appendix
%%%%%%%%%%%%%%%%%%%%%%%%%%%%%%%%%%%%%%%%%%%%%%%%%%%%%%%%%%%%%%%%%%%%%%%%%%%%%%
%% Section: A Mathematical Model for the Optimal Control Problem
%%%%%%%%%%%%%%%%%%%%%%%%%%%%%%%%%%%%%%%%%%%%%%%%%%%%%%%%%%%%%%%%%%%%%%%%%%%%%%
\section{Exact Mathematical Models for  Optimal Opinion Control}
\label{sec:Mathematical-Model}

In this appendix, we present and briefly explain the detailed
mathematical Mixed Integer Linear Programming (MILP) models used in
this paper.  This should allow the interested reader to replicate our
results.  

The chosen modeling technique is -- not surprisingly -- much more
powerful in the DG-model than in the BC-model; the former takes profit
of the linear system dynamics whereas the latter suffers a lot from
the highly non-continuous system dynamics and the numerical
instability.  More specifically, our model is not able to represent
the original problem exactly.  We will provide, however, actually two
models: one is correct in the sense that every upper bound on the
objective value of the model is an upper bound on the optimal number
of convinced voters in the original problem (but possibly not vice
versa); the other one is correct in the sense that any feasible
solution to it is a feasible solution to the original problem (but
possibly not vice versa).

% FIXME: Der naechste Absatz klingt arg "coloquial"... (CN)
%% Why do we choose a MILP formulation as a model? Our original problem does not
%% look like a mixed integer problem at all.  Well, MILP is at least able to
%% capture combinatorial structures like the confidence set or the conviction
%% set, and the objective value of mixed integer programs is in general not
%% continuous in the input data.  After all, frankly speaking, we had no better
%% idea.

The motivation for using integral variables in a model for our optimal
control problem is that the dynamics mainly depends on the structures
of the confidence sets and the conviction sets: We can use binary
variables to indicate how the conviction sets and the confidence sets,
resp., look like.

Since the BC-model requires some experience in modeling with MILPs, we
start with a model for the DG optimal control problem.  Later on, when
the main principles are explained, we will present a model for the
BC-model.

\subsection{An MILP Model for the DG Optimal Control Problem}
\label{sec:MILP-DeGroot}

In this section we present a mathematical model, a Mixed Integer
Linear Programming model (MILP), for the solution of the DeGroot
optimal opinion control problem.  We start with the DeGroot
dynamics in order to explain some crucial MILP-modeling techniques in
this easier dynamics.  These techniques will be used extensively for
the bounded-confidence dynamics, in which the logic is considerably
more complicated.

The following MILP model is based on standard modeling techniques in
(MILP).  We first list the \emph{variables} of the model.
\begin{itemize}
\item The continuous variables $x^t_0 \in [0,1]$, $t = 0, 1, \dots, N-1$ denote the
  positions in opinion space where we place a control in the various stages;
  these are the variables that we are really after.
\item The continuous variables $x^t_i \in [0,1]$, $i \in I$, $t = 0, 1, \dots,
  N$ denote the positions of the voters in the various stages; these
  variables measure the system states.  The variables in stage~$0$ are given
  as input data (start state / start value).
\item For each voter, we want to measure whether its position in
  stage~$N$ is inside the conviction interval; to this end, we use binary
  variables $z_i \in \{0, 1\}$, $i \in I$, with the following meaning:
  $z_i = 1$ if and only if $i$ is convinced in stage~$N$, i.e., $x^N_i \in
  [\ell, r]$.
\end{itemize}

With this, we may formulate the \emph{goal} of the model: we want to maximize the
number of convinced voters, which can be expressed as follows:
\begin{equation}
  \label{eq:DeGroot-MILP-Goal}
  \max \sum_{i \in I} z_i.
\end{equation}

Now, the success measuring variables $z_i$ have to be coupled with our
decisions~$x^t_0$ via the system states and the system dynamics.  The
following linear \emph{side constraint} couples the decisions to the system
states:

\begin{equation}
  \label{eq:DeGroot-MILP-Dynamics}
  x^{t+1}_i = \sum_{i \in I_0} w_{ij} x^t_j \quad \text{for all }i \in I, t = 0, 1, \dots, N-1.
\end{equation}

So far, we did not restrict the binary variables.  A solver would
simply set them all to~$1$ and achieve an objective value of~$n$ (all
convinced), because the binary variables so far have nothing to do
with the underlying dynamical system.

The binary variables can now be coupled to the system state variables
in stage~$N$ by a standard MILP modeling trick as follows. The logical
implication must be: If $z_i = 1$, i.e., if we want to count an
voter as convinced, then $\ell \le x^N_i \le r$ must hold.  In
other words, the inequalities $\ell \le x^N_i \le r$ can be violated
when $z_i = 0$, but they must be satisfied whenever $z_i = 1$.  Thus,
whether or not we demand the restriction $\ell \le x^N_i \le r$
depends on the value of a variable.  We call such a conditional
restriction a \emph{variable-conditioned constraint} and write it as
$\ell \le x^N_i \le r \ \text{vif $z_i = 1$}$.  The MILP modeling
trick can transform such a variable-conditioned constraint into a set
of unconditioned constraints in all cases where the violation of the
variable-conditioned constraint is bounded.

We show the transformation for the inequality $\ell \le x^N_i$, the
other inequality can be handled analogously.  The maximal violation of
the inequality $\ell - x^N_i \le 0$ is~$\ell$, since $\ell - x \le
\ell$ for all $x \in [0,1]$.  That means, the inequality $\ell - x^N_i
\le \ell$ does trivially hold, no matter where $x^N_i$ is in $[0,1]$.
We want to impose the trivial inequality $\ell - x^N_i \le \ell$
whenever $z_i = 0$ and the non-trivial inequality $\ell - x^N_i \le 0$
whenever $z_i = 1$.  But this can be achieved in one step by imposing
the inequality
\begin{equation}
  \label{eq:DeGroot-MILP-Convince-Left}
  \ell - x^N_i \le \ell(1 - z_i) \quad \text{for all }i \in I.
\end{equation}

The analogously derived inequality for the right border of the conviction
interval reads
\begin{equation}
  \label{eq:DeGroot-MILP-Convince-Right}
  x^N_i - r \le (1 - r)(1 - z_i) \quad \text{for all }i \in I.
\end{equation}

The complete MILP reads as follows:
\begin{align}
  &\lefteqn{\max\sum_{i \in I} z_i}\\
  &\text{subject to}\notag\\
  &&x^{t+1}_i    &= \sum_{i \in I_0} w_{ij} x^t_j && \text{for all }i \in I, t = 0, 1, \dots, N-1,\\
  &&\ell - x^N_i &\le \ell(1 - z_i)               && \text{for all }i \in I,\\
  &&x^N_i - r    &\le (1 - r)(1 - z_i)            && \text{for all }i \in I,\\
  && z_{i} &\in \{0, 1\} 
  &&\text{$\text{for all }i \in I$},\\
  && x^t_{i} &\in [0, 1] 
  &&\text{$\text{for all }t = 0, 1, \dots, N$, $i \in I \cup \{0\}$}.
\end{align}
In the following we will not spell out anymore the results of such
transformations.  Instead, we will present the variable-conditioned
constraints literally in order to make the logic more decipherable.
The above MILP with literally expressed variable-conditioned
constraints reads as follows:
\begin{align}
  &\lefteqn{\max\sum_{i \in I} z_i}\\
  &\text{subject to}\notag\\
  &&x^{t+1}_i    &= \sum_{i \in I_0} w_{ij} x^t_j && \text{for all }i
  \in I, t = 0, 1, \dots, N-1,\\
  &\lefteqn{\text{vif $z_i = 1$}}\\
  &&\ell &\le x^N_i \le r  && \text{for all }i \in I,\\
  &\lefteqn{\text{end}}\\
  && z_{i} &\in \{0, 1\} 
  &&\text{$\text{for all }i \in I$},\\
  && x^t_{i} &\in [0, 1] 
  &&\text{$\text{for all }t = 0, 1, \dots, N$, $i \in I \cup \{0\}$}.
\end{align}
The reader should bear in mind that MILP models with additional
variable-conditioned constraints with bounded violation can be
transformed into true MILP models.  Thus, such models are accessible
for standard MILP solvers like \texttt{cplex} \cite{cplex:2014}.
Modeling languages like \texttt{zimpl}
\cite{Koch:RapidMathematicalPrototyping:2004} even support
variable-conditioned constraints directly.

The MILP for the DG-model can be solved efficiently by
of-the-shelf software like \texttt{cplex}.  In particular, solving our
benchmark problem for any number of stages between $1$ and~$10$ is
possible.  For example, $11$ convinced voters are possible with only
one round for uniform weights, and this does not even need the help of
a control.

\subsection{MILP Models for the BC Optimal Control Problem}
\label{sec:MILP-Bounded-Confidence}

The optimal control problem in bounded-confidence dynamics is
non-continuous, thus non-linear.  Nevertheless, one can construct an
MILP model for it by using variable-conditioned constraints in a
similar way as in the previous section.

We introduce a real parameter $\hat{\epsilon}$ (meant to be of small
absolute value; in our computational experiments we chose
$\hat{\epsilon} = \pm 10^{-5}$) with the following meaning: whenever
$j$ is not in the confidence interval of~$i$, then $\abs{x_i - x_j}
\ge \epsilon + \hat{\epsilon}$ must hold.  For $\hat{\epsilon} > 0$,
this is stronger than the original condition, which is: if $j$ is not
in the confidence interval of~$i$, then $\abs{x_i - x_j} > \epsilon$
must hold, and vice versa.  This original condition is a strict
inequality that can not be handled directly in MILPs, and a
transformation to a different MILP (in modified so-called homogeneous
variables) is usually numerically highly unstable.

With the modified condition we can choose either to exclude
potentially feasible solutions (this happens for $\hat{\epsilon} > 0$)
or to include potentially infeasible solutions (this happens for
$\hat{\epsilon} \le 0$).  In the latter case, we grant the
optimization algorithm to choose freely in particular whether or not
$j$ is in the confidence interval of~$i$ whenever $\abs{x_i - x_j} =
\epsilon$.

Thus, the model needs to be applied twice: for the identification of
feasible solutions we need to set $\hat{\epsilon}$ to something
strictly positive, and for the determination of upper bounds on the
optimal number of convinced voters we need to set $\hat{\epsilon}$ to
at most zero.  The larger the absolute value of~$\hat{\epsilon}$ is,
the more robust the conclusions are against rounding errors.

If we run the model only once with, e.g., $\hat{\epsilon} = 10^{-5}$
we allow for a small inaccuracy in the upper bound obtained by the
model.  Such an inaccuracy can not be avoided when a standard
MILP-solver is used -- the most powerful solvers like \texttt{cplex}
\cite{cplex:2014}, \texttt{xpress}, or \texttt{gurobi} use bounded-precision floating
point arithmetics, and an accuracy of~$10^{-6}$ is a common setting.
Any solution that we miss this way, however, would be non-robust in
the sense that a slight deviation from the system dynamics would lead
to a different objective.

There are several modeling options out of which we present two.  The
first model extends the MILP for the DG-model using similar
techniques to a much larger extent.  We could solve it by the standard
MILP solver \texttt{cplex} up to $N = 4$ in less than an hour.  For $N
= 5$ the solver could not even get close to a proven optimal solution
in weeks.  The second model is a carefully engineered, more
complicated system comprising some experience in MILP techniques.
With the second model we were able to solve the benchmark problem up
to $N = 6$ with \texttt{cplex}.

We suspect that the solution for $N = 7$ and above requires
tailor-made MILP models and solution techniques.

For both our models, we assume that all voters are numbered according
to their starting opinion, i.e., $i < j$ implies $x^0_i \le x^0_j$ for
$i, j \in I$.  This saves some work since the order of voters in the
opinion space does never change, due to Lemma~\ref{lem:ordering}.
%%of the following:
%%\begin{Observation}
%%  \label{obs:Order-Preserving}
%%  The order of voters in opinion space does never change, i.e., if
%%  $x^t_i \le x^t_j$ for some $i, j \in I$ and some $t = 0, 1, \dots,
%%  N-1$, then $x^{t+1}_i \le x^{t+1}_j$.
%%\end{Observation}

Our first, basic model uses the following variables:
\begin{itemize}
\item The control variables $x^t_0$, $t = 0, 1, \dots, N-1$ model the
  opinions published by the controller, as above.  These are the only
  independent decision variables. The remaining variables are
  dependent measurements to compute the objective function.
\item The state variables $x^t_i$, $i \in I$, $t = 0, 1, \dots, N$ measure
  the opinions of agent~$i$ in stage~$t$, as above.
\item The binary indicator variables $v^t_{i,j}$ are one if and only
  if voters $i < j$ are within distance~$\epsilon$, i.e., they
  influence each other.
\item The binary indicator variables $l^t_{i}$, $r^t_{i}$, and
  $c^t_{i}$ are one if and only if the control in stage~$t$ is to the
  left by a margin of at least~$\hat{\epsilon}$, strictly to the right by a
  margin of at least~$\hat{\epsilon}$, or inside the confidence interval of
  voter~$i$.
\item The binary indicator variables $z_i$, $i \in I$, are one if and
  only if voter~$i$ is within the conviction interval
  $[\ell, r]$ in the final stage~$N$, as above.
\item The measurement variables $\bar{x}^t_{0,i}$, $i \in I$,
  $t = 0, 1, \dots, N$, denote the contribution of the control
  opinion~$x^t_0$ in the system dynamics formula 
  in stage~$t$; this variable must
  equal $x^t_0$ if the control is in the confidence interval of
  voter~$i$; it must be zero otherwise.
\item The measurement variables $\bar{x}^t_{j,i}$, $i, j \in I$, $t =
  0, 1, \dots, N$, denote the contribution of the voter's
  opinion~$x^t_j$ in the system dynamics formula of voter~$i$ in
  stage~$t$; this variable must equal $x^t_j$ if that opinion is in
  the confidence interval of voter~$i$; it must be zero otherwise.
\item The count variables $k^t_i$, $i \in I$, $t = 0, 1, \dots, N-1$,
  denote the number of voters in the confidence interval of voter~$i$
  in stage~$t$.
\end{itemize}

With this set-up and the aforementioned use of linearized
variable-conditioned constraints, we can formulate the following basic
model.  The logical details are explained right after the presentation
of the MILP.
\begin{footnotesize}
  \begin{align}
    &\lefteqn{\max \sum_{i \in I} z_i}\label{eq:BC-basic:Obj}\\
    &\lefteqn{\text{subject to}}\notag\\
    && x^0_{i} &= x^{\text{start}}_i
    &&\text{$\forall\;i \in I$}\label{eq:BC-basic:Start},\\
    && l^t_i + r^t_i + c^t_i &= 1
    && \text{$\forall\;t = 0, 1, \dots, N-1$, $i \in I$}\label{eq:BC-basic:ControlPos},\\
    &\lefteqn{\text{vif $c^t_i = 1$ then}}\notag\\
    && x^t_0 - x^t_i & \le \epsilon\label{eq:BC-basic:Control-RightBound}\\
    && x^t_i - x^t_0 & \le \epsilon\label{eq:BC-basic:Control-LeftBound}\\
    &\lefteqn{\text{end}}&&
    && \text{$\forall\;t = 0, 1, \dots, N-1$, $i \in I$}\notag,\\
    &\lefteqn{\text{vif $r^t_i = 1$ then}}\notag\\
    && x^t_0 - x^t_i & \ge \epsilon + \hat{\epsilon}\label{eq:BC-basic:Control-RightExceed}\\
    &\lefteqn{\text{end}} &&
    && \text{$\forall\;t = 0, 1, \dots, N-1$, $i \in I$}\notag,\\
    &\lefteqn{\text{vif $l^t_i = 1$ then}}\notag\\
    && x^t_i - x^t_0 & \ge \epsilon + \hat{\epsilon}\label{eq:BC-basic:Control-LeftExceed}\\
    &\lefteqn{\text{end}} && 
    && \text{$\forall\;t = 0, 1, \dots, N-1$, $i \in I$}\notag,\\
    &\lefteqn{\text{vif $v^t_i = 1$ then}}\notag\\
    && x^t_j - x^t_i & \le \epsilon\label{eq:BC-basic:Voter-Bound}\\
    &\lefteqn{\text{else}} &&\notag\\
    %%&\lefteqn{\text{end}} &&
    %%&& \text{$\forall\;t = 0, 1, \dots, N-1$, $i, j \in I\colon i < j$}\notag\\
    %%\lefteqn{\text{vif $l^t_i = 1$ then}}\notag\\
    && x^t_j - x^t_i & \ge \epsilon + \hat{\epsilon}\label{eq:BC-basic:Voter-Exceed}\\
    &\lefteqn{\text{end}} &&
    && \text{$\forall\;t = 0, 1, \dots, N-1$, $i, j \in I\colon i < j$}\notag,\\
    && k^t_i &= \sum_{j \in I\setminus\{i\}} v^t_{\min(i, j), \max(i, j)} + 1 + c^t_i
    && \text{$\forall\;t = 0, 1, \dots, N-1$, $i \in I$}\label{eq:BC-basic:BC-Count},\\
    &\lefteqn{\text{vif $c^t_i = 1$ then}}\notag\\
    && \bar{x}^t_{0, i} &= x^t_0\label{eq:BC-basic:Control-Contribution}\\
    &\lefteqn{\text{else}}\notag\\
    && \bar{x}^t_{0, i} &= 0\label{eq:BC-basic:Control-NoContribution}\\
    &\lefteqn{\text{end}} &&
    && \text{$\forall\;t = 0, 1, \dots, N-1$, $i \in I$}\notag,\\
    &\lefteqn{\text{vif $v^t_{\min(i,j), \max(i,j)} = 1$ then}}\notag,\\
    && \bar{x}^t_{j, i} &= x^t_j\label{eq:BC-basic:Voter-Contribution}\\
    &\lefteqn{\text{else}}\notag\\
    && \bar{x}^t_{j, i} &= 0\label{eq:BC-basic:Voter-NoContribution}\\
    &\lefteqn{\text{end}} &&
    && \text{$\forall\;t = 0, 1, \dots, N-1$, $i, j \in I\colon i \neq j$}\notag,\\
    &\lefteqn{\text{vif $k^t_i = k$ then}}\notag\\
    && x^{t+1}_i &= \frac{1}{k} \Bigl(\sum_{j \in I\setminus\{i\}} \bar{x}^{t}_{j, i} +
    x^{t}_i + \bar{x}^{t}_{0, i}\Bigr)\label{eq:BC-basic:Dynamics}\\
    &\lefteqn{\text{end}} &&
    && \text{$\forall\;k = 1, 2 \dots, \abs{I} + 1$, $t = 0, 1, \dots, N-1$, $i \in I$}\notag,\\
    &\lefteqn{\text{vif $z_i = 1$ then}}\notag\\
    && x^N_i & \ge \ell\label{eq:BC-basic:Convinced-LeftBound}\\
    && x^N_i & \le r\label{eq:BC-basic:Convinced-RightBound}\\
    &\lefteqn{\text{end}} &&
    && \text{$\forall\;i \in I$}\notag\\
    && v^t_{i,j}, l^t_i, r^t_i, c^t_i, z_i &\in \{0, 1\} 
    &&\text{$\forall\;t = 0, 1, \dots, N-1$, $i, j \in I\colon i <
      j$}\label{eq:BC-basic:Binary},\\
    && k^t_{i} &\in \mathbb{N} 
    &&\text{$\forall\;t = 0, 1, \dots, N-1$, $i \in I$}\label{eq:BC-basic:Integral},\\
    && x^t_{i} &\in [0, 1] 
    &&\text{$\forall\;t = 0, 1, \dots, N$, $i \in I \cup \{0\}$}.\label{eq:BC-basic:Continuous}.
  \end{align}
\end{footnotesize}
The objective function \eqref{eq:BC-basic:Obj} counts the number of
voters in the conviction interval in stage~$N$.  Restriction
\eqref{eq:BC-basic:Start} sets the positions of the opinions in
stage~$0$ to the given start
values. Constraint~\eqref{eq:BC-basic:ControlPos} demands (together
with the fact that all involved variables are binary) that exactly one
of the variables $l^t_i, r^t_i, c^t_i$ must be one.  The meaning is
that the control is either strictly to the left, to the right, or
inside the confidence interval of voter~$i$ in each stage~$t$.  With
restrictions \eqref{eq:BC-basic:Control-RightBound} and
\eqref{eq:BC-basic:Control-LeftBound} we request that whenever $c^t_i
= 1$ the distance between the control and voter~$i$ is no more
than~$\epsilon$ so that the control is really inside $i$'s confidence
interval.  In contrast to this, inequalities
\eqref{eq:BC-basic:Control-RightExceed}
and~\eqref{eq:BC-basic:Control-LeftExceed} make sure that whenever
$r^t_i = 1$ resp.{} $l^t_i = 1$ the control must be to the right
resp.{} to the left with a distance of at least $\epsilon +
\hat{\epsilon}$ from voter~$i$ so that the control is really outside
the $i$'s confidence interval.
Restrictions~\eqref{eq:BC-basic:Voter-Bound} and
\eqref{eq:BC-basic:Voter-Exceed} make sure in a similar way hat the
value $v^t_{i, j}$ correctly indicates whether or not $i$ and~$j$ are
in each others' confidence intervals.  The case distinction between a
large distance to the left or to the right is unnecessary because of
the order of all voters' opinions, reflected by the indices, stays
fixed throughout the process.
Constraints~\eqref{eq:BC-basic:BC-Count} sets $k^t_i$ to the number of
opinions in the confidence interval of voter~$i$.  Constraints
\eqref{eq:BC-basic:Control-Contribution} and
\eqref{eq:BC-basic:Control-NoContribution} compute how much the
control's opinion contributes to the next opinion of voter~$i$.  This
is either the control's opinion in case $c^t_i = 1$ or zero in case
$c^t_i = 0$.  Similarly, constraints
\eqref{eq:BC-basic:Voter-Contribution}
and~\eqref{eq:BC-basic:Voter-NoContribution} compute the contribution
of voter~$j$ to the next opinion of voter~$i$ depending on the value
of~$v^t_{\min(i, j),\max(i, j)}$.  Depending on how many opinions are
in the confidence interval of voter~$i$, we can now compute its next
opinion by restriction~\eqref{eq:BC-basic:Dynamics}.
Constraints~\eqref{eq:BC-basic:Convinced-LeftBound}
and~\eqref{eq:BC-basic:Convinced-RightBound} make sure \rev{the}
classification in variable $z_i$ of being convinced is consistent with
the distance of $i$'s opinion to the conviction interval.

Our second, more sophisticated model uses the following variables:
\begin{itemize}
\item The control variables $x^t_0$, $t = 0, 1, \dots, N-1$ are as above.
\item Similarly, the state variables $x^t_i$, $i \in I$, $t = 0, 1, \dots, N$
  are as above.
\item For $j_{\min}, j_{\max} \in I$ and $c_l, c_r \in \{0,1\}$, we
  introduce variables
  $\varconfidence{t}{i}{j_{\min}}{j_{\max}}{c_l}{c_r}$ where
  $\varconfidence{t}{i}{j_{\min}}{j_{\max}}{c_l}{c_r} = 1$ if and only
  if the following holds: voter~$j_{\min}$ is the minimal index of a
  voter in the confidence interval of~$i$, voter~$j_{\max}$ is the
  maximal index of a voter in the confidence interval of~$i$,
  Index~$c_l = 1$ if and only if $x^t_0 \ge x^t_i - \epsilon$ (i.e.,
  the control is not to the left of the confidence interval of
  voter~$i$), and Index~$c_r = 1$ if and only if $x^t_0 \le x^t_i +
  \epsilon$ (i.e., the control is not to the right of the confidence
  interval of voter~$i$).  In particular, all variables
  $\varconfidence{t}{i}{j_{\min}}{j_{\max}}{0}{0}$ must be zero.  The
  motivation for these variables is that they are indicating the
  unique \emph{combinatorial confidence configuration}
  $\confconfidence{j_{\min}}{j_{\max}}{c_l}{c_r}$ of a voter: If
  $\varconfidence{t}{i}{j_{\min}}{j_{\max}}{c_l}{c_r} = 1$ then we
  know by Lemma~\ref{lem:ordering} that all voters $j
  \in I$ with $j_{\min} \le j \le j_{\max}$ influence~$i$ and that the
  current control influences~$i$ if and only if $l = r = 1$.  In MILP
  language, these variables are \emph{assignment variables} that
  assign to each voter a unique combinatorial confidence
  configuration.
\item For $j_{\min}, j_{\max} \in I$, we introduce variables
  $\varconviction{j_{\min}}{j_{\max}}$ where
  $\varconviction{j_{\min}}{j_{\max}} = 1$ if and only if the following holds:
  $j_{\min}$ is the minimal index of a voter in the conviction interval
  in stage~$N$, and $j_{\max}$ is the maximal index of a voter in the
  conviction interval in stage~$N$.  The motivation for these variables is
  that they are indicating the unique \emph{combinatorial conviction
    configuration} $\confconviction{j_{\min}}{j_{\max}}$ in the final stage:
  If $\varconviction{j_{\min}}{j_{\max}} = 1$ then the number of convinced
  voters in stage~$N$ is simply $j_{\max} - j_{\min} + 1$.
\end{itemize}
With the variables above, a logically consistent model can be
formulated, which can solve the benchmark instance up to $N = 5$.
Some additional engineering effort was required in order to help
\texttt{cplex} \cite{cplex:2014} to obtain the optimal value for $N =
6$ as well.  For this, we need the following auxiliary variables.
\begin{itemize}
\item For each voter~$i \in I$ and each stage~$t = 1, \dots, N$ we
  introduce measurement variables $\vardistleft{t}{i}$ and
  $\vardistright{t}{i}$ denoting the left and right distances of
  voter~$i$ to the conviction interval $[\ell, r]$.  The motivation
  for these variables is that they provide a continuous measurement
  for how close we are to convince more voters in stage~$t+1$.  Thus,
  with these variables we can perturb the objective function to reduce
  the dual degeneracy of the model, i.e., solutions with identical
  original objective value up to stage~$t$ have distinct perturbed
  objective values, hinting at which solution has better chances to
  improve in the later stages.
\item For $i, j \in I$ with $i < j$ and $t = 0, 1, \dots, N$ we
  introduce binary indicator variables $\varvoterconf{t}{i}{j}$ with the
  following meaning: $\varvoterconf {t}{i}{j} = 1$ if and only if in
  stage~$t$ the confidence interval of voter~$j$ contains voter~$i$.
  This is the case if and only if in stage~$t$ the confidence interval
  of voter~$i$ contains voter~$j$.  The motivation for these variables
  is, first, to transfer the above symmetry relation to a relation
  among combinatorial confidence configurations and, second, that
  branching on these new additional variables leads to more balanced
  subproblems than branching on the variables for the combinatorial
  confidence configurations.
\item In the same spirit, we introduce for $i \in I$ and $t = 0, 1,
  \dots, N$ binary indicator variables $\varcontrolconf{t}{i}$ with
  the following meaning: $\varcontrolconf{t}{i} = 1$ if and only if in
  stage~$t$ the control is in the confidence interval of voter~$i$.
  The motivation is again that a more balanced branching is possible.
\end{itemize}

The resulting model, presenting all variable-conditioned constraints
literally as above, reads as follows.  Again, detailed explanations
follow the presentation of the model.
\begin{footnotesize}
  \begin{align}
    &\lefteqn{\max \sum_{(j_{\min} \le j_{\max})}
      (j_{\max} - j_{\min} + 1) \varconviction{j_{\min}}{j_{\max}}}\label{eq:BC-advanced:Obj}\\
    &\lefteqn{
      {} + 1 
      - 
      %\frac{1}{N - 1} \cdot
      \frac{1}{N} \cdot   
      \sum_{t = 1}^N 
      \frac{1}{\ell} \cdot  
      \frac{1}{n} \cdot 
      \sum_{i \in I} \vardistleft{t}{i}
      -
      %\frac{1}{N - 1} \cdot
      \frac{1}{N} \cdot   
      \sum_{t = 1}^N 
      \frac{1}{1 - r} \cdot
      \frac{1}{n} \cdot 
      \sum_{i \in I} \vardistright{t}{i}
    }
    \label{eq:BC-advanced:ObjPerturbation}\\
    &\text{subject to} \notag\\
    && x^0_{i} &= x^{\text{start}}_i
    &&\text{$\forall\;i \in I$}\label{eq:BC-advanced:Start},\\
    &&\sum_{\substack{j_{\min} \le i \le j_{\max}\\c_l, c_r \in \{0, 1\}}} 
    \varconfidence{t}{i}{j_{\min}}{j_{\max}}{c_l}{c_r} &= 1
    && \text{$\forall\;t = 0, 1, \dots, N-1$},\notag\\*
    &&&&&\text{$i \in I$}\label{eq:BC-advanced:ConfidenceAssignment},\\
    &\lefteqn{\text{vif 
        $\sum_{\substack{j_{\max} \ge i\\c_l, c_r \in \{0, 1\}}} 
        \varconfidence{t}{i}{j_{\min}}{j_{\max}}{c_l}{c_r} = 1$ then}}\notag\\
    &&x^t_i - x^t_{j_{\min}} &\le\epsilon\label{eq:BC-advanced:Voter-BoundLeft}\\
    &\lefteqn{\text{end}} &&
    && \text{$\forall\;t = 0, 1, \dots, N-1$},\notag\\*
    &&&&&\text{$i \in I$},\notag\\*
    &&&&&\text{$j_{\min} \le i$}\notag,\\
    &\lefteqn{\text{vif 
        $\sum_{\substack{j_{\min} \le i\\c_l, c_r \in \{0, 1\}}} 
        \varconfidence{t}{i}{j_{\min}}{j_{\max}}{c_l}{c_r} = 1$ then}}\notag\\
    &&x^t_{j_{\max}} - x^t_i &\le\epsilon\label{eq:BC-advanced:Voter-BoundRight}\\
    &\lefteqn{\text{end}} &&
    && \text{$\forall\;t = 0, 1, \dots, N-1$},\notag\\*
    &&&&&\text{$i \in I$},\notag\\*
    &&&&&\text{$j_{\max} \ge i$}\notag,\\
    &\lefteqn{\text{vif 
        $\sum_{\substack{j_{\max} \ge i\\c_l, c_r \in \{0, 1\}}} 
        \varconfidence{t}{i}{j_{\min}}{j_{\max}}{c_l}{c_r} = 1$ then}}\notag\\
    &&x^t_i - x^t_{j_{\min} - 1} &\ge \epsilon + \hat{\epsilon}\label{eq:BC-advanced:Voter-ExceedLeft}\\
    &\lefteqn{\text{end}}
    &&
    && \text{$\forall\;t = 0, 1, \dots, N-1$},\notag\\*
    &&&&&\text{$i \in I$},\notag\\*
    &&&&&\text{$0 < j_{\min} \le i$}\notag,\\
    &\lefteqn{\text{vif 
        $\sum_{\substack{j_{\min} \le i\\c_l, c_r \in \{0, 1\}}}
        \varconfidence{t}{i}{j_{\min}}{j_{\max}}{c_l}{c_r} = 1$ then}}\notag\\
    && x^t_{j_{\max} + 1} - x^t_i &\ge \epsilon + \hat{\epsilon}\label{eq:BC-advanced:Voter-ExceedRight}\\
    &\lefteqn{\text{end}}
    &&
    && \text{$\forall\;t = 0, 1, \dots, N-1$},\notag\\*
    &&&&&\text{$i \in I$},\notag\\*
    &&&&&\text{$i \le j_{\max} < n$}\notag,\\
    &\lefteqn{\text{vif 
        $\sum_{\substack{j_{\min} \le i \le j_{\max}\\c_r \in \{0, 1\}}}
        \varconfidence{t}{i}{j_{\min}}{j_{\max}}{1}{c_r} = 1$ then}}\notag\\
    && x^t_i - x^t_0 &\le \epsilon\label{eq:BC-advanced:Control-BoundLeft}\\
    &\lefteqn{\text{end}}
    &&
    && \text{$\forall\;t = 0, 1, \dots, N-1$},\notag\\*
    &&&&&\text{$i \in I$},\notag\\
    &\lefteqn{\text{vif 
        $\sum_{\substack{j_{\min} \le i \le j_{\max}\\c_l \in \{0, 1\}}}
        \varconfidence{t}{i}{j_{\min}}{j_{\max}}{c_l}{1} = 1$ then}}\notag\\
    &&x^t_0 - x^t_i &\le \epsilon\label{eq:BC-advanced:Control-BoundRight}\\
    &\lefteqn{\text{end}}
    && &&
    \text{$\forall\;t = 0, 1, \dots, N-1$},\notag\\*
    &&&&&\text{$i \in I$},\notag\\
    &\lefteqn{\text{vif 
        $\sum_{\substack{j_{\min} \le i \le j_{\max}\\c_r \in \{0, 1\}}}
        \varconfidence{t}{i}{j_{\min}}{j_{\max}}{0}{c_r} = 1$ then}}\notag\\
    &&x^t_i - x^t_0 &\ge \epsilon + \hat{\epsilon}\label{eq:BC-advanced:Control-ExceedLeft}\\
    &\lefteqn{\text{end}}
    &&
    && \text{$\forall\;t = 0, 1, \dots, N-1$},\notag\\*
    &&&&&\text{$i \in I$},\notag\\
    &\lefteqn{\text{vif 
        $\sum_{\substack{j_{\min} \le i \le j_{\max}\\c_l \in \{0, 1\}}}
        \varconfidence{t}{i}{j_{\min}}{j_{\max}}{c_l}{0} = 1$ then}}\notag\\
    &&x^t_0  - x^t_i &\ge \epsilon + \hat{\epsilon}\label{eq:BC-advanced:Control-ExceedRight}\\
    &\lefteqn{\text{end}}
    &&
    && \text{$\forall\;t = 0, 1, \dots, N-1$},\notag\\*
    &&&&&\text{$i \in I$},\notag\\
    &&\sum_{j_{\min} \le j_{\max}} \varconviction{j_{\min}}{j_{\max}} &\le 1
    &&
    \text{$\forall\;t = 1, \dots, N$}\label{eq:BC-advanced:ConvictionAssignment},\\
    &\lefteqn{\text{vif 
        $\sum_{j_{\max} \ge j_{\min}} \varconviction{j_{\min}}{j_{\max}} = 1$ then}}\notag\\
    &&x^t_{j_{\min}} &\ge \ell\label{eq:BC-advanced:Conviction-BoundLeft}\\
    &\lefteqn{\text{end}}
    &&
    && \text{$\forall\;t = 0, 1, \dots, N-1$},\notag\\*
    &&&&&\text{$j_{\min} \in I$},\notag\\
    &\lefteqn{\text{vif 
        $\sum_{j_{\min} \le j_{\max}} \varconviction{j_{\min}}{j_{\max}} = 1$ then}}\notag\\
    &&x^t_{j_{\min}} &\le r\label{eq:BC-advanced:Conviction-BoundRight}\\
    &\lefteqn{\text{end}}
    &&
    && \text{$\forall\;t = 0, 1, \dots, N-1$},\notag\\*
    &&&&&\text{$j_{\max} \in I$},\notag\\
    &\lefteqn{\text{vif 
        $\varconfidence{t}{i}{j_{\min}}{j_{\max}}{c_l}{c_r} = 1$ then}}\notag\\
    &&x^t_i &= \sum_{j \in I: j_{\min} \le j \le j_{\max}} x^{t-1}_j +
              c_l c_r x^{t-1}_0\label{eq:BC-advanced:Dynamics}\\
    &\lefteqn{\text{end}}
    &&
    && \text{$\forall\;t = 1, \dots, N$},\notag\\*
    &&&&&\text{$i \in I$},\notag\\*
    &&&&&\text{$j_{\min}, j_{\max} \in I\colon j_{\min} \le j_{\max}$},\notag\\*
    &&&&&\text{$c_l, c_r \in \{0,1\}$}\notag,\\
    &&\vardistleft{t}{i} &\ge \ell - x^t_i
    && \text{$\forall\;t = 1, \dots, N$},\notag\\*
    &&&&&\text{$i \in I$},\label{eq:BC-advanced:DistanceLeft-LB1}\\
    &&\vardistleft{t}{i} &\ge 0
    && \text{$\forall\;t = 1, \dots, N$},\notag\\*
    &&&&&\text{$i \in I$},\label{eq:BC-advanced:DistanceLeft-LB2}\\
    &&\vardistright{t}{i} &\ge x^t_i - r
    && \text{$\forall\;t = 1, \dots, N$},\notag\\*
    &&&&&\text{$i \in I$},\label{eq:BC-advanced:DistanceRight-LB1}\\
    &&\vardistright{t}{i} &\ge 0
    && \text{$\forall\;t = 1, \dots, N$},\notag\\*
    &&&&&\text{$i \in I$},\label{eq:BC-advanced:DistanceRight-LB2}\\
    &&\vardistleft{t}{i} &\le \ell 
    && \text{$\forall\;t = 1, \dots, N$},\notag\\*
    &&&&&\text{$i \in I$},\label{eq:BC-advanced:DistanceLeft-UB}\\
    &&\vardistright{t}{i} &\le r 
    && \text{$\forall\;t = 1, \dots, N$},\notag\\*
    &&&&&\text{$i \in I$},\label{eq:BC-advanced:DistanceRight-UB}\\
    &\lefteqn{\text{vif 
        $\varconviction{j_{\min}}{j_{\max}} = 1$ then}}\notag\\
    && \vardistleft{N}{i} &\le 0\label{eq:BC-advanced:DistanceLeft-Zero}\\
    &\lefteqn{\text{end}}
    &&
    && \text{$i \in I\colon i \ge j_{\min}$},\notag\\*
    &&&&&\text{$j_{\min}, j_{\max} \in I\colon j_{\min} \le j_{\max}$},\notag\\
    &\lefteqn{\text{vif 
        $\varconviction{j_{\min}}{j_{\max}} = 1$ then}}\notag\\
    && \vardistright{N}{i} &\le 0\label{eq:BC-advanced:DistanceRight-Zero}\\
    &\lefteqn{\text{end}}
    &&
    && \text{$i \in I\colon i \le j_{\max}$},\notag\\*
    &&&&&\text{$j_{\min}, j_{\max} \in I\colon j_{\min} \le j_{\max}$},\notag\\
    && \varvoterconf{t}{i}{j} &= 
    \sum_{
      \substack{
        j_{\max} \ge j\\
        j_{\min} \le j_{\max}\\
        c_l, c_r \in \{0, 1\}\\
      }
    }
    \varconfidence{t}{i}{j_{\min}}{j_{\max}}{c_l}{c_r}
    && \text{$\forall\;t = 1, \dots, N$},\notag\\*
    &&&&&\text{$i, j \in I\colon i < j$},\label{eq:BC-advanced:SymmetryRelation1}\\
    && \varvoterconf{t}{i}{j} &= 
    \sum_{
      \substack{
        j_{\min} \le i\\
        j_{\max} \ge j_{\min}\\
        c_l, c_r \in \{0, 1\}\\
      }
    } 
    \varconfidence{t}{j}{j_{\min}}{j_{\max}}{c_l}{c_r}
    && \text{$\forall\;t = 1, \dots, N$},\notag\\*
    &&&&&\text{$i, j \in I\colon i < j$},\label{eq:BC-advanced:SymmetryRelation2}\\
    && \varcontrolconf{t}{i} &=
    \sum_{
      \substack{
        j_{\min}, j_{\max} \in I\colon\\
        j_{\min} \le j_{\max}
      } 
    } \varconfidence{t}{j}{j_{\min}}{j_{\max}}{1}{1}
    && \text{$\forall\;t = 1, \dots, N$},\notag\\*
    &&&&&\text{$i \in I$},\label{eq:BC-advanced:ControlRelation}\\
    && x^0_0 &\le \frac{1}{2}\label{eq:BC-advanced:Control-BreakSymmetry}\\
    %% 
    %&& x^t_i - x^{t-1}_i &\le \frac{n - i}{n - i + 1} \epsilon
    && x^t_i - x^{t-1}_i &\le \frac{n - i+1}{n - i + 2} \epsilon
    &&  \text{$\forall\;t = 1, \dots, N$},\notag\\*
    &&&&&\text{$i \in I$},\label{eq:BC-advanced:MaxReachRight}\\
    %% 
    %&& x^t_i - x^{t-1}_i &\ge -\frac{i + 1}{i + 2} \epsilon
    && x^t_i - x^{t-1}_i &\ge -\frac{i}{i + 1} \epsilon
    &&  \text{$\forall\;t = 1, \dots, N$},\notag\\*
    &&&&&\text{$i \in I$},\label{eq:BC-advanced:MaxReachLeft}\\
    && x^t_{i} &\le x^t_{j}
    && \text{$\forall\;t = 1, \dots, N$},\notag\\*
    &&&&&\text{$i, j \in I\colon i < j$},\label{eq:BC-advanced:Monotonicity}\\
    &&x^t_i &\in [0,1] && \text{$\forall\;t = 1, \dots, N$},\notag\\*
    &&&&&\text{$i \in I$}\label{eq:BC-advanced:Continuous}\\
    &&\varconfidence{t}{i}{j_{\min}}{j_{\max}}{c_l}{c_r}
    &\in \{0,1\} && \text{$\forall\;t = 1, \dots, N$},\notag\\*
    &&&&&\text{$j_{\min}, i, j_{\max} \in I\colon j_{\min} \le i \le j_{\max}$},\notag\\*
    &&&&&\text{$c_l, c_r \in \{0, 1\}$}\label{eq:BC-advanced:Binary1},\\
    && \varconviction{j_{\min}}{j_{\max}}
    &\in \{0,1\} && \text{$j_{\min}, j_{\max} \in I\colon j_{\min} \le i \le j_{\max}$}.\label{eq:BC-advanced:Binary2}
  \end{align}
\end{footnotesize}
The main term of the objective \eqref{eq:BC-advanced:Obj} determines
the number of convinced voters by the help of the variable
$\varconviction{j_{\min}}{j_{\max}}$, which is one if and only if
$j_{\min}$ is the minimal index and ${j_{\max}}$ is the maximal index
of a convinced voter.  The
perturbation~\eqref{eq:BC-advanced:ObjPerturbation} adds one and
subtracts a penalty term less than one from this number.  The penalty
is essentially the normalized average distance of the non-convinced
voters to the conviction interval.  The motivation of this
perturbation is, that the standard solver, when branching on variables
with increasing stage index, has a chance to identify those partial
solutions up to a stage that have greater chances (heuristically) to
increase the number of convinced voters in future stages.  This
influences which branches are inspected first and can lead to faster
identification of good primal solutions.
Restriction~\eqref{eq:BC-advanced:Start} fixes the start values, as in
the basic model.
Constraint~\eqref{eq:BC-advanced:ConfidenceAssignment} demands that
exactly one confidence configuration is selected for each voter in
each stage.  Constraints~\eqref{eq:BC-advanced:Voter-BoundLeft}
through~\eqref{eq:BC-advanced:Control-ExceedRight} makes sure that the
selection of confidence configurations is consistent with the opinions
and their distances (in an analogous way to the basic model).
Constraint~\eqref{eq:BC-advanced:ConvictionAssignment} models the fact
that there can be at most one conviction configuration at the end.  If
none of the possible conviction configurations is selected then no
voter is convinced in the end.
Restrictions~\eqref{eq:BC-advanced:Conviction-BoundLeft} and
\eqref{eq:BC-advanced:Conviction-BoundRight} make sure that the
selected conviction configuration is consistent with the distances of
voters to the conviction interval.  The dynamics is represented by
restriction~\eqref{eq:BC-advanced:Dynamics}.  Note how much simpler
the computation of the dynamics becomes with the help of the
confidence configuration variables compared to the basic model.  So
far, the logic of bounded confidence control is complete.  The
remaining restrictions are heuristic add-ons in order to accelerate
the solutions process in a standard solver by means of the additional
variables.  Contraints~\eqref{eq:BC-advanced:DistanceLeft-LB1} through
\eqref{eq:BC-advanced:DistanceRight-Zero} impose bounds on the
distances of voters to the conviction interval.  If we put them all
together, the distance variables are urged to exactly those distances.
Constraint~\eqref{eq:BC-advanced:SymmetryRelation1} and
\eqref{eq:BC-advanced:SymmetryRelation2} make sure that the additional
variables $\varvoterconf{t}{i}{j}$ receive values that are consistent
with the selected confidence configurations: voters $i$ and~$j$
influence each other if and only if one of the confidence
configuration variables
$\varconfidence{t}{i}{j_{\min}}{j_{\max}}{c_l}{c_r}$ and
$\varconfidence{t}{j}{j_{\min}}{j_{\max}}{c_l}{c_r}$, respectively,
for configurations in which $i$ and~$j$ influence each other is one.
The sum is taken over all such configurations, thus it does not matter
which confidence configuration variable contributes the one.  Totally
analogous is the effect of
constraint~\eqref{eq:BC-advanced:ControlRelation} for the additional
variable~$\varcontrolconf{t}{i}$: it is set to one whenever one of the
confidence configuration variables of the
form~$\varconfidence{t}{j}{j_{\min}}{j_{\max}}{1}{1}$ is one.  Some
additional cutting planes are provided by
restriction~\eqref{eq:BC-advanced:Control-BreakSymmetry}, which
chooses the first control value in the left half of the opinion
space. This is possible because the benchmark problem is symmetric.
Restrictions \eqref{eq:BC-advanced:MaxReachRight}
and~\eqref{eq:BC-advanced:MaxReachLeft} pose bounds on how far an
opinion can move in just one stage.  Finally,
constraint~\eqref{eq:BC-advanced:Monotonicity} explicitly demands that
the order of opinions is consistent with the indices.  The remaining
constraints \eqref{eq:BC-advanced:Continuous}
through~\eqref{eq:BC-advanced:Binary2} specify the types of the
variables.

If one spells out all variable-conditioned constraints in linear
restrictions, then one obtains the problem class \texttt{rocII}
contained in the \texttt{MIPLIB 2010}
\cite{Koch+Achterberg:MIPLIB2010:2011} benchmark suite.  The instance
\texttt{rocII-4-11} (11 voters, 4 stages) is classified as ``easy''
whereas already \texttt{rocII-7-11} is classified as ``challenge''
(open problem).  The full benchmark problem \texttt{rocII-10-11}
(status ``challenge'') is also contained in the suite.  The
\texttt{MIPLIB 2010} suite constitutes the probably most important
test bed used by virtually all developers of standard solvers for
tuning their software products, and it may very well be that general
MILP research that is totally unrelated to opinion dynamics will lead
to the solution of some of our benchmark instances.

\section{The parameter settings for the MILP solver}
\label{sec:param-sett-milp}

\begin{table}[h]
  \centering\footnotesize\ttfamily
  \begin{tabular}{rl}
    \toprule
    simplex tolerance feasibility & 1e-09\\
    simplex tolerance optimality  & 1e-3\\
    mip strategy variableselection& 3 \textrm{(=strong branching)}\\
    mip tolerance absmipgap       & 1e-3\\
    emphasis numerical            & yes\\
    timelimit                     & 3600 \textrm{(in the respective cases)}\\
  \bottomrule    
  \end{tabular}
  \caption{The \texttt{cplex} parameter settings that were used for all computations.} 
  \label{tab:MILP-cplex-parameters}
\end{table}
Table~\ref{tab:MILP-cplex-parameters} shows the cplex parameter
setting that we used for our computations.  This is meant for possible
replication of our results.  There is no reason to believe that these
parameter values are the best possible.  They have been set based on
our general computational experience in MILP.

\end{document}